\documentclass[11pt]{article}
\usepackage[margin=2.4cm,centering]{geometry} %
\usepackage{amsmath, amsthm, amssymb}
\usepackage{graphicx}
\usepackage[small,it]{caption}
\usepackage{enumitem}
\usepackage{mdframed}
\usepackage{lipsum}

\usepackage[T1]{fontenc} %
\usepackage[sc]{mathpazo} %
\usepackage{parskip} %
\usepackage{standalone} %imported for pictures
\usepackage{tikz}
\usepackage{pgfplots}
\pgfplotsset{compat=1.10}
\usepgfplotslibrary{fillbetween}
\usetikzlibrary{decorations.pathreplacing}
\usepackage{subcaption}
\usepackage{float}

\tikzstyle{b_vertex}=[circle,fill=black!100,text=white,inner sep=0.8mm,draw]
\tikzstyle{w_vertex}=[circle,fill=white!100,text=white,inner sep=0.4mm,draw]
\tikzstyle{point}=[circle,fill=white!1,text=white,inner sep=0.4mm,draw]
\tikzstyle{point2}=[circle,fill=black,inner sep=0.3mm]
\tikzstyle{point}=[circle,fill=black,inner sep=0.1mm] %why is this here?
\tikzstyle{path_edge}=[thick]

\def\up{\mathcal{\varUpsilon}}

\title{Classes of graphs without star forests and related graphs}
\author{Aistis Atminas\thanks{Department of Mathematics, London School of Economics, London, WC2A 2AE, United Kingdom, A.Atminas@lse.ac.uk} }
\date{November 8, 2017}

\newcommand{\X}{\mathcal{X}}
\newcommand{\Y}{\mathcal{Y}}
\newcommand{\F}{\mathcal{F}}

\begingroup
    \makeatletter
    \@for\theoremstyle:=definition,remark,plain\do{%
        \expandafter\g@addto@macro\csname th@\theoremstyle\endcsname{%
            \addtolength\thm@preskip\parskip
            }%
        }
\endgroup

\theoremstyle{plain}
\newtheorem{theorem}{Theorem}[section]

\newtheorem{lemma}[theorem]{Lemma}
\newtheorem{definition}[theorem]{Definition}

\newtheorem{corollary}[theorem]{Corollary}

\newtheorem{observation}[theorem]{Observation}

\usepackage{hyperref}

\begin{document}
\maketitle
\begin{abstract}
This work provides a structural characterisation of hereditary graph classes that do not contain a star forest, several graphs obtained from star forests by subset complementation, a union of cliques, and the complement of a union of cliques as induced subgraphs. This provides, for instance, structural results for graph classes not containing a matching and several complements of a matching. In terms of the speed of hereditary graph classes, our results imply that all such classes have at most factorial speed of growth.
\end{abstract}

\section{Introduction}

A \emph{graph class} is a set of graphs closed under isomorphism. A graph class is \emph{hereditary} if it is closed under taking induced subgraphs.  It is well-known (and can be easily seen) that a graph property $\X$ is hereditary if and only if $\X$ can be
described in terms of forbidden induced subgraphs. More formally, for a set $\F$ of graphs
we write $Free(\F)$ for the class of graphs containing no induced subgraph isomorphic to any graph in the set $\F$. 
A graph class $\X$ is hereditary if and only if $\X = Free(\F)$ for some set $\F$.
We call $\F$ a set of \emph{forbidden induced subgraphs} for $\X$ and say that graphs in $X$ are $\F$-\emph{free}.

One of the systematic ways of exploring structural properties of graph classes is by looking at the asymptotic growth of the number of graphs it contains. More formally, given a class $\X$, we write $\X_n$ for the number of graphs in $\X$ with vertex set $\{1, 2, \ldots  , n\}$ and call this sequence \emph{the speed of hereditary class} $\X$. The possible structures and speeds of a hereditary or monotone property of graphs have been extensively studied, originally in the special case where a single subgraph is forbidden \cite{one1, one2, one3, one4, one5, one6}, and more recently in general \cite{general1, general2, general3, general4, bbw05, general5}.  For example, Erd\H{o}s, Kleitman and Rothschild \cite{one2} and Kolaitis, Pr\"omel and Rothschild \cite{one3} studied $K_r$-free graphs, Erd\H{o}s, Frankl
and R\"odl \cite{one1} studied monotone properties when a single graph is forbidden, and Pr\"omel
and Steger \cite{one4, one5} obtained very precise results on the structure of almost all (induced-)$C_4$-free and $C_5$-free graphs. In a more general setting, precise structural results were obtained for the classes with lower speeds of growth: constant, polynomial, exponential and factorial. One structural result of our interest is given by Balogh, Bollob\'as and Weinreich  \cite{bbw05}. The result provides us with induced forbidden characterisation of the classes $\X$ for which there exist numbers $S=S(\X)$ and $d=d(\X)$ such that the vertices of any graph can be partitioned into $S$ parts which are cliques or independent sets and between the parts the graph has either degree bounded by $d$ or co-degree bounded by $d$. All the classes with constant, polynomial, exponential, factorial speed below Bell number and some classes of factorial speed above Bell number, namely those that have finite distinguishing number, admit this partition. In this work, we propose a similar partition result, that provides with structural description of a more general family of classes from the factorial layer. %We provide a structural description of the classes of graphs that forbid star forests and related graphs (graphs shown in Figure~\ref{fig:stars}) together with union of cliques and complement of union of cliques.   

To state our results, we will first introduce some notation. We denote by $K_k$ a clique on $k$ vertices and by $K_{1,k}$ a star with $k$ leaves. Let $G^1_{n,k}=nK_{1,k}$ be a disjoint union of $n$ stars and let $G^2_{n,k}, G^3_{n,k}, G^4_{n,k}$ be the graphs obtained from $nK_{1,k}$ by adding a clique on centres, leaves or on each of centres and leaves of the stars, respectively. Now let $H^i_{n,k}=\overline{G^i_{n,k}}$ be the complement of $G^i_{n,k}$ for each $i=1,2,3,4$. We will denote the family of these graphs as $\F_{n,k}=\{G_{n,k}^1, G_{n,k}^2, G_{n,k}^3, G_{n,k}^4, H_{n,k}^1, H_{n,k}^2, H_{n,k}^3, H_{n,k}^4\}$ (see Figure~\ref{fig:stars}). We will denote a bipartite graph by $G=(A, B, \mathcal{E})$, where the vertex set is $A \cup B$, with $A$ called top part and $B$ bottom part of the bipartition, and the edge set is $\mathcal{E} \subseteq A \times B$. When talking about induced subgraph containment between bipartite graphs we will require the embedding to respect the bipartition, i.e. the top and bottom parts are required to embed into top and bottom parts, respectively. We will use the following notation for the two bipartitions of $K_{1,k}$: $\Lambda_k=(\{a\}, B, \{a\} \times B)$ for $|B|=k$ and $\up_k=(A, \{b\}, A \times \{b\})$ for $|A|=k$. In this work our main result is as follows:

\begin{theorem}~\label{main}
Let $G \in Free(\F_{n,k} \cup \{nK_l, \overline{nK_l}\})$. Then there is a constant $T(n,k,l)$ such that the vertex set $V(G)$ can be partitioned into $t \leq T$ parts $V(G)=V_1 \cup V_2 \cup \ldots \cup V_t$ such that:
\begin{itemize}
\item $G[V_i]$ induces a clique or an independent set for each $1 \leq i \leq t$,
\item The bipartite graph induced between the parts $V_i$ and $V_j$, is 
$(2 \Lambda_{2k-1}, 2 \up_{2k-1})$-free for each pair $1 \leq i,j \leq t$, $i \neq j$, . 
\end{itemize}
\end{theorem} 

We note that the fact that this class has bounded cochromatic number, i.e. can be partitioned into bounded number of cliques and independent sets, follows readily from the works~\cite{cs14, kp94}. Therefore, the focus of this project will be to describe what happens between the parts. The condition for bipartite graph being $(2 \Lambda_{2k-1}, 2 \up_{2k-1})$-free as described in the thoerem above is equivalent to saying that for any two vertices from one part of the bipartition, the neighbourhood of one vertex contains all but at most $2k-2$ neighbours of the other.

\begin{figure}[H]
	%\centering
	\begin{tikzpicture}[scale=.6,auto=left]
	% S_{3,3,3}
	\foreach \start in {0, 7, 14, 21}
	{
	\fill[black!40!white] (2.25+\start, -0.1) ellipse (3.2cm and 0.3cm);
		
	%%%%%%%%longest%%%%%%%%%
	\foreach \start in {0,7,14,21}
{
}

		%\foreach \last in {-0.4+1.5, 0+1.5, 0.4+1.5}
		%{
		%\foreach \first in {0}
		%{ 
		%\draw (\start+\first, 0) arc (270-90/(3+\last):270+90/(3+\last):3+\last);
		%\draw (\start+\first,0) arc (240:300:4.5+\last-\first);
		%\draw (\start+\first,0) arc (240:300:3+\last-\first);
		%\draw (\start+\first,0) arc (255:285:3+2*\last-2*\first);
		%\draw (\start+\first+1.5,0) arc (240:300:3+\last-\first);
		%\draw (\start+\first+1.5,0) arc (255:285:3+2*\last-2*\first);
		%\draw (\start+\first+3,0) arc (255:285:3+2*\last-2*\first);	
		%}
		%}

		%\draw (\start+1.5,0) arc (240:300:3);	    			
		%\draw (\start+1.5,0) arc (255:285:3);
		%\draw (\start+3,0) arc (255:285:3);
	}

	\foreach \start in {14, 21}
	{
	\foreach \height in {0,4}
	{
	\fill[black!30!white] (2.25+\start, 2.6+\height) ellipse (3cm and 0.3cm);
	\iffalse
	\draw (\start+1.5,2.5+\height) arc (120:60:3);	    	
	\draw (\start+0,2.5+\height) arc (120:60:3);
	\draw (\start+0,2.5+\height) arc (120:60:4.5);
	\draw (\start+0,2.5+\height) arc (105:75:3);
	\draw (\start+1.5,2.5+\height) arc (105:75:3);
	\draw (\start+3,2.5+\height) arc (105:75:3);
	\fi
	}
	}

	\foreach \height in {0, 4}
	{
		\foreach \start in {0, 14}
		{
		\node[w_vertex] (1) at (\start+0, 2.5+\height) { }; 	
		\node[w_vertex] (2) at (\start+1.5, 2.5+\height) { };
		\node[w_vertex] (3) at (\start+3, 2.5+\height) { };
		\node[w_vertex] (4) at (\start+4.5, 2.5+\height) { };
		\foreach \leg in {-0.4, 0 ,0.4}
			{		
			\node[w_vertex] (5) at (\start+0+\leg,0+\height) { };
			\node[w_vertex] (6) at (\start+1.5+\leg,0+\height) { };
			\node[w_vertex] (7) at (\start+3+\leg,0+\height) { };
			\node[w_vertex] (8) at (\start+4.5+\leg,0+\height) { };				

			\foreach \from/\to in {1/5,2/6,3/7,4/8}
	    		\draw (\from) -- (\to);
			}
		}
	
		\foreach \start in {7, 21}
		{
		\node[w_vertex] (1) at (\start+0, 2.5+\height) { }; 	
		\node[w_vertex] (2) at (\start+1.5, 2.5+\height) { };
		\node[w_vertex] (3) at (\start+3, 2.5+\height) { };
		\node[w_vertex] (4) at (\start+4.5, 2.5+\height) { };
		\foreach \leg in {-0.4, 0 ,0.4}
			{		
			\node[w_vertex] (5) at (\start+0+\leg,0+\height) { };
			\node[w_vertex] (6) at (\start+1.5+\leg,0+\height) { };
			\node[w_vertex] (7) at (\start+3+\leg,0+\height) { };
			\node[w_vertex] (8) at (\start+4.5+\leg,0+\height) { };				

			\foreach \from/\to in {1/6,1/7,1/8, 2/5, 2/7, 2/8, 3/5, 3/6, 3/8, 4/5, 4/6, 4/7}
	    		\draw (\from) -- (\to);
			}
		}
	}

	\end{tikzpicture}
	%\includestandalone[width=\textwidth]{stars}
	\caption{Family $\F_{4,3}$ (vertices in shaded regions form cliques)}
	\label{fig:stars}
\end{figure}
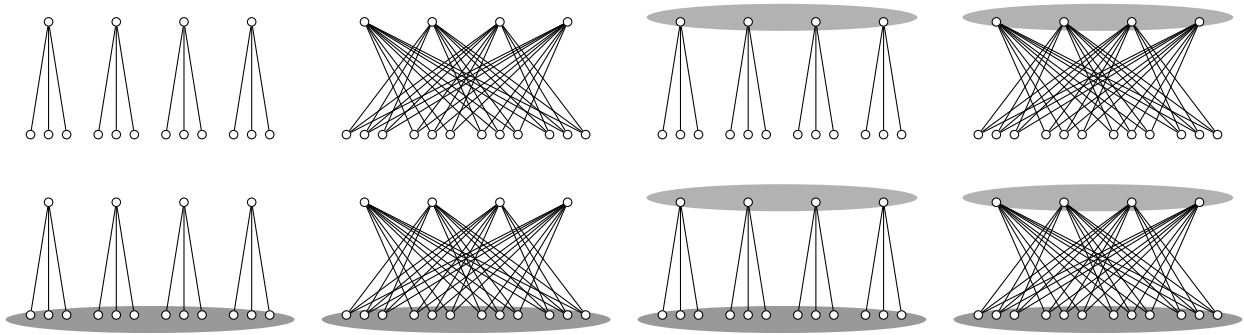

An interesting special case of the theorem above is when $k=1$, in which case the family $\F_{n,1}$ consists of a matching and related graphs (see Figure~\ref{fig:matching}). Noting that forbidding matchings also excludes unions of cliques we have the following special case of the main theorem: 

\begin{theorem}\label{thm:matchings}
Let $G \in Free(\F_{n,1})$, then there is a constant $T(n)$ such that the vertex set $V(G)$ can be partitioned into $t \leq T$ parts $V(G)=V_1 \cup V_2 \cup \ldots \cup V_t$ such that:
\begin{itemize}
\item $G[V_i]$ induces a clique or an independent set for each $1 \leq i \leq t$,
\item The bipartite graph induced between the parts $V_i$ and $V_j$ is $2K_2$-free.
\end{itemize}
\end{theorem}

\begin{figure}[H]
	%\centering
	\begin{tikzpicture}[scale=.6,auto=left]
	% S_{3,3,3}

		\foreach \start in {0, 9, 18}
		{		
		\node[w_vertex] (1) at (\start+0, 2) { }; 	
		\node[w_vertex] (2) at (\start+1, 2) { };
		\node[w_vertex] (3) at (\start+2, 2) { };
		\node[w_vertex] (4) at (\start+3, 2) { };
		\node[w_vertex] (5) at (\start+0,0) { };
		\node[w_vertex] (6) at (\start+1,0) { };
		\node[w_vertex] (7) at (\start+2,0) { };
		\node[w_vertex] (8) at (\start+3,0) { };
				
		\foreach \from/\to in {1/5,2/6,3/7,4/8}
	    	\draw (\from) -- (\to);
		%\foreach \from/\to in {5/6,6/7,7/8}
	    	%\draw (\from) -- (\to);
		
		}

	    	%\foreach \from/\to in {5/8,6/9,7/10}
	    	%\draw[thick] (\from) -- (\to);
	   % 	\coordinate [label=center:$G_4$] (S333) at (1.5,-1);
	   % 	\coordinate [label=center:$H_4$] (H1) at (6,-1);
	   %	\coordinate [label=center:$G_4^*$] (S333) at (10.5,-1);
	   % 	\coordinate [label=center:$H_4^*$] (H1) at (15,-1);
		% \coordinate [label=center:$\overline{H_4}$] (S333) at (19.5,-1);
	    %	\coordinate [label=center:$\overline{G_4}$] (H1) at (24,-1);

	% H_1
		\foreach \start in {18, 22.5}
		{
		\draw (\start+1,0) arc (240:300:2);	    	
		\draw (\start+0,0) arc (240:300:2);
		\draw (\start+0,0) arc (240:300:3);
		\draw (\start+0,0) arc (255:285:2);
		\draw (\start+1,0) arc (255:285:2);
		\draw (\start+2,0) arc (255:285:2);
		}

		\foreach \start in { 9,13.5, 18, 22.5}
		{
		\draw (\start+1,2) arc (120:60:2);	    	
		\draw (\start+0,2) arc (120:60:2);
		\draw (\start+0,2) arc (120:60:3);
		\draw (\start+0,2) arc (105:75:2);
		\draw (\start+1,2) arc (105:75:2);
		\draw (\start+2,2) arc (105:75:2);
		}    	
	    	
 		\foreach \start in {4.5, 13.5, 22.5}
		{		
		\node[w_vertex] (1) at (\start+0, 2) { }; 	
		\node[w_vertex] (2) at (\start+1, 2) { };
		\node[w_vertex] (3) at (\start+2, 2) { };
		\node[w_vertex] (4) at (\start+3, 2) { };
		\node[w_vertex] (5) at (\start+0,0) { };
		\node[w_vertex] (6) at (\start+1,0) { };
		\node[w_vertex] (7) at (\start+2,0) { };
		\node[w_vertex] (8) at (\start+3,0) { };
				
		\foreach \from/\to in {1/6,1/7,1/8,2/5,2/7,2/8,3/5,3/6,3/8,4/5,4/6,4/7}
	    	\draw (\from) -- (\to);
		%\foreach \from/\to in {5/6,6/7,7/8}
	    	%\draw (\from) -- (\to);
		}		

	\end{tikzpicture}
	%\includestandalone[width=\textwidth]{matchings}
	\caption{Family $\F_{4,1}$}
	\label{fig:matching}
\end{figure}

Theorem~\ref{thm:matchings} is interesting because it provides an analog to graphs of an important notion in permutations called "monotone griddability". The classes that are monotone griddable are the ones for which the permutations are partitionable into bounded number of cells with increasing or decreasing patterns in each cell. In \cite{HW06} Huczynska and Vatter proved that a class of permutations is monotone griddable if and only if it does not contain a large sum of 21's or skew-sum of 12's. This result provided an alternative proof of the jump of the speeds for permutations classes from 1 to the golden ratio (approximately 1.618) provided in \cite{permutation2} and this classification of permutation classes was later extended by Vatter \cite{permutation1} up to growth rate approximately 2.20557 (which could be thought of as a Bell number analog in graph theory as in both cases arbitrarily long "paths" appear for the first time). To see the correspondence between our result and the one in \cite{HW06} one needs to convert each entry in a permutation to a vertex and put an edge between two vertices if and only if the corresponding entries of the permutation form a decreasing pattern. In this case "a sum of 21's" corresponds to a matching, "a skew-sum of 12's" to a co-matching and the partition into increasing/decreasing cells corresponds to a partition into independent sets or cliques such that between every pair the induced bipartite graphs are $2K_2$-free. The relation of the results proved in this paper and the results of \cite{HW06} and \cite{bbw05} are presented in Figure~\ref{mot}.  

Our results can also be stated in terms of graph classes. Let $\X_1$ be a graph class consisting of star forests and let $\X_2, \X_3, \ldots, \X_8$ be the other 7 related classes obtained from star forests by various complementations. More formally, let $\X_i=\{G: G$ induced subgraph of $G_{n,k}^i$ for some $n, k \in \mathbb{N} \}$ for $i=1,2,3,4$ and let $\X_i=\{G: G$ induced subgraph of $H_{n,k}^{i-4}$ for some $n, k \in \mathbb{N} \}$ for $i=5,6,7,8$. Let $\X_9$ be the class of disjoint union of cliques and $\X_{10}$ the class of complements of disjoint union of cliques. Also define analogously $\Y_1$ to be the hereditary closure of the class of matchings and let $\Y_2, \Y_3, \Y_4, \Y_5, \Y_6$ be the hereditary closures of the different complements of matchings, one class for each type of complementation. Further, let us say that a graph is a \emph{$(t,k)$-graph} if it can be partitioned into $t$ cliques or independent sets with $(\Lambda_{2k-1}, \up_{2k-1})$-free graphs between the parts. Then our results are as follows:

\begin{theorem}\label{classes}
For a hereditary graph class $\X$ there exist two constants $T=T(\X)$ and $k=k(\X)$ such that each graph in the class is a $(T,k)$-graph if and only if $\X$ does not contain any of the classes $\X_1, \X_2, \ldots, \X_{10}$. For a hereditary graph class $\X$ there exist a constant $T=T(\X)$ such that each graph in the class is a $(T,1)$-graph if and only if $\X$ does not contain classes $\Y_1, \Y_2, \ldots, \Y_{6}$.

Moreover, all classes $\X$ consisting of $(T,k)$-graphs for some $T=T(\X)$ and $k=k(\X)$ have at most factorial speed of growth.  

\end{theorem}

The rest of the paper is structured as follows. For the benefit of the reader, in Sections~\ref{section:matching} and~\ref{section:matchingg} we provide a direct proof of Theorem~\ref{thm:matchings} because it encapsulates many of the ideas for the more general result, Theorem~\ref{main}, but in a simpler setting. In Sections~\ref{section:stars} and~\ref{section:starsg} we prove Theorem~\ref{main} and in Section~\ref{section:factorial} we prove Theorem~\ref{classes}. Section~\ref{section:conclusion} contains concluding remarks and open questions.

\begin{figure}[H] 
\scalebox{1.0}
{
\minipage{0.35\textwidth}
\begin{center}
\begin{tikzpicture}[scale=.6,auto=left]

\fill[blue!15!white] (-0.8,0) -- (-2.2,0) -- (-3.2,2.7) -- (-1.8,2.7) -- cycle;
\fill[blue!15!white] (0.8,0) -- (2.2,0) -- (3.2,2.7) -- (1.8,2.7) -- cycle;

\fill[blue!15!white] (0,4.2) -- (0,5.8) -- (2.5,3.5) -- (2.5,1.9) -- cycle;
\fill[blue!15!white] (0,4.2) -- (0,5.8) -- (-2.5,3.5) -- (-2.5,1.9) -- cycle;

\fill[blue!15!white] (-1.5,-0.7) -- (-1.5,0.7) -- (1.5,0.7) -- (1.5,-0.7) -- cycle;

\draw[fill=white] (-1.5,0) ellipse (0.8cm and 0.8cm);
\draw[fill=white] (1.5,0) ellipse (0.8cm and 0.8cm);
\draw[fill=white] (-2.5,2.7) ellipse (0.8cm and 0.8cm);
\draw[fill=white] (2.5,2.7) ellipse (0.8cm and 0.8cm);
\draw[fill=white] (0,5) ellipse (0.8cm and 0.8cm);

	\end{tikzpicture}
\end{center}
\endminipage
\minipage{0.67\textwidth}
\begin{center}
\begin{tikzpicture}[scale=.6,auto=left]
\draw (0,0) -- (0,4);
\draw (2,0) -- (2,4);
\draw (4,0) -- (4,4);
\draw (0,0) -- (4,0);
\draw (0,2) -- (4,2);
\draw (0,4) -- (4,4);

%%%%%%%%%%%%%%% lower left points %%%%%%%%%%%%%%%%%%%%%
\node[point2] at (0.3, 0.3) {};
\node[point2] at (0.6, 0.6) {};
\node[point2] at (0.9, 0.9) {};
\node[point2] at (1.2, 1.2) {};
\node[point2] at (1.5, 1.5) {};
\node[point2] at (1.8, 1.8) {};

%%%%%%%%%%%%%%% lower right points %%%%%%%%%%%%%%%%%%%%%
\node[point2] at (2+0.3, 0.3) {};
\node[point2] at (2+0.6, 0.6) {};
\node[point2] at (2+0.9, 0.9) {};
\node[point2] at (2+1.2, 1.2) {};
\node[point2] at (2+1.5, 1.5) {};
\node[point2] at (2+1.8, 1.8) {};

%%%%%%%%%%%%%%% upper left points %%%%%%%%%%%%%%%%%%%%%
\node[point2] at (0.3, 2+0.3) {};
\node[point2] at (0.6, 2+0.6) {};
\node[point2] at (0.9, 2+0.9) {};
\node[point2] at (1.2, 2+1.2) {};
\node[point2] at (1.5, 2+1.5) {};
\node[point2] at (1.8, 2+1.8) {};

%%%%%%%%%%%%%%% upper right points %%%%%%%%%%%%%%%%%%%%%
\node[point2] at (2+0.3, 2+1.8) {};
\node[point2] at (2+0.6, 2+1.5) {};
\node[point2] at (2+0.9, 2+1.2) {};
\node[point2] at (2+1.2, 2+0.9) {};
\node[point2] at (2+1.5, 2+0.6) {};
\node[point2] at (2+1.8, 2+0.3) {};

\draw[thick, <->] (5.8,2) -- (7.8,2);

\foreach \s in {-0.8}
{
\draw[fill=white] (\s+11,0) ellipse (0.8cm and 0.8cm);
\draw[fill=white] (\s+11,4) ellipse (0.8cm and 0.8cm);
\draw[fill=white] (\s+14,0) ellipse (0.8cm and 0.8cm);
\draw[fill=blue!15!white] (\s+14,4) ellipse (0.8cm and 0.8cm);

%%%%%%%%%%%%%%%%  lower two %%%%%
\draw (\s+11.4, 0.6) -- (\s+13.6, 0.6);
\draw (\s+11.4, 0.6) -- (\s+13.4, 0.2);
\draw (\s+11.4, 0.6) -- (\s+13.4, -0.2);
\draw (\s+11.4, 0.6) -- (\s+13.6, -0.6);

\draw (\s+11.6, 0.2) -- (\s+13.4, 0.2);
\draw (\s+11.6, 0.2) -- (\s+13.4, -0.2);
\draw (\s+11.6, 0.2) -- (\s+13.6, -0.6);

\draw (\s+11.6,-0.2) -- (\s+13.4, -0.2);
\draw (\s+11.6,-0.2) -- (\s+13.6, -0.6);

\draw (\s+11.4,-0.6) -- (\s+13.6, -0.6);

%%%%%%%%%%%%%%%%% upper two %%%%%%%%%%%%%%%%%

\draw (\s+11.4, 4.6) -- (\s+13.6, 4.6);
\draw (\s+11.4, 4.6) -- (\s+13.4, 4.2);
\draw (\s+11.4, 4.6) -- (\s+13.4, 4-0.2);
\draw (\s+11.4, 4.6) -- (\s+13.6, 4-0.6);

\draw (\s+11.6, 4.2) -- (\s+13.4, 4.2);
\draw (\s+11.6, 4.2) -- (\s+13.4, 4-0.2);
\draw (\s+11.6, 4.2) -- (\s+13.6, 4-0.6);

\draw (\s+11.6, 4-0.2) -- (\s+13.4, 4-0.2);
\draw (\s+11.6, 4-0.2) -- (\s+13.6, 4-0.6);

\draw (\s+11.4, 4-0.6) -- (\s+13.6, 4-0.6);

%%%%%%%%%%%%%%%%% left side  %%%%%%%%%%%%%
\draw (\s+10.4, 4-0.6) -- (\s+10.4, 0.6);
\draw (\s+10.4, 4-0.6) -- (\s+10.8, 0.7);
\draw (\s+10.4, 4-0.6) -- (\s+11.2, 0.7);
\draw (\s+10.4, 4-0.6) -- (\s+11.6, 0.6);

\draw (\s+10.8, 4-0.7) -- (\s+10.8, 0.7);
\draw (\s+10.8, 4-0.7) -- (\s+11.2, 0.7);
\draw (\s+10.8, 4-0.7) -- (\s+11.6, 0.6);

\draw (\s+11.2, 4-0.7) -- (\s+11.2, 0.7);
\draw (\s+11.2, 4-0.7) -- (\s+11.6, 0.6);

\draw (\s+11.6, 4-0.6) -- (\s+11.6, 0.6);

%%%%%%%%%%%%%%%%% right side %%%%%%%%%%

\draw (\s+13.4, 4-0.6) -- (\s+13.4, 0.6);
\draw (\s+13.4, 4-0.6) -- (\s+13.8, 0.7);
\draw (\s+13.4, 4-0.6) -- (\s+14.2, 0.7);
\draw (\s+13.4, 4-0.6) -- (\s+14.6, 0.6);

\draw (\s+13.8, 4-0.7) -- (\s+13.8, 0.7);
\draw (\s+13.8, 4-0.7) -- (\s+14.2, 0.7);
\draw (\s+13.8, 4-0.7) -- (\s+14.6, 0.6);

\draw (\s+14.2, 4-0.7) -- (\s+14.2, 0.7);
\draw (\s+14.2, 4-0.7) -- (\s+14.6, 0.6);

\draw (\s+14.6, 4-0.6) -- (\s+14.6, 0.6);

}

	\end{tikzpicture}
	
\end{center}
\endminipage
}

\scalebox{0.4}
{
\minipage{\textwidth}
.
\endminipage
}

\scalebox{.85}
{
\minipage{0.4\textwidth}
\begin{center}
cliques/independent sets, bounded \\
degree/codegree between bags
\end{center}
\endminipage
\minipage{0.4\textwidth}
\begin{center}
cells with increasing/decreasing\\
subpermutations
\end{center}
\endminipage
\minipage{0.4\textwidth}
\begin{center}
cliques/independent sets\\
with $2K_2$-free graphs between them
\end{center}
\endminipage
}

\scalebox{1.0}
{
\minipage{1\textwidth}
\begin{center}
\begin{tikzpicture}[scale=.6,auto=left]

\fill[blue!15!white] (-0.8,0) -- (-2.2,0) -- (-3.2,2.7) -- (-1.8,2.7) -- cycle;
\fill[blue!15!white] (0.8,0) -- (2.2,0) -- (3.2,2.7) -- (1.8,2.7) -- cycle;

\fill[blue!15!white] (0,4.2) -- (0,5.8) -- (2.5,3.5) -- (2.5,1.9) -- cycle;
\fill[blue!15!white] (0,4.2) -- (0,5.8) -- (-2.5,3.5) -- (-2.5,1.9) -- cycle;

\fill[blue!15!white] (-1.5,-0.7) -- (-1.5,0.7) -- (1.5,0.7) -- (1.5,-0.7) -- cycle;

\draw (-1.2, 0.7) -- (1.2,0.7);
\draw (-1.2, -0.7) -- (1.2,-0.7);
\draw (-1.2, 0.7) -- (1.2,-0.7);

\draw (-3.2, 2.7) -- (-2.2,0);
\draw (-1.8, 2.7) -- (-0.8,0);
\draw (-3.2, 2.5) -- (-0.8,0.2);
 
\draw (3.2, 2.7) -- (2.2,0);
\draw (1.8, 2.7) -- (0.8,0);
\draw (3.2, 2.5) -- (0.8,0.2);

\draw (0, 5.8) -- (2.5,3.5);
\draw (0, 4.2) -- (2.5,1.9);
\draw (0, 4.2) -- (2.5,3.5);

\draw (0, 5.8) -- (-2.5,3.5);
\draw (0, 4.2) -- (-2.5,1.9);
\draw (0, 4.2) -- (-2.5,3.5);

\draw[fill=white] (-1.5,0) ellipse (0.8cm and 0.8cm);
\draw[fill=white] (1.5,0) ellipse (0.8cm and 0.8cm);
\draw[fill=white] (-2.5,2.7) ellipse (0.8cm and 0.8cm);
\draw[fill=white] (2.5,2.7) ellipse (0.8cm and 0.8cm);
\draw[fill=white] (0,5) ellipse (0.8cm and 0.8cm);

\draw[thick, ->] (-8, 6) -- (-4, 4);
\draw[thick, ->] (8, 6) -- (4, 4);

	\end{tikzpicture}
	
\end{center}
\endminipage
}

\scalebox{0.4}
{
\minipage{\textwidth}
.
\endminipage
}

\scalebox{0.85}
{
\minipage{1.17\textwidth}
\begin{center}
cliques/independent sets  \\
with $(2\Lambda_{2k-1}, 2 \up_{2k-1})$-free  \\
graphs between them  
\end{center}
\endminipage
}

%\minipage{1.4\textwidth}
\iffalse
\begin{table}
\begin{tabular}{*3c}
 				%\includegraphics{motivation2}
 				%& \multicolumn{2}{c}{ \includegraphics{motivation}} \\
\input{motivation2} & \multicolumn{2}{c}{\input{motivation}} \\
cliques/independent sets, bounded  & cells with increasing/decreasing 
   & cliques/independent sets \\ 
 degree/codegree between bags & subpermutations & with $2K_2$-free graphs between them \\
	\multicolumn{3}{c}{\input{motivation3}} 
   			%\multicolumn{3}{c}{\includegraphics{motivation3}} \\
\\
& cliques/independent sets & \\
& with $(2\Lambda_{2k-1}, 2 \up_{2k-1})$-free  & \\
& graphs between them & \\
\end{tabular}
\end{table}
\fi
%\endminipage

\caption{Partitions from \cite{bbw05} (upper left), \cite{HW06} (upper middle) and the results of this paper}
\label{mot}
\end{figure}
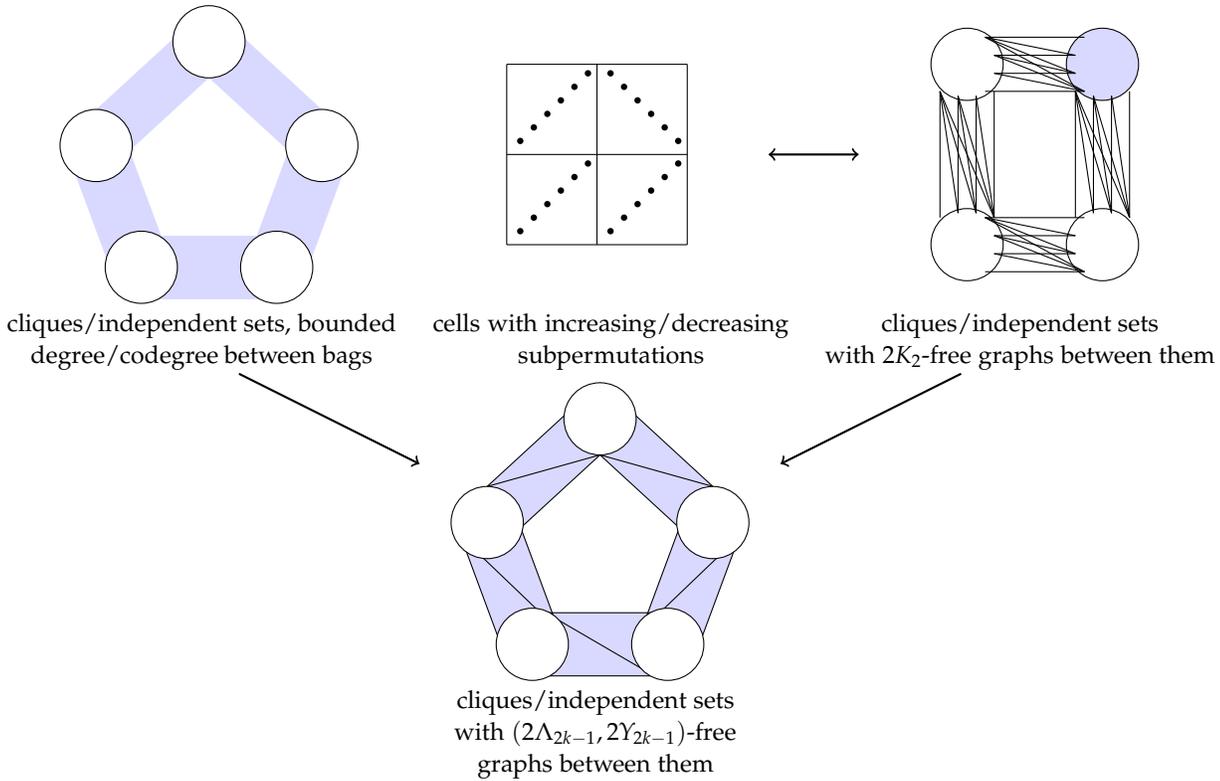

%%%%%%%%%%%%%%%%%%%%%% END OF INTRODUCTION %%%%%%%%%%%%%%%

\section{Bipartite graphs without a matching and a bipartite co-matching} \label{section:matching}

For a bipartite graph $G=(A, B, \mathcal{E})$ we denote its \emph{bipartite complement} by $\overline{G}^{bip}=(A, B, A \times B \backslash \mathcal{E})$. We denote the bipartite subgraph of $G=(A,B, \mathcal{E})$ induced by the vertex subsets $A' \subset A$ and $B' \subset B$ by $G[A',B']=(A', B', \mathcal{E} \cap A' \times B')$. In this section we prove the following theorem, which is the bipartite case of Theorem~\ref{thm:matchings} and which will be extended to the general case in Section~\ref{section:matchingg}. 

\begin{theorem} \label{bipartitepartition}
For all $n, m \in \mathbb{N}$ there is a fixed integer $f(n,m) \in \mathbb{N}$ such that the following holds. Let $G=(A \cup B, \mathcal{E} \subseteq A \times B)$ be a bipartite graph which does not contain a matching $nK_2$ and a bipartite co-matching $\overline{mK_2}^{bip}$. Then there is a partition $A=A_1 \cup A_2 \cup \ldots \cup A_u$, 
$B=B_1 \cup B_2 \cup \ldots \cup B_u$ with $u \leq f(n,m)$ such that $G[A_i, B_j]$ is $2K_2$-free for all $1 \leq i, j \leq u$. 
\end{theorem}
 
Throughout the section we will be using the following notation.

\begin{definition}\label{basicdef}
For a graph $G$ and two disjoint vertex subsets $V_1, V_2 \subset V(G)$, we will say $V_1$ is joined (resp. co-joined) to $V_2$ if $G$ contains all (resp. none) of the edges between $V_1$ and $V_2$. By $N(V_1)$ we will denote the neighbourhood of the subset $V_1$, i.e $N(V_1)=\cup_{v \in V_1} N(v)$. The non-neighbourhood of subset $V_1$ will be denote by $\overline{N}(V_1)= \{v \in V(G) \backslash V_1: v$ has a non-neighbour in $V_1\}$. Finally, we will say that a subset $V_1$ covers (resp. co-covers) subset $V_2$ if $N(V_1) \supseteq V_2$ (resp. $\overline{N}(V_1) \supseteq V_2$). 
\end{definition}
 
The proof of Theorem~\ref{bipartitepartition} is by induction on $n+m$ and for induction step we will show how to construct the following partition, which we call \emph{$(n,m,q)$-chain template}.   

\begin{definition} \label{templatedef}
Let $G=(A, B, \mathcal{E})$ be a bipartite graph. We will call the partition $A=A_1 \cup A_2 \cup \ldots \cup A_z$,
$B=B_1 \cup B_2 \cup \ldots \cup B_z$ a chain template, if it satisfies the following two conditions:
\begin{itemize}[noitemsep,topsep=0pt,parsep=0pt,partopsep=0pt,labelwidth=2.5cm,align=left,itemindent=2.5cm]
\item[(*)] $A_i$ is joined to $B_j$ for all $j>{i+1}$.
\item[(**)   ] $A_i$ is co-joined to $B_j$ for all $j<i$. 
\end{itemize}

Suppose each part is subdivided into $q$ pieces $A_i=A_{i1} \cup A_{i2} \cup \ldots \cup A_{iq}$ and $B_i=B_{i1} \cup B_{i2} \cup \ldots \cup B_{iq}$ such that the following holds:
\begin{itemize}[noitemsep,topsep=0pt,parsep=0pt,partopsep=0pt,labelwidth=2.5cm,align=left,itemindent=2.5cm]
\item[(***)   ] $G[A_{ig}, B_{jh}]$ is in $Free((n-1)K_2, \overline{mK_2}^{bip})$ for $j=i+1$ and any $1 \leq g,h \leq q$.
\item[(****)   ] $G[A_{ig}, B_{jh}]$ is in $Free(nK_2, \overline{(m-1)K_2}^{bip})$ for $j=i$ and any $1 \leq g,h \leq q$.

Then we will call this refined partition an $(n, m, q)$-chain template or a refined chain template (with parameters $(n,m,q)$).
\end{itemize}
\end{definition}

A chain template is a partition that ensures that all non-trivial graphs appear only between consecutive parts of the partition (the shaded regions in Figure~\ref{template}), and between other pairs of parts the induced bipartite graphs are either complete or empty as specified by the conditions (*)--(**). As we will see later, a chain template can be obtained for any bipartite graph and might be a useful tool as it might give additional insights/restrictions on graphs that can lie between the the consecutive bags. In our particular case, when dealing with graphs excluding matchings and co-matchings, we will be able to deduce that (***)--(****) holds.

Before we proceed to constructing the $(n,m,q)$-chain template, we will first prove the lemma, which will show how $(n,m,q)$-chain template can be collapsed to a bounded number of parts, which will be required for our induction step. 

\begin{lemma} \label{partition2}
Let $G=(A,B,\mathcal{E})$ be a bipartite graph that admits an $(n,m,q)$ - chain template. Then the parts $A$ and $B$ can be partitioned into $2q$ parts each such that the graph between any two parts is either in $Free(nK_2, \overline{(m-1)K_2}^{bip})$ or in $Free((n-1)K_2, \overline{mK_2}^{bip})$. 
\end{lemma}

\begin{proof}

We use the notation denoting the parts of an $(n,m,q)$-chain template as given in the Definition~\ref{templatedef}. For each $1 \leq g \leq q$ define 
\[ A'_{1g}= \bigcup_{\substack{i \ odd,\\1 \leq i \leq z}} A_{ig}, \quad  
A'_{2g}= \bigcup_{\substack{i \ even,\\1 \leq i \leq z}} A_{ig}, \quad 
B'_{1g}= \bigcup_{\substack{i \ odd,\\1 \leq i \leq z}} B_{ig} \quad \textrm{and} \quad 
B'_{2g}= \bigcup_{\substack{i \ even,\\1 \leq i \leq z}} B_{ig}.\]

Figure~\ref{template} illustrates the bags that are picked to construct the union $A'_{11}$ and $B'_{13}$. We will show that $A=\bigcup_{\substack{1\leq i \leq 2,\\ 1 \leq g \leq q}} A'_{ig}$, $B=\bigcup_{\substack{1\leq i \leq 2,\\ 1 \leq g \leq q}}B'_{ig}$ is the required partition into $2q$ parts. In other words, we will show that between any two parts $A'_{ig}$ and $B'_{jh}$, with $i, j \in \{1, 2\}$ and $g, h \in \{1, 2, \ldots, q\}$, the induced graph $G[A'_{ig}, B'_{jh}]$ is either in $Free(nK_2, \overline{(m-1)K_2}^{bip})$ or in $Free((n-1)K_2, \overline{mK_2}^{bip})$.

Consider the graph $G[A'_{1g}, B'_{1h}]$. Take two vertices $v, w \in A'_{1g}$. If $v \in A_{(2i+1) g}$, $w \in A_{(2j+1) g}$ 
with $i<j$,  then it is easy to see that $N(v) \supseteq N(w)$. Hence the two vertices can have non-nested neighbourhoods only if they 
 belong to the same set $A_{(2i+1) g}$. Similarly, two vertices in $B'_{1h}$ can have incomparable neighbourhoods only if they belong to the same  
 subset $B_{(2j+1) h}$. Thus, any nontrivial matching or co-matching must be contained in some $G[A_{(2i+1) g}, B_{(2j+1)h}]$. 
If $i \neq j$, then $A_{(2i+1) g}$ joined or co-joined to $B_{(2j+1)h}$ and if $i=j$, then the graph $G[A_{(2i+1) g}, B_{(2i+1)h}] \in Free((n-1)K_2, \overline{mK_2})$. Hence, we conclude that $G[A'_{1g}, B'_{1h}]$ contains no matchings of size
 $n-1$ or co-matchings of size $m$, i.e. $G[A'_{1g}, B'_{1h}] \in Free((m-1)K_2, \overline{nK_2})$.    

The analogous arguments hold for the graphs $G[A'_{1g}, B'_{2h}], G[A'_{2g}, B'_{1h}], G[A'_{2g}, B'_{2h}]$.
\end{proof}

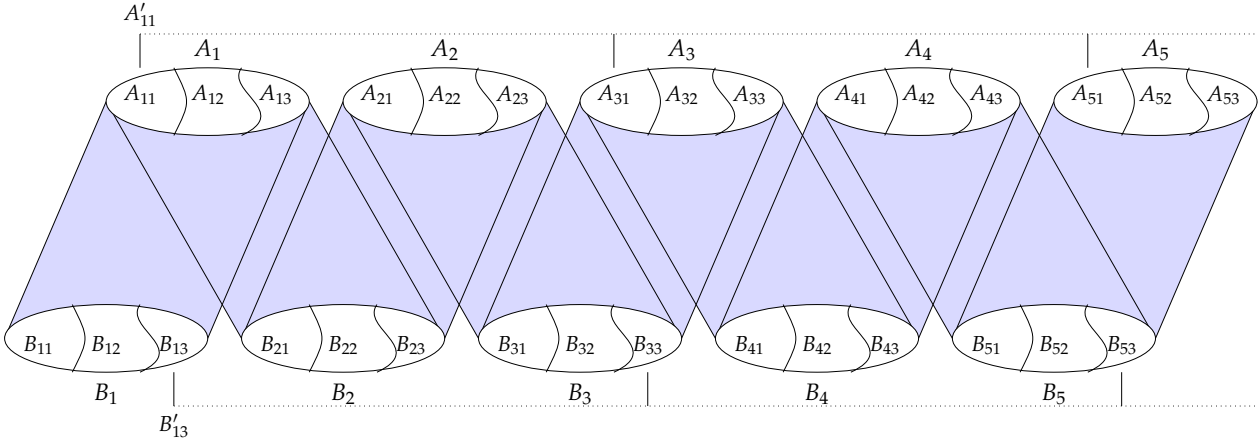
\begin{figure}[H]
	%\centering
	\scalebox{.75}
	{
	\begin{tikzpicture}[scale=.6,auto=left]

%%%%%%%%%%%%%%%%%%%%%%%%%%%%%%%%%%%%%%%%%%%%
		%\draw[step=1cm,gray,very thin] (0,0) grid (13,5);
		%\foreach \x in {0,1,2,3,4,5,6,7,8,9,10,11,12,13}
    		%\draw (\x cm,1pt) -- (\x cm,-1pt) node[anchor=north] {$\x$};
		%\foreach \y in {0,1,2,3,4}
    		%\draw (1pt,\y cm) -- (-1pt,\y cm) node[anchor=east] {$\y$};
		\foreach \start in {}
		{		
		\node[w_vertex] (1) at (\start+0, 2) {}; 	
		\node[w_vertex] (2) at (\start+1, 2) { };
		\node[w_vertex] (3) at (\start+2, 2) { };
		\node[w_vertex] (4) at (\start+3, 2) { };
		\node[w_vertex] (5) at (\start+0,0) { };
		\node[w_vertex] (6) at (\start+1,0) { };
		\node[w_vertex] (7) at (\start+2,0) { };
		\node[w_vertex] (8) at (\start+3,0) { };				
		\foreach \from/\to in {1/5,2/6,3/7,4/8}
	    	\draw (\from) -- (\to);
		%\foreach \from/\to in {5/6,6/7,7/8}
	    	%\draw (\from) -- (\to);		
		}
%%%%%%%%%%%%%%%%%%%%%%%%%%%%%%%%%%%%%%%%%%%%%%%

	\foreach{\h} in {7} 
	{
		\foreach \x/\y in {0/\h+2}
		{
		\draw [dotted] (\x+0,\y+0) -- (\x+33,\y+0); 
		\draw  (\x+0,\y+0) -- (\x+0, \y-1);
		\draw  (\x+14,\y+0) -- (\x+14, \y-1);
		\draw  (\x+28, \y+0) -- (\x+28, \y-1);
		}
		\draw (0, \h+2) node[anchor=south] {$A'_{11}$};
	
		\foreach \x/\y in {1/-2}
		{
		\draw [dotted] (\x+0,\y+0) -- (\x+32,\y+0); 
		\draw  (\x+0,\y+0) -- (\x+0, \y+1);
		\draw  (\x+14,\y+0) -- (\x+14, \y+1);
		\draw  (\x+28, \y+0) -- (\x+28, \y+1);
		}
		\draw (1, -2) node[anchor=north] {$B'_{13}$};

	%%%%%%%              Shades            %%%%%%
		
		\foreach{\start} in {0, 7, 14, 21, 28}
			{\fill[blue!15!white] (\start-4,0) -- (\start-1,\h) -- (\start+5,\h) -- (\start+2,0) -- cycle;}
		\foreach{\start} in {0, 7, 14, 21}
			{\fill[blue!15!white] (\start+3,0) -- (\start-1,\h) -- (\start+5,\h) -- (\start+9,0) -- cycle;}
		
		\foreach{\start} in {0, 7, 14, 21,28}
			{
			\draw (\start+-4, 0) -- (\start-1, \h);
			\draw (\start+2, 0) -- (\start+5, \h);
			}
		\foreach{\start} in {0, 7, 14, 21}
			{
			\draw (\start+-1, \h) -- (\start+3, 0);
			\draw (\start+5, \h) -- (\start+9, 0);		
			}	
		
	%	\draw [dotted] (10, 0) -- (12, 0);
	%	\draw [dotted] (13, \h) -- (15, \h);
	%	\draw [dotted] (11.5, \h*0.5) -- (13.5, \h*0.5);
	%	\draw [dotted] (34, 0) -- (36, 0);	
	%	\draw [dotted] (35, \h*0.5) -- (36, \h*0.5);

	%%%%%%%%%%%   Ellipses     %%%%%%%%%%%%
		\foreach \start in {0,7,14,21,28}
		{\draw[fill=white] (\start+2,\h) ellipse (3cm and 1cm);	}
		
		\foreach \start in {-7, 0,7,14,21}
		{\draw[fill=white] (\start+6,0) ellipse (3cm and 1cm);}

		\foreach \z in {-7, 0, 7, 14, 21}
			{
			\draw (\z+5,-1) .. controls (\z+5.5,0)  .. (\z+5,1);
			\draw (\z+7,-1) .. controls (\z+8.5,-0.3) and (\z+6.5, 0.3) .. (\z+7,1);
			}
		\foreach \z in {-7, 0, 7, 14,21}
			{
			\draw (\z+3+5,\h-1) .. controls (\z+3+5.5,\h)  .. (\z+3+5,1+\h);
			\draw (\z+3+7,\h-1) .. controls (\z+3+8.5,\h-0.3) and (\z+3+6.5, 0.3+\h) .. (\z+3+7,1+\h);
			}

	%%%%%%%%%%%%%labelling %%%%%%%%%%%%%

	%	\foreach \x/\y in {-3/-1.6, 4/-1.6,  11/-1.6, 18/-1.6, 25/-1.6}		
	%	{
	%	\draw [dotted] (\x+0,\y+0) -- (\x+4,\y+0); 
	%	\draw [dashed] (\x+0,\y+0) -- (\x+0, \y+0.15);
	%	\draw [dashed] (\x+2,\y+0) -- (\x+2, \y+0.15);
	%	\draw [dashed] (\x+4, \y+0) -- (\x+4, \y+0.15);
	%	}

	%	\foreach \x/\y in {0/\h+1.65, 7/\h+1.65, 14/\h+1.65, 21/\h+1.65, 28/\h+1.65}		
	%	{
	%	\draw [dotted] (\x+0,\y) -- (\x+4,\y); 
	%	\draw [dashed] (\x+0,\y+0) -- (\x+0, \y-0.15);
	%	\draw [dashed] (\x+2,\y+0) -- (\x+2, \y-0.15);
	%	\draw [dashed] (\x+4, \y+0) -- (\x+4, \y-0.15);
	%	}
		\foreach \z/\y in {-1/-0.5}
		{	
		\draw (2, \h+1.5+\y) node[font=\fontsize{12}{12}, anchor=south] {$A_{1}$};
		\draw (0, \h+0.7+\z) node[anchor=south] {$A_{11}$};
		\draw (0+2, \h+0.7+\z) node[anchor=south] {$A_{12}$};
		\draw (0+4, \h+0.7+\z) node[anchor=south] {$A_{13}$};

		\draw (2+7, \h+1.5+\y) node[font=\fontsize{12}{12}, anchor=south] {$A_{2}$};
		\draw (7, \h+0.7+\z) node[anchor=south] {$A_{21}$};
		\draw (7+2, \h+0.7+\z) node[anchor=south] {$A_{22}$};
		\draw (7+4, \h+0.7+\z) node[anchor=south] {$A_{23}$};

		\draw (2+14, \h+1.5+\y) node[font=\fontsize{12}{12}, anchor=south] {$A_{3}$};
		\draw (14, \h+0.7+\z) node[anchor=south] {$A_{31}$};
		\draw (14+2, \h+0.7+\z) node[anchor=south] {$A_{32}$};
		\draw (14+4, \h+0.7+\z) node[anchor=south] {$A_{33}$};

		\draw (2+21, \h+1.5+\y) node[font=\fontsize{12}{12}, anchor=south] {$A_{4}$};
		\draw (21, \h+0.7+\z) node[anchor=south] {$A_{41}$};
		\draw (21+2, \h+0.7+\z) node[anchor=south] {$A_{42}$};
		\draw (21+4, \h+0.7+\z) node[anchor=south] {$A_{43}$};
	
		\draw (10+20, \h+1.5+\y) node[font=\fontsize{12}{12}, anchor=south] {$A_{5}$};
		\draw (10+18, \h+0.7+\z) node[anchor=south] {$A_{51}$};
		\draw (10+18+2, \h+0.7+\z) node[anchor=south] {$A_{52}$};
		\draw (10+18+4, \h+0.7+\z) node[anchor=south] {$A_{53}$};		
		}
		
	%	\draw (3+2+28, \h+1.5) node[anchor=south] {$A_{i+1}$};
				
		\foreach \z/\y in {1/0.5}
		{
		\draw (-3+2, -1.6+\y) node[font=\fontsize{12}{12}, anchor=north] {$B_{1}$};
		\draw (-3, -0.7+\z) node[anchor=north] {$B_{11}$};
		\draw (-3+2, -0.7+\z) node[anchor=north] {$B_{12}$};
		\draw (-3+4, -0.7+\z) node[anchor=north] {$B_{13}$};

		\draw (4+2, -1.6+\y) node[font=\fontsize{12}{12}, anchor=north] {$B_{2}$};
		\draw (4, -0.7+\z) node[anchor=north] {$B_{21}$};
		\draw (4+2, -0.7+\z) node[anchor=north] {$B_{22}$};
		\draw (4+4, -0.7+\z) node[anchor=north] {$B_{23}$};

		\draw (4+2+7, -1.6+\y) node[font=\fontsize{12}{12}, anchor=north] {$B_3$};
		\draw (11, -0.7+\z) node[anchor=north] {$B_{31}$};
		\draw (11+2, -0.7+\z) node[anchor=north] {$B_{32}$};
		\draw (11+4, -0.7+\z) node[anchor=north] {$B_{33}$};

		\draw (-1+21, -1.6+\y) node[font=\fontsize{12}{12}, anchor=north] {$B_{4}$};
		\draw (18, -0.7+\z) node[anchor=north] {$B_{41}$};
		\draw (18+2, -0.7+\z) node[anchor=north] {$B_{42}$};
		\draw (18+4, -0.7+\z) node[anchor=north] {$B_{43}$};
	
		\draw (-1+28, -1.6+\y) node[font=\fontsize{12}{12}, anchor=north] {$B_{5}$};	
		\draw (25, -0.7+\z) node[anchor=north] {$B_{51}$};
		\draw (25+2, -0.7+\z) node[anchor=north] {$B_{52}$};
		\draw (25+4, -0.7+\z) node[anchor=north] {$B_{53}$};
		}

	\foreach{\z} in {17} 
	{
	}

	}
	\end{tikzpicture}
	
	}
	%\includegraphics[scale=0.75]{proofchain01}
	%\includestandalone[width=\textwidth]{matchings}
	\caption{Proof of Lemma~\ref{partition2}}
	\label{template}
\end{figure}

Now we proceed to showing how to construct an $(n,m,q)$-chain template. Our first observation is that we only need to provide these constructions for \emph{skew-join prime} bipartite graphs.

\begin{definition}
For two bipartite graphs $G_1=(A_1, B_1, \mathcal{E}_1)$ and $G_2=(A_2, B_2, \mathcal{E}_2)$ the skew-join is defined to be $G_1 \oslash G_2 = (A_1 \cup A_2, B_1 \cup B_2, \mathcal{E}_1 \cup \mathcal{E}_2 \cup (A_1 \times B_2))$. We say that a bipartite graph $G=(A, B, \mathcal{E})$ is a skew-join prime, if, whenever $G=G_1 \oslash G_2$, we have either $G_1=\emptyset$ or $G_2=\emptyset$. 
\end{definition}

\begin{observation} \label{obs1}
If two graphs $G_1$ and $G_2$ both admit chain templates (resp. $(n,m,q)$-chain templates), then $G_1 \oslash G_2$ admits a chain template (resp. $(n,m,q)$-chain template). 
\end{observation}

\begin{proof}
It is easy to see that the union of the two partitions of $G_1$ and $G_2$ is the required partition for $G_1 \oslash G_2$.
\end{proof}

The following procedure describes our construction of chain templates.

\begin{lemma} \label{generictemplate}
Let $G=(A, B, \mathcal{E})$ be a bipartite skew-join prime graph with $B \neq \emptyset$. Consider the following procedure:
\begin{itemize}
\item Let $b_1 \in B$ be a vertex of minimal degree in $B$ and define $B_1:=\{b_1\}$.
\item Let $A_1:=N(B_1)$. 
\end{itemize}
Starting with $i=1$, we define further subsets $A_{i+1}$ and $B_{i+1}$ as follows:
\begin{itemize}
\item[(1)] Let $B_{i+1}:=(\overline{N}(A_i) \cap B) \backslash \cup_{h=1}^{i}B_h$. If $B_{i+1} \neq \emptyset$, go to (2), 
otherwise stop the procedure.  
\item[(2)] Let $A_{i+1}:= (N(B_{i+1}) \cap A) \backslash \cup_{h=1}^{i}A_h$. Increase the value $i:=i+1$ and go to (1).
\end{itemize}

Suppose the procedure stops at some $i=z$, i.e. $B_{z+1}=\emptyset$. Then the partition $A=\cup_{i=1}^z A_ i$,  $B=\cup_{i=1}^z B_i$ is a chain template of $G$. 
\end{lemma}

\begin{proof}

In this procedure we start by creating $B_1$, then we follow by $A_1, B_2, A_2, B_3, A_3$ and continue in alternating order. We need to show that this is a chain template, i.e. we need to show that this partition satisfies conditions (*) and (**) of Definition~\ref{templatedef}.  

We start by looking at the construction of bags $A_i$. The bag $A_1$ consists of the neighbourhood of $B_1$, $A_2$ consist of all the neighbourhood of $B_2$ not in $A_1$, $A_3$ - all the neighbourhood of $B_3$ that is not in $A_1 \cup A_2$, etc. As $A_1$ includes the neighbourhood of $B_1$, it follows that $B_1$ is co-joined to all the other bags $A_2, A_3, \ldots$ we will create in the future. Similarly, as $A_2$ contains all the new neighbourhood of $B_2$ not contained in $A_1$, $B_2$ will be cojoined to all the bags $A_3, A_4 \ldots$, we will create in the future. Thus, by construction, we can deduce that $B_j$ will be disjoint from all $A_i$ for all $j<i$, hence condition (**) holds. 

The construction of bags $B_i$ is similar, but this time we look at non-neighbourhood extensions. The bag $B_2$ consists of all non-neighbours of $A_1$ that do not belong to $B_1$, $B_3$ consists of all new non-neighbours of $A_2$, that do not belong to $B_1 \cup B_2$, and so on. Thus we conclude that $A_1$ will be joined to all the subsequent bags $B_3, B_4, B_5, \ldots$. Similarly, $A_2$ is joined to all bags $B_4, B_5, \ldots$, we will create in the future. Thus, we deduce that $A_i$ is joined to $B_j$ for all $j>i+1$, hence (*) holds.    

Suppose the procedure ended at some value $i=z$, i.e. $B_{z+1}=\emptyset$. By construction, it is not hard to note that $G[\cup_{i=1}^z A_i, \cup_{i=1}^z B_i]$ is skew-joined to the remaining graph $G[(A \backslash \cup_{i=1}^z A_i), (B \backslash \cup_{i=1}^z B_i])$. Indeed, by following the reasoning in the above two paragraphs, one may obtain that $\cup_{i=1}^z B_i$ is co-joined to $A \backslash \cup_{i=1}^z A_i$ and $\cup_{i=1}^z A_i$ is joined to $B \backslash \cup_{i=1}^{z+1} B_i = B \backslash \cup_{i=1}^{z} B_i$ as $B_{z+1}=\emptyset$. By our assumption that $G$ is skew-join prime, we conclude that the graph on the remaining vertices not covered by partition procedure $G[(A \backslash \cup_{i=1}^z A_i), (B \backslash \cup_{i=1}^z B_i)]$  must be empty. Hence, $A=\cup_{i=1}^z A_i$ and $B=\cup_{i=1}^z B_i$, and as (*) and (**) are satistied, this partition is a chain template of $G$. 

\end{proof}

Now we will prove structural results that concern consecutive bags of the chain template of $(nK_2, \overline{mK_2}^{bip})$-free graphs. Our target will be to refine the chain template constructed above to obtain $(n,m,q)$-chain template. Some of the critical observations that we use are as follows: 

\begin{observation}\label{simple1}
If a bipartite graph $G$ with parts $A$ and $B$ is $nK_2$-free, then there is a set $S \subseteq A$ with at most $|S| \leq n-1$ vertices such that $N(S)=N(A)$, i.e. such that $S$ covers the neighbourhood of $A$.  
\end{observation}

\begin{proof}
Start with $S:=A$. If there is an $a \in S$ such that $N(a) \subseteq N(S \backslash \{a\})$, then set $S:=S \backslash \{a\}$. 
Repeat this process until it is not possible remove a vertex from $S$. Clearly $N(S)=N(A)$ throughout the process.
Also, it is clear that at the end of the process, for every $a \in S$ the set $N(a) \backslash N(S \backslash \{a\}) \neq \emptyset$ as otherwise $a$ could be removed from $S$. Hence for every $a \in S$, we can pick $a' \in N(a) \backslash N(S \backslash \{a\})$. These pairs $\{aa': a \in S\}$ induce a matching, thus $|S| \leq n-1$ and we are done.
\end{proof}

\begin{observation}\label{simple2}
[Complementary statement to the previous one]
If the bipartite graph $G$ with parts $A$ and $B$ is $\overline{mK_2}$-free, then there is a set $S \subset A$ with at most $|S| \leq m-1$ vertices such that $\overline{N}(S)=\overline{N}(A)$, i.e. such that $S$ co-covers the non-neighbourhood of $A$. 
\end{observation}

The following lemma provides us with the required $(n, m, q)$ - chain template.

\begin{lemma}\label{nmchaintemplate}
Every bipartite graph $G = (A, B, \mathcal{E})  \in Free(nK_2, \overline{mK_2}^{bip})$ admits an $(n,m,q)$ - chain template for $q=(n-1)(m-1)$. \end{lemma}

\begin{proof}

First observe that it is enough to show that the statement of the lemma holds for every skew-join prime graph in $Free(nK_2, \overline{mK_2}^{bip})$. Indeed, if we showed for all skew-join prime graphs, by Observation~\ref{obs1} we would be able to extend this to the whole class. So let us assume that $G=(A, B, \mathcal{E}) \in Free(nK_2, \overline{mK_2}^{bip})$ is a skew-join prime graph and that $B \neq \emptyset$ as otherwise the lemma holds trivially. We start by constructing the chain template as described in Lemma~\ref{generictemplate}. We now need to describe the structural properties of the graphs induced by consecutive bags of the chain template and establish the conditions (***)--(****) of Definition~\ref{templatedef}. In the following three paragraphs we will analyse separately the structure of the graphs induced by $G[A_i, B_{i+1}]$ and $G[A_i, B_i]$ for $i \geq 2$ together with $G[A_1,B_2]$. This covers the structure of all the consecutive bags, except for $G[A_1, B_1]$ for which the structure is just a trivial join. The reader is advised to read the proofs below in conjunction with Figure~\ref{proofchainab}. 

First of all, consider the graph $G[A_i, B_i]$ for some $2 \leq i \leq z$ (see Figure~\ref{proofchainab} (b)). 
By construction described in Lemma~\ref{generictemplate}, $N(B_i)$ covers $A_i$ and the graph $G[A_i, B_i]$ is $nK_2$-free. Hence, by Observation~\ref{simple1}, we can deduce that 
there are $n-1$ vertices $b_{i1}, b_{i2}, \ldots, b_{i(n-1)} \in B_i$ whose neighbourhoods cover $A_i$. We can then 
define $A_{i1}=N(b_{i1}) \cap A_i$ and $A_{il}=(N(b_{il}) \cap A_i) \backslash \cup_{h=1}^{l-1}N(b_{ih})$ for $2 \leq l \leq n-1$.
In this way $A_{i1} \cup A_{i2} \cup \ldots \cup A_{i(n-1)}$ is a subpartition of $A_i$. Morever, we claim that $G[A_{il}, B_{i+1}]$
is $\overline{(m-1)K_2}^{bip}$-free. Indeed, $b_{il} \in B_i \subseteq \overline{N}(A_{i-1})$, hence there is some vertex $a \in A_{i-1}$ such that $ab_{il}$ is a non-edge of $G$. Since $a$ is joined to $B_{i+1}$ and $b_{il}$ is joined to $A_{il}$ and $G$ does not contain a bipartite co-matching on $m$ vertices, we conclude that $G[A_{il}, B_{i+1}] \in Free(\overline{(m-1)K_2}^{bip})$.       
 
Secondly, consider the graph $G[A_{i-1}, B_i]$ for some $2 \leq i \leq z$. Here, by definition of $B_i$, $\overline{N}(A_{i-1}) \supseteq B_i$. 
By Observation~\ref{simple2}, we can deduce that there are $m-1$ vertices 
$a_{(i-1)1}, a_{(i-1)2}, \ldots,$ $a_{(i-1)(m-1)}$ whose non-neighbourhood co-covers $B_i$. 
Now we define $B_{i1}=B_i \cap \overline{N}(a_{(i-1)1})$ and 
$B_{il}=B_i \cap \overline{N}(a_{(i-1)l}) \backslash \cup_{h=1}^{l-1} \overline{N}(a_{(i-1)h})$ for $2 \leq l \leq m-1$. 
This defines a subpartition $B_{i1} \cup B_{i2} \cup \ldots \cup B_{i(n-1)}$ of $B_i$.
Moreover, we claim that $G[B_{il},A_i]$ is $(n-1)K_2$-free. Indeed, $a_{(i-1)l} \in A_{i-1} \subseteq N(B_{i-1})$, hence
there is $b \in B_{i-1}$ which is adjacent to $a_{(i-1)l}$. As $b$ is co-joined 
to $A_i$ and $a_{(i-1)l}$ is co-joined to $B_{il}$, we conclude that $G[B_{il},A_i] \in Free{((n-1)K_2)}$.

We will now show how to partition the part $B_2$ into at most $n-1$ parts such that the induced subgraph between each part and $A_1$ contains no bipartite co-matching of size $m-1$ (see Figure~\ref{proofchainab} (a)). Consider $v \in B_2$. As $B_2$ is 
the non-neighbourhood of $A_1$, we know that $v$ has a non-neighbour in $A_1$, i.e. $N(v) \cap A_1 \neq A_1$. 
Now we will use the fact that $b_1$ is the vertex with minimal degree in $B$ as chosen in our construction in Lemma~\ref{generictemplate}. Since $b_1$ is a vertex with minimal degree in $B$, $v \in B$ has degree at least $|N(b_1)|=|A_1|$. We conclude that $v$ must have a neighbour outside $A_1$, and hence in $A_2$. 
Hence, each vertex of $B_2$ has a neighbour in $A_2$, i.e. $N(A_2) \supseteq B_2$. By Observation~\ref{simple1}, 
we deduce that  there are vertices $\{a_{11}, a_{12}, \ldots, a_{1(n-1)}\} \subseteq A_2$ whose neighbourhood 
covers $B_2$. Define $Y_{21}=B_2 \cap N(a_{11})$, 
$Y_{2l}=(B_2 \cap N(a_{1l})) \backslash \cup_{h=1}^{l-1}N(a_{1h})$ for every $2 \leq l \leq n-1$.
Now, as $b_1a_{1l}$ is a non-edge, $b_1$ is joined to $A_1$ and $a_{1l}$ is joined to $Y_{2l}$, we deduce that $G[Y_{2l}, A_1] \in Free(\overline{(m-1)K_2}^{bip})$ as required. 

Finally, let us put all the subpartitions together. We have the subpartitions $B_i=B_{i1} \cup \ldots \cup B_{i(m-1)}$
and $A_i=A_{i1} \cup A_{i2} \cup \ldots \cup A_{i(n-1)}$ for $i=2, 3, \ldots, z$. 
We also obtained a subpartition $B_2=Y_{21} \cup Y_{22} \cup \ldots \cup Y_{2(n-1)}$.
Since we have two subpartitions of $B_2$, we take a refinement of these, i.e. we subpartition $B_2$
into $(n-1)(m-1)$ parts $Y_{2j} \cap B_{2l}$ for $1 \leq j \leq n-1$ and $1 \leq l \leq m-1$.   
By the discussion in the previous paragraphs it easily follows that this final subpartition into
at most $q=(n-1)(m-1)$ parts satisfies conditions (***) and (****) of Definition~\ref{partition2}. 
(Note, in the definition of $(n,m,q)$-chain template, each part $A_i$ and $B_i$ is partitioned into \emph{exactly} $q$ parts and here we provided a partition into \emph{at most} $q$ sets, but it is easy to fix this by simply adding empty sets to the partition.)

\end{proof}

	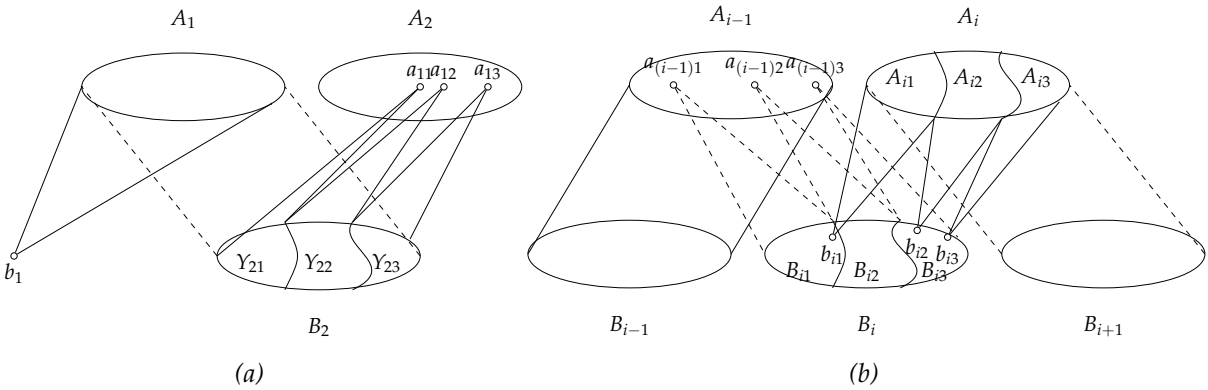
\begin{figure}[H]
       	
       	\begin{subfigure}[b]{0.43\textwidth}
					\scalebox{.75}
					{
					\begin{tikzpicture}[scale=.6,auto=left]

%%%%%%%%%%%%%%%%%%%%%%%%%%%%%%%%%%%%%%%%%%%%
		%\draw[step=1cm,gray,very thin] (0,0) grid (13,5);
		%\foreach \x in {0,1,2,3,4,5,6,7,8,9,10,11,12,13}
    		%\draw (\x cm,1pt) -- (\x cm,-1pt) node[anchor=north] {$\x$};
		%\foreach \y in {0,1,2,3,4}
    		%\draw (1pt,\y cm) -- (-1pt,\y cm) node[anchor=east] {$\y$};
		\foreach \start in {}
		{		
		\node[w_vertex] (1) at (\start+0, 2) {}; 	
		\node[w_vertex] (2) at (\start+1, 2) { };
		\node[w_vertex] (3) at (\start+2, 2) { };
		\node[w_vertex] (4) at (\start+3, 2) { };
		\node[w_vertex] (5) at (\start+0,0) { };
		\node[w_vertex] (6) at (\start+1,0) { };
		\node[w_vertex] (7) at (\start+2,0) { };
		\node[w_vertex] (8) at (\start+3,0) { };				
		\foreach \from/\to in {1/5,2/6,3/7,4/8}
	    	\draw (\from) -- (\to);
		%\foreach \from/\to in {5/6,6/7,7/8}
	    	%\draw (\from) -- (\to);		
		}
%%%%%%%%%%%%%%%%%%%%%%%%%%%%%%%%%%%%%%%%%%%%%%%

	\foreach{\h} in {5} 
	{
	%%%%%%%              Ellipses            %%%%%%
		\foreach \start in {0,7}
		{
 		\draw (\start+2,\h) ellipse (3cm and 1cm);
		}
		
		\foreach \start in {0}
		{
 		\draw (\start+6,0) ellipse (3cm and 1cm);
		}

		\draw [dashed] (3, 0) -- (-1, \h);
		\draw [dashed] (9, 0) -- (5, \h);

	%	\draw (13, 0) -- (16, \h);
	%	\draw (19, 0) -- (22, \h);
	%	\draw [dashed] (33, 0) -- (29, \h);
	%	\draw [dashed] (27, 0) -- (23, \h);
		
	%	\draw [dotted] (10, 0) -- (12, 0);
	%	\draw [dotted] (13, \h) -- (15, \h);
	%	\draw [dotted] (11.5, \h*0.5) -- (13.5, \h*0.5);
	%	\draw [dotted] (34, 0) -- (36, 0);	
	%	\draw [dotted] (35, \h*0.5) -- (36, \h*0.5);
		
		\draw (2, \h+1.5) node[anchor=south] {$A_{1}$};
		\draw (2+7, \h+1.5) node[anchor=south] {$A_{2}$};
	%	\draw (3+2+14, \h+1.5) node[anchor=south] {$A_{i-1}$};
	%	\draw (3+2+21, \h+1.5) node[anchor=south] {$A_{i}$};
	%	\draw (3+2+28, \h+1.5) node[anchor=south] {$A_{i+1}$};

	%	\draw (7, \h+0.7) node[anchor=south] {$A_{21}$};
	%	\draw (7+2, \h+0.7) node[anchor=south] {$A_{22}$};
	%	\draw (7+4, \h+0.7) node[anchor=south] {$A_{23}$};

	%	\draw (24, \h-0.3) node[anchor=south] {$A_{i1}$};
	%	\draw (24+2, \h-0.3) node[anchor=south] {$A_{i2}$};
	%	\draw (24+4, \h-0.3) node[anchor=south] {$A_{i3}$};

		\draw (4+2, -1.6) node[anchor=north] {$B_{2}$};
	%	\draw (7+2+7, -1.6) node[anchor=north] {$B_{i-1}$};
	%	\draw (2+21, -1.6) node[anchor=north] {$B_{i}$};
	%	\draw (2+28, -1.6) node[anchor=north] {$B_{i+1}$};

		\draw (4, 0.3) node[anchor=north] {$Y_{21}$};
		\draw (4+2, 0.3) node[anchor=north] {$Y_{22}$};
		\draw (4+4, 0.3) node[anchor=north] {$Y_{23}$};
		
	%	\draw (7+4, -0.7) node[anchor=north] {$B_{31}$};
	%	\draw (7+4+2, -0.7) node[anchor=north] {$B_{32}$};
	%	\draw (7+4+4, -0.7) node[anchor=north] {$B_{33}$};

	%	\draw (21, 0) node[anchor=north] {$B_{i1}$};
	%	\draw (21+2, 0) node[anchor=north] {$B_{i2}$};
	%	\draw (21+4, 0) node[anchor=north] {$B_{i3}$};
		
	%	\draw (28, -0.7) node[anchor=north] {$B_{(i+1)1}$};
	%	\draw (28+2, -0.7) node[anchor=north] {$B_{(i+1)2}$};
	%	\draw (28+4, -0.7) node[anchor=north] {$B_{(i+1)3}$};

	%	\foreach \x/\y in {4/-1.6, 21/-1.6}		
	%	{
	%	\draw [dotted] (\x+0,\y+0) -- (\x+4,\y+0); 
	%	\draw [dashed] (\x+0,\y+0) -- (\x+0, \y+0.15);
	%	\draw [dashed] (\x+2,\y+0) -- (\x+2, \y+0.15);
	%	\draw [dashed] (\x+4, \y+0) -- (\x+4, \y+0.15);
	%	}

	%	\foreach \x/\y in {24/6.65}		
	%	{
	%	\draw [dotted] (\x+0,\y) -- (\x+4,\y); 
	%	\draw [dashed] (\x+0,\y+0) -- (\x+0, \y-0.15);
	%	\draw [dashed] (\x+2,\y+0) -- (\x+2, \y-0.15);
	%	\draw [dashed] (\x+4, \y+0) -- (\x+4, \y-0.15);
	%	}

	%%%%%%           vertex b_1            %%%%%
		\node[w_vertex](1) at (-3, 0) {};
		\draw (1) node[anchor=north] {$b_1$};
		\node[point] (2) at (-1, \h) {};
		\node[point] (3) at (4.6, \h-0.5) {};
		\draw (1) -- (2);
		\draw (1) -- (3);	

	%%%%%%      vertices a_1 --- a_3     %%%%%%

		\node[w_vertex] (4) at (7+2, \h ) {};		
		\node[w_vertex] (5) at (7+2.7, \h) {};
		\node[w_vertex] (6) at (7+4, \h) {};
		\draw (4) node[anchor=south] {$a_{11}$};
		\draw (5) node[anchor=south] {$a_{12}$};
		\draw (6) node[anchor=south] {$a_{13}$};

		\draw  (4) -- (5,1);
		\draw  (4) -- (3,0);
		\draw  (5) -- (5,1);
		\draw  (5) -- (7,1);
		\draw  (6) -- (7,1);
		\draw  (6) -- (8.7,0.5);

		\draw (5,-1) .. controls (5.5,0)  .. (5,1);
		\draw (7,-1) .. controls (8.5,-0.3) and (6.5, 0.3) .. (7,1);

	}
	\end{tikzpicture}
					}
                	\caption{}
                	\label{fig:coP6_fig1}
        	\end{subfigure}%
        ~
        	\begin{subfigure}[b]{0.57\textwidth}
					\scalebox{.75}
					{
              	 	\begin{tikzpicture}[scale=.6,auto=left]

%%%%%%%%%%%%%%%%%%%%%%%%%%%%%%%%%%%%%%%%%%%%
		%\draw[step=1cm,gray,very thin] (0,0) grid (13,5);
		%\foreach \x in {0,1,2,3,4,5,6,7,8,9,10,11,12,13}
    		%\draw (\x cm,1pt) -- (\x cm,-1pt) node[anchor=north] {$\x$};
		%\foreach \y in {0,1,2,3,4}
    		%\draw (1pt,\y cm) -- (-1pt,\y cm) node[anchor=east] {$\y$};
		\foreach \start in {}
		{		
		\node[w_vertex] (1) at (\start+0, 2) {}; 	
		\node[w_vertex] (2) at (\start+1, 2) { };
		\node[w_vertex] (3) at (\start+2, 2) { };
		\node[w_vertex] (4) at (\start+3, 2) { };
		\node[w_vertex] (5) at (\start+0,0) { };
		\node[w_vertex] (6) at (\start+1,0) { };
		\node[w_vertex] (7) at (\start+2,0) { };
		\node[w_vertex] (8) at (\start+3,0) { };				
		\foreach \from/\to in {1/5,2/6,3/7,4/8}
	    	\draw (\from) -- (\to);
		%\foreach \from/\to in {5/6,6/7,7/8}
	    	%\draw (\from) -- (\to);		
		}
%%%%%%%%%%%%%%%%%%%%%%%%%%%%%%%%%%%%%%%%%%%%%%%

	\foreach{\h} in {5} 
	{
	%%%%%%%              Ellipses            %%%%%%
		\foreach \start in {17,24}
		{
 		\draw (\start+2,\h) ellipse (3cm and 1cm);
		}
		
		\foreach \start in {10,17,24}
		{
 		\draw (\start+6,0) ellipse (3cm and 1cm);
		}

	%	\draw [dashed] (3, 0) -- (-1, \h);
	%	\draw [dashed] (9, 0) -- (5, \h);

		\draw (13, 0) -- (16, \h);
		\draw (19, 0) -- (22, \h);
		\draw [dashed] (33, 0) -- (29, \h);
		\draw [dashed] (27, 0) -- (23, \h);
		
	%	\draw [dotted] (10, 0) -- (12, 0);
	%	\draw [dotted] (13, \h) -- (15, \h);
	%	\draw [dotted] (11.5, \h*0.5) -- (13.5, \h*0.5);
	%	\draw [dotted] (34, 0) -- (36, 0);	
	%	\draw [dotted] (35, \h*0.5) -- (36, \h*0.5);
		
	%	\draw (2, \h+1.5) node[anchor=south] {$A_{1}$};
	%	\draw (2+7, \h+1.5) node[anchor=south] {$A_{2}$};
		\draw (3+2+14, \h+1.5) node[anchor=south] {$A_{i-1}$};
		\draw (3+2+21, \h+1.5) node[anchor=south] {$A_{i}$};
	%	\draw (3+2+28, \h+1.5) node[anchor=south] {$A_{i+1}$};

	%	\draw (7, \h+0.7) node[anchor=south] {$A_{21}$};
	%	\draw (7+2, \h+0.7) node[anchor=south] {$A_{22}$};
	%	\draw (7+4, \h+0.7) node[anchor=south] {$A_{23}$};

		\draw (24, \h-0.3) node[anchor=south] {$A_{i1}$};
		\draw (24+2, \h-0.3) node[anchor=south] {$A_{i2}$};
		\draw (24+4, \h-0.3) node[anchor=south] {$A_{i3}$};

	%	\draw (4+2, -1.6) node[anchor=north] {$B_{2}$};
		\draw (7+2+7, -1.6) node[anchor=north] {$B_{i-1}$};
		\draw (2+21, -1.6) node[anchor=north] {$B_{i}$};
		\draw (2+28, -1.6) node[anchor=north] {$B_{i+1}$};

	%	\draw (4, 0.3) node[anchor=north] {$Y_{21}$};
	%	\draw (4+2, 0.3) node[anchor=north] {$Y_{22}$};
	%	\draw (4+4, 0.3) node[anchor=north] {$Y_{23}$};
		
	%	\draw (7+4, -0.7) node[anchor=north] {$B_{31}$};
	%	\draw (7+4+2, -0.7) node[anchor=north] {$B_{32}$};
	%	\draw (7+4+4, -0.7) node[anchor=north] {$B_{33}$};

		\draw (21, 0) node[anchor=north] {$B_{i1}$};
		\draw (21+2, 0) node[anchor=north] {$B_{i2}$};
		\draw (21+4, 0) node[anchor=north] {$B_{i3}$};

	\foreach{\z} in {17} 
	{
		\node[w_vertex] (21) at (\z+0.3, \h ) {};		
		\node[w_vertex] (22) at (\z+2.7, \h) {};
		\node[w_vertex] (23) at (\z+4.5, \h) {};
		\draw (21) node[anchor=south] {$a_{(i-1)1}$};
		\draw (22) node[anchor=south] {$a_{(i-1)2}$};
		\draw (23) node[anchor=south] {$a_{(i-1)3}$};

		\draw [dashed] (21) -- (\z+5,1);
		\draw [dashed] (21) -- (\z+3,0);
		\draw [dashed] (22) -- (\z+5,1);
		\draw [dashed] (22) -- (\z+7,1);
		\draw [dashed] (23) -- (\z+7,1);
		\draw [dashed] (23) -- (\z+8.7,0.5);

		\draw (\z+5,-1) .. controls (\z+5.5,0)  .. (\z+5,1);
		\draw (\z+7,-1) .. controls (\z+8.5,-0.3) and (\z+6, 0.3) .. (\z+7,1);

	%%%%%       vertices b_i1 --- b_i3     %%%%%%%

		\node[w_vertex] (24) at (\z+3+2,  0.5) {};		
		\node[w_vertex] (25) at (\z+3+4.5, 0.7) {};
		\node[w_vertex] (26) at (\z+3+5.4, 0.5) {};
		\draw (24) node[anchor=north] {$b_{i1}$};
		\draw (25) node[anchor=north] {$b_{i2}$};
		\draw (26) node[anchor=north] {$b_{i3}$};

		\draw  (25) -- (\z+3+5,\h-1);
		\draw  (25) -- (\z+3+7,\h-1);
		\draw  (26) -- (\z+3+7,\h-1);
		\draw  (26) -- (\z+3+8.7,\h-0.5);
		\draw  (24) -- (\z+3+5,\h-1);
		\draw  (24) -- (\z+3+3,\h);

		\draw (\z+3+5,\h-1) .. controls (\z+3+5.5,\h)  .. (\z+3+5,1+\h);
		\draw (\z+3+7,\h-1) .. controls (\z+3+8.5,\h-0.3) and (\z+3+6.5, 0.3+\h) .. (\z+3+7,1+\h);

		%%%%%%      vertices a_i1 --- a_i3     %%%%%%

%		\node[w_vertex] (27) at (\z+2+7, \h ) {};		
%		\node[w_vertex] (28) at (\z+2.7+7, \h) {};
%		\node[w_vertex] (29) at (\z+4+7, \h) {};
%		\draw (27) node[anchor=south] {$a_{i1}$};
%		\draw (28) node[anchor=south] {$a_{i2}$};
%		\draw (29) node[anchor=south] {$a_{i3}$};

%		\draw [dashed] (27) -- (\z+7+5,1);
%		\draw [dashed] (27) -- (\z+7+3,0);
%		\draw [dashed] (28) -- (\z+7+5,1);
%		\draw [dashed] (28) -- (\z+7+7,1);
%		\draw [dashed] (29) -- (\z+7+7,1);
%		\draw [dashed] (29) -- (\z+7+8.7,0.5);

%		\draw (\z+7+5,-1) .. controls (\z+7+5.5,0)  .. (\z+7+5,1);
%		\draw (\z+7+7,-1) .. controls (\z+7+8.5,-0.3) and (\z+7+6.5, 0.3) .. (\z+7+7,1);
	}

	}
	\end{tikzpicture}
	
					}
                	\caption{}
                	\label{fig:coP6_fig2}
        	\end{subfigure}

       	\caption{Proof of Lemma~\ref{nmchaintemplate}}\label{proofchainab}
	\end{figure}

Using the Lemmas above we are now ready to conclude that Theorem~\ref{bipartitepartition} holds.

% can be partitioned into $2(n-1)(m-1)$ pieces such that
% between the pieces is a graph from $X_{n-1,m}=Free((n-1)K_2, \overline{mK_2}^{bip})$ or 
% $X_{n, m-1}=Free(nK_2, \overline{(m-1)K_2})$.

%Hence, we obtained that every bipartite graph $G \in Free(nK_2, \overline{mK_2})$ has a partition satisfying 
% conditions of  Lemma~\ref{partition2} with $t=(n-1)(m-1)$. By Lemma~\ref{partition2} it follows that each such 
%graph is partitionable into $2(n-1)(m-1)$ sets such that between the sets the graph is either in 
%$Free((n-1)K_2, \overline{mK_2})$ or in $Free(nK_2, \overline{(m-1)K_2})$ and hence we are done.

\begin{proof}[Proof of Theorem~\ref{bipartitepartition}]
We will use induction on $n+m$. 
For $m \leq 2$ or $n \leq 2$, the graph is $2K_2$-free, so we can set $f(n,m)=1$. 
Suppose $m, n \geq 3$ and we proved the theorem for all pairs of integers with the sum smaller than $n+m$. 
Let $G$ be a bipartite graph in $Free(nK_2, \overline{mK_2}^{bip})$. By Lemma~\ref{nmchaintemplate}, we know that $G$ admits an $(n,m,q)$-chain template with $q=(n-1)(m-1)$. Hence by Lemma~\ref{partition2} $G$ can be partitioned into $2q=2(n-1)(m-1)$ parts such that between the parts, the graph is
either in $Free((n-1)K_2, \overline{mK_2})$ or in $Free(nK_2, \overline{(m-1)K_2})$. By induction, each pair of bags could be subpartitioned into $\max\{f(n, m-1), f(n-1, m)\}$ parts, with
$2K_2$-free graphs between them. To obtain the final partition, we just take the refinement of all these partitions. As each bag is partitioned into $\max\{f(n-1, m), f(n, m-1)\}$ parts with respect to each of the $2(n-1)(m-1)$ bags on the opposite side of the bipartition, this gives us $f(n,m) \leq 2(m-1)(n-1) \max\{f(n, m-1), f(n-1, m)\}^{2(m-1)(n-1)}$. This proves the existance of the number $f(n,m)$ and hence finishes the proof of the theorem. 

%$= 2^{2^{2(m+n) log(2mn) log(log(2mn)) }}$. 
\end{proof}

From the inductive formula provided, one can obtain a doubly exponential bound $f(n,m) \leq (2mn)^{(2mn)^{2(m+n)}}$. Though this general bound might look rather large for practical applications, we note that a single induction step gives a partition into only $2(n-1)(m-1)$ parts. This is a much smaller number and could be used to obtain better bounds for small $n$ and $m$ or given a graph with a large matching/co-matching obtain a partition not containing some smaller matching/co-matching between the parts (several inductive steps performed, instead of running until $2K_2$-free graphs are obtained). For example, with $n = m = 3$, noting that $\overline{3K_2}^{bip}=C_6$ - a cycle on 6 vertices, our result yields a partition into $2(m-1)(n-1)=8$ parts:

\begin{corollary}
Every $(3K_2, C_6)$-free bipartite graph has a partition of each part into 8 sets such that between the sets the 
induced graph is $2K_2$-free. 
\end{corollary}

\section{General graphs without a matching and complements of a matching} \label{section:matchingg}

We say that a graph $G$ has a cochromatic number at most $z$ if $V(G)$ can be partitioned into $z$ sets, each of them being either independent or a clique.
We will first prove that each graph in $Free(nK_2, \overline{mK_2})$ has cochromatic number at most $z(n,m)$ for some function depending only on $n$ and $m$ (here $\overline{mK_2}$ denotes the complement of $mK_2$, not the bipartite complement).  This gives us partition of a graph into finitely many cliques and independent sets. Then we use the previous section to conclude that the parts can be subdivided into finitely many pieces to ensure that the bipartite graph induced between the any pair of these pieces is $2K_2$-free, and hence giving us the desired partition. We start with the base case.

\begin{lemma} \label{cochrom1}
Any graph $G \in Free(2K_2, C_4)$ has cochromatic number at most 3. 
\end{lemma}

\begin{proof}
Let $G$ be a graph in $Free(2K_2, C_4)$ and let $X \subseteq V(G)$ be a set inducing a maximal independent set in $G$.
Split the set of vertices $V(G)\backslash X$ into a set $Y$ of vertices with exactly one neighbour in $X$, and a
 set $Z$ of vertices with at least 2 neighbours in $X$. We claim that both $Y$ and $Z$ are cliques. 
Indeed, take two vertices $y_1, y_2 \in Y$, and let $x_1, x_2 \in X$ be their unique neighbours in $X$, respectively. 
Suppose, for contradiction that $y_1$ and $y_2$ are non-adjacent. If $x_1=x_2=x$, then $(X \backslash x) \cup \{y_1, y_2\}$ is an 
independent set of size larger than $X$ and if $x_1 \neq x_2$, then $G[\{x_1, y_1, x_2, y_2\}]=2K_2$.
In both cases we arrive at contradiction, hence $Y$ is a clique.
Consider now $z_1, z_2 \in Z$, and assume for contradiction that the vertices are non-adjacent. If $z_1$ and $z_2$
has two common neighbours $x_1, x_2 \in X$, then $G[\{x_1, x_2, z_1, z_2\}]=C_4$. 
Otherwise, $z_1$ and $z_2$ has at most one common neighbour, which by definition of $Z$ implies that
they have at least one private neighbour $x_1 \in (N(z_1) \backslash N(z_2)) \cap X$ and $x_2 \in (N(z_2) \backslash N(z_1)) \cap X$.
These vertices induce $G[\{x_1, x_2, z_1, z_2\}]=2K_2$. A contradiction in both cases leads to conclusion that $Z$ is a clique.
Hence any $G \in Free(2K_2, C_4)$ can be partitioned into an independent set and two cliques which means that 
the cochromatic number is at most 3.  
\end{proof}

\begin{lemma}\label{comatching}
For any $n, m \geq 2$, there is an integer $z(n,m)$ such that any graph  
$G \in Free(nK_2, \overline{mK_2})$ has cochromatic number at most $z(n,m)$. Moreover, 
$z(n,m) \leq 3 \times 6^{(m-2)+(n-2)}$.
\end{lemma}

\begin{proof}
We will prove the statement by induction on $n+m$. By Lemma~\ref{cochrom1}, the result holds for $n=m=2$. Suppose now that $n \geq 2, m \geq 2$ and at least one of the inequalities is strict. Take $G \in Free(nK_2, \overline{mK_2})$.
Then either $G \in Free(2K_2, C_4)$, in which case we are done, or $G$ contains $2K_2$ or $C_4$. 

Suppose first $G$ contains a $C_4$ with vertex set $\{v_1, v_2, v_3, v_4\}$ and edges $\{v_1v_2, v_2v_3, v_3v_4, v_4v_1\}$. Partition $V(G)$ into 6 bags as follows. Place a vertex $v \in V(G)$ in the bag $B_{13}$ if it is adjacent to both $v_1$ and $v_3$ and place a vertex in the bag $B_{24}$ if it is adjacent to both $v_2$ and $v_4$. Further, place
a vertex $v$ in a bag $B_{12}$, $B_{23}$, $B_{34}$ or $B_{41}$ if $v$ is non-adjacent to both $v_1$ and $v_2$,
$v_2$ and $v_3$, $v_3$ and $v_4$, or $v_4$ and $v_1$, respectively. If a vertex is eligible to go to several bags, choose one arbitrarily. It is not hard to see that every vertex will go into some bag. Indeed, if $v$ does not go into bags $B_{13}$ or $B_{24}$, it means that it is non-adjacent to one of $v_1$ and $v_3$
and one of $v_2$ and $v_4$. This means that $v$ is co-joined to $\{v_1, v_2\}$, $\{v_1, v_4\}$, $\{v_3, v_2\}$ or $\{v_3, v_4\}$ and hence goes to one of $B_{12}$, $B_{23}$, $B_{34}$ or $B_{41}$. As $B_{13}$ is joined to $\{v_1, v_3\}$ and $v_1v_3$ is a non-edge, we deduce that $G[B_{13}]$ does not contain 
a comatching of size m-1, i.e. it lies in $Free(nK_2, \overline{(m-1)K_2})$. For the same reason $G[B_{24}] \in Free(nK_2, \overline{(m-1)K_2})$. Similarly, $B_{12}, B_{23}, B_{34}, B_{41}$ are all co-joined to a set that induces a $K_2$, hence the subgraphs induced by these sets lie in $Free((n-1)K_2, \overline{mK_2})$. By induction, each of these bags are partitionable into 
$3 \times 6^{(m-2)+(n-2)-1}$ independent set or cliques. Hence, $G$ has cochromatic number at most $6 \times (3 \times 6^{(m-2)+(n-2)-1})=3 \times 6^{(m-2)+(n-2)}$.
 
A similar argument works in the case when $G$ contains a $2K_2$, because it is the complement of $C_4$. 
Thus the induction proof holds in this case as well.
\end{proof}

We finish this section by adding all the ingredients together to obtain a proof of Theorem~\ref{thm:matchings}.

\begin{proof}[Proof of Theorem~\ref{thm:matchings}]
Let $G \in Free(\F_{n,1})$. From Lemma~\ref{comatching} we know that $G$ has cochromatic number bounded by $z(n,n)$. Consider a partition of $G$ into at most $z(n,n)$ cliques or independent sets. Between every pair of bags the induced bipartite graph must be $\{nK_2, \overline{nK_2}^{bip}\}$-free. Hence by Theorem~\ref{bipartitepartition}
we know it can be subpartitioned into $f(n,n)$ bags such that between the bags the graph is $2K_2$-free. Now each bag can be refined with respect to any other $z(n,n)-1$ bags, so this provides us with $z(n,n)-1$ partitions of each bag into $f(n,n)$ parts. Taking the refinement (intersection) of all such partitions on each bag we have a partition of each bag into $f(n,n)^{z(n,n)-1}$ parts. It is clear that each bag of the refined partition consists of cliques and independent sets and between the bags the graph is $2K_2$-free. Thus we obtained the desired partition with at most $z(n,n)f(n,n)^{z(n,n)}$ bags. Inserting the proved bounds for $z(n,n)$ and $f(n,n)$, we get a doubly exponential bound for $T(n)$.  
\end{proof}

\section{Bipartite graphs without a union of stars and a bipartite complement of a union of stars} \label{section:stars}

Our main structural result of this section is the following:

\begin{theorem} \label{thm:stars}
Let $G=(A, B, \mathcal{E})$ be a bipartite $(n'\up_k, n''\Lambda_k, \overline{m'\up_k}^{bip}, \overline{m''\Lambda_k}^{bip})$-free graph and let $\mu=n' + n'' + m' + m''$. 
Then there is a constant $U=U(\mu ,k)$ and a partition
$A=A_1 \cup A_2 \cup \ldots \cup A_u$, 
$B=B_1 \cup B_2 \cup \ldots \cup B_u$, with $u \leq U$ such that
for any $1 \leq i, j \leq u$ we have $G[A_i, B_j]$ is $2 \Lambda_{2k-1}$-free and $2 \up_{2k-1}$-free.
\end{theorem}

Most of what follows in this section is devoted to proving 
Theorem~\ref{thm:stars}. We will adapt most of the ideas from Section~\ref{section:matching} dealing with forbidden matching and co-matching. We will use the notation and basic concepts from Section~\ref{section:matching} introduced in Definition~\ref{basicdef}.

%%%%%%%%%%%%%%%%%%%%%%%%%%%%%%%%%
%%%%%             C      O     V      E      R      I      N       G         %%%%%%
%%%%%%%%%%%%%%%%%%%%%%%%%%%%%%%%%

\subsection{Covering $nK_{1,k}$-free bipartite graphs}

In Observation~\ref{simple1} we have shown that a bipartite graph $G=(A,B,\mathcal{E})$, which does not contain
$nK_2$ has a set of $n-1$ vertices in $A$ which cover the neighbourhood $N(A)$.
This, however, does not hold for the bipartite graphs not containing $nK_{1,k}$, for $k>1$. Indeed, for any integer $l$, one can construct a $nK_{1,k}$-free 
bipartite graph $G=(A,B,\mathcal{E})$ for which no set of $l$ vertices of $A$ would cover the neighbourhood $N(A)$. One easy example of such bipartite graph is a matching $(l+1)K_2$. In this subsection we argue about
the parts which can be covered and even covered multiple times by some
small sets in $A$. 

\begin{observation}
Let $G=(A, B, \mathcal{E})$ be an $n\Lambda_k$-free bipartite graph. Then there is a set $S \subseteq A$, $|S| \leq n-1$ such that 
for all $a \in A$, $|N(a) \backslash N(S)|<k$. 
\end{observation}

\begin{proof}
If $|A| \leq n-1$, then set $S=A$ and the required conditions are trivially satisfied. Otherwise, let $S \subseteq A$ be a set of $n-1$ vertices such that $|N(S)|$ is maximal. Suppose, for contradiction, there is $a \in A$ such that 
$|N(a) \backslash N(S)| \geq k$. Then, for any $s \in S$ we have 
\begin{align*}
|N(s) \backslash N(S-s \cup \{a\})| &= |N(S \cup \{a\}) \backslash N(S-s \cup \{a\})| \\ 
&= |N(S \cup \{a\})| - |N(S-s \cup \{a\})| \\
&\geq |N(S \cup \{a\})| - |N(S)| \\
&\geq k.
\end{align*}
Thus, the vertices $S \cup a$ are the centres of $n$ stars of size $k$, a contradiction.
\end{proof}

\begin{lemma} \label{covering}
Let $G=(A, B, \mathcal{E})$ be an $n\Lambda_k$-free bipartite graph, let $r \in \mathbb{N}$ be a positive integer and suppose $|A| \geq (n-1)r$. Then there is a set $W \subseteq A$ with 
$|W| = (n-1)r$ and 
a set $B^c \subseteq B$ such that every vertex $a \in A \backslash W$ satisfies $|N(a) \backslash B^c|<kr$ and every vertex 
$b \in B^c$ has at least $r$ neighbours in $W$.

Furthermore, there is a subset $W' \subseteq W$ of size $|W'|=n-1$ such that $B^c$ is covered by $W'$ and every vertex of $W'$ has at most $kr$ neighbours in $B \backslash B^c$.    
\end{lemma}

\begin{proof}
Take $S_1 \subseteq A$ to be a set such that $|S_1|=n-1$ and $|N(S_1)|$ is maximal. Now, for $h=2, 3, \ldots, r$ define
$S_h \subseteq A \backslash \cup_{i=1}^{h-1}S_i$ with $|S_h|=n-1$ and $|N(S_h)|$ maximal. Let $W=\cup_{i=1}^r S_i$ and $B^c= \cap_{i=1}^{r}N(S_i)$ (see Figure~\ref{fig:starproofL6andL7} (left)). Then clearly $|W| = (n-1)r$ and each vertex $b \in B^c$ is adjacent to at least one vertex in $S_i$ for each $i=1,2, \ldots r$, hence is adjacent to at least $r$ vertices in $W$. Furthermore, by the previous observation, for all $a \in A \backslash W$ and for all $1 \leq i \leq r$ we have $|N(a) \backslash N(S_i)|<k$. Hence for all $a \in A \backslash W$, 
$|N(a) \backslash B^c|= |N(a) \backslash \cap_{i=1}^r N(S_i)| = |\cup_{i=1}^r (N(a) \backslash N(S_i))| \leq \sum_{i=1}^r |N(a) \backslash N(S_i)| \leq kr$.

Let us take a subset $W'=S_r$. Clearly $W'$ covers $B^c$. Note that for any $1 \leq i \leq r-1$ and any $w \in W'$, we have $|N(w) \backslash N(S_i)| \leq k$, and clearly as $w \in S_r$, we have $|N(w) \backslash N(S_r)|=0$. Hence, we have $|N(w) \backslash B^c| \leq \sum_{i=1}^r |N(w) \backslash N(S_i)| \leq k(r-1)<kr$. 
\end{proof}

Lemma~\ref{covering} motivates the following definition.
\begin{definition}
Let $G$ be a graph and $V_1, V_2 \subset V(G)$ be two disjoint vertex subsets of $G$. We will say that $V_2$ is $r$-covered by $V_1$ if each vertex of $V_2$ has $r$ neighbours in $V_1$. Similarly, we will say $V_2$ is $r$-co-covered by $V_1$ if each vertex of $V_2$ has $r$ non-neighbours in $V_1$.  
\end{definition}

Thus, when we refer to Lemma~\ref{covering}, we will frequently say that $B^c$ is \emph{$r$-covered} by $W$. We finish this subsection with one more preliminary result:

\begin{lemma}\label{boundeddegree}
Let $G=(A, B, \mathcal{E})$ be an $n\Lambda_k$-free bipartite graph such that the vertices in part $B$ have degree bounded by $d$. 
Then there are at most $nkd^2$ vertices of $A$ with degree at least $k$.
\end{lemma}

\begin{proof}
By the previous lemma there is a set $W \subseteq A$ with $|W| \leq (n-1)d$ and a set $B^c \subseteq B$ such that every vertex of $B^c$ is covered by $d$ elements of $W$ and such that every $a \in A \backslash W$ satisfies $|N(a) \backslash B^c|<kd$. 
It is easy to see that any $a \in A \backslash W$ has no neighbours in $B^c$ as otherwise a vertex of $B^c$ would have a degree more than $d$ (see Figure~\ref{fig:starproofL6andL7} (right)). 
Hence, for every $a \in A \backslash W$, the degree is bounded by $kd$. Now take any vertex $a \in A \backslash W$ such that
$k \leq |N(a)| \leq kd$. As each vertex of $N(a)$ has degree bounded by $d$, there are at most $kd(d-1)$ vertices in $A$ which
intersect neighbourhood of $N(a)$. Thus, if we have $kd(d-1)(n-1)+n$ vertices with $k \leq |N(a)| \leq kd$, then it is
easy to see that we have $n\Lambda_k$. Hence there are at most $kd(d-1)(n-1)+n+(n-1)d \leq kd^2n$ vertices of degree at least $k$ in $A$.    
\end{proof}

\begin{figure}[H]
	\centering
	\begin{tikzpicture}[scale=.6,auto=left]

%%%%%%%%%%%%%%%%%%%%%%%%%%%%%%%%%%%%%%%%%%%%
		%\draw[step=1cm,gray,very thin] (0,0) grid (13,5);
		%\foreach \x in {0,1,2,3,4,5,6,7,8,9,10,11,12,13}
    		%\draw (\x cm,1pt) -- (\x cm,-1pt) node[anchor=north] {$\x$};
		%\foreach \y in {0,1,2,3,4}
    		%\draw (1pt,\y cm) -- (-1pt,\y cm) node[anchor=east] {$\y$};
\iffalse
		\foreach \start in {}
		{		
		\node[w_vertex] (1) at (\start+0, 2) {}; 	
		\node[w_vertex] (2) at (\start+1, 2) { };
		\node[w_vertex] (3) at (\start+2, 2) { };
		\node[w_vertex] (4) at (\start+3, 2) { };
		\node[w_vertex] (5) at (\start+0,0) { };
		\node[w_vertex] (6) at (\start+1,0) { };
		\node[w_vertex] (7) at (\start+2,0) { };
		\node[w_vertex] (8) at (\start+3,0) { };				
		\foreach \from/\to in {1/5,2/6,3/7,4/8}
	    	\draw (\from) -- (\to);
		%\foreach \from/\to in {5/6,6/7,7/8}
	    	%\draw (\from) -- (\to);		
		}
\fi
%%%%%%%%%%%%%%%%%%%%%%%%%%%%%%%%%%%%%%%%%%%%%%%

	\foreach{\h} in {7} 
	{
\iffalse
		\foreach \x/\y in {0/\h+2}
		{
		\draw [dotted] (\x+0,\y+0) -- (\x+33,\y+0); 
		\draw  (\x+0,\y+0) -- (\x+0, \y-1);
		\draw  (\x+14,\y+0) -- (\x+14, \y-1);
		\draw  (\x+28, \y+0) -- (\x+28, \y-1);
		}

		\draw (0, \h+2) node[anchor=south] {$A'_{11}$};
	
		\foreach \x/\y in {1/-2}
		{
		\draw [dotted] (\x+0,\y+0) -- (\x+32,\y+0); 
		\draw  (\x+0,\y+0) -- (\x+0, \y+1);
		\draw  (\x+14,\y+0) -- (\x+14, \y+1);
		\draw  (\x+28, \y+0) -- (\x+28, \y+1);
		}
		\draw (1, -2) node[anchor=north] {$B'_{13}$};
\fi
	%%%%%%%              Shades            %%%%%%
	
\iffalse	
		\foreach{\start} in {0, 7}
			{\fill[blue!15!white] (\start-4,0) -- (\start-1,\h) -- (\start+5,\h) -- (\start+2,0) -- cycle;}
		\foreach{\start} in {0, 7}
			{\fill[blue!15!white] (\start+3,0) -- (\start-1,\h) -- (\start+5,\h) -- (\start+9,0) -- cycle;}
		
		\foreach{\start} in {0, 7}
			{
			\draw (\start+-4, 0) -- (\start-1, \h);
			\draw (\start+2, 0) -- (\start+5, \h);
			}
		\foreach{\start} in {0, 7}
			{
			\draw (\start+-1, \h) -- (\start+3, 0);
			\draw (\start+5, \h) -- (\start+9, 0);		
			}	
\fi
		
	%	\draw [dotted] (10, 0) -- (12, 0);
	%	\draw [dotted] (13, \h) -- (15, \h);
	%	\draw [dotted] (11.5, \h*0.5) -- (13.5, \h*0.5);
	%	\draw [dotted] (34, 0) -- (36, 0);	
	%	\draw [dotted] (35, \h*0.5) -- (36, \h*0.5);

	%%%%%%%%%%%   Ellipses     %%%%%%%%%%%%
		\foreach \start in {0,14}
		{\draw[fill=white] (\start,\h) ellipse (4cm and 1cm);	}
		
		\foreach \start in {0,14}
		{\draw[fill=white] (\start,0) ellipse (4cm and 1cm);}

		\foreach \start in {0,14}
		{\draw[fill=white] (\start+1.8, 0) ellipse (1.6cm and 0.5cm);}

		\foreach \z in {}
			{
			\draw (\z+5,-1) .. controls (\z+5.5,0)  .. (\z+5,1);
			\draw (\z+7,-1) .. controls (\z+8.5,-0.3) and (\z+6.5, 0.3) .. (\z+7,1);
			}
		\foreach \z in {}
			{
			\draw (\z+3+5,\h-1) .. controls (\z+3+5.5,\h)  .. (\z+3+5,1+\h);
			\draw (\z+3+7,\h-1) .. controls (\z+3+8.5,\h-0.3) and (\z+3+6.5, 0.3+\h) .. (\z+3+7,1+\h);
			}

%%%%%%%%%%%%%%%% square %%%%%%%%%%%%%%%%%%%

\foreach \x/\y in {2/6.7, 16/6.7}
{
\draw (\x, \y) -- (\x+1, \y) -- (\x+1, \y+0.5) -- (\x, \y+0.5) -- cycle ;
}

	%%%%%%%%%%%%%labelling %%%%%%%%%%%%%

	%	\foreach \x/\y in {-3/-1.6, 4/-1.6,  11/-1.6, 18/-1.6, 25/-1.6}		
	%	{
	%	\draw [dotted] (\x+0,\y+0) -- (\x+4,\y+0); 
	%	\draw [dashed] (\x+0,\y+0) -- (\x+0, \y+0.15);
	%	\draw [dashed] (\x+2,\y+0) -- (\x+2, \y+0.15);
	%	\draw [dashed] (\x+4, \y+0) -- (\x+4, \y+0.15);
	%	}

	%	\foreach \x/\y in {0/\h+1.65, 7/\h+1.65, 14/\h+1.65, 21/\h+1.65, 28/\h+1.65}		
	%	{
	%	\draw [dotted] (\x+0,\y) -- (\x+4,\y); 
	%	\draw [dashed] (\x+0,\y+0) -- (\x+0, \y-0.15);
	%	\draw [dashed] (\x+2,\y+0) -- (\x+2, \y-0.15);
	%	\draw [dashed] (\x+4, \y+0) -- (\x+4, \y-0.15);
	%	}

\foreach \z/\y in {-1/-0.5}
		{	
		\draw (0, \h+1.5+\y) node[font=\fontsize{12}{12}, anchor=south] {$A$};
		\draw (0+14, \h+1.5+\y) node[font=\fontsize{12}{12}, anchor=south] {$A$};
		}

\foreach \z/\y in {1/0.5}
		{
		\draw (0, -1.6+\y) node[font=\fontsize{12}{12}, anchor=north] {$B$};
		\draw (0+14, -1.6+\y) node[font=\fontsize{12}{12}, anchor=north] {$B$};
		}

		\node[w_vertex] (11) at (1.5, 0 ) {};		
		\node[w_vertex] (12) at (2, 0) {};
		\node[w_vertex] (13) at (2.5, 0) {};
		\node[w_vertex](14) at (3.1, 0){};
		
		\node[w_vertex] (15) at (1.5+14, 0 ) {};		
		\node[w_vertex] (16) at (2+14, 0) {};
		\node[w_vertex] (17) at (2.5+14, 0) {};
		\node[w_vertex](18) at (3.1+14, 0){};

		\node[w_vertex](19) at (-2, \h){};
		\draw(19) -- (-2.5, 0); 
		\draw(19) -- (-1.5, 0);
		\draw[fill=white] (-2, 0) ellipse (0.5cm and 0.2cm);
		\draw(19) -- (12);
		\draw (-2, \h) node[font=\fontsize{12}{12}, anchor=south] {$a$};
		\draw (-2, 0) node[font=\fontsize{12}{12}, anchor=north] {$N(a) \backslash B^c$};
		
		\node[w_vertex](20) at (-2.4+14, \h){};
		\node[w_vertex](21) at (-1.6+14, \h){};
		\draw (-2.4+14, \h) node[font=\fontsize{12}{12}, anchor=south] {$a$};
		\draw (-1.6+14, \h) node[font=\fontsize{12}{12}, anchor=south] {$a'$};
		\draw (-1.7+14, 0) ellipse (0.5cm and 0.2cm);
		\draw (-2.3+14, 0) ellipse (0.5cm and 0.2cm);
		\draw(20) -- (-2.8+14, 0); 
		\draw(20) -- (-1.8+14, 0); 
		\draw(21) -- (-2.2+14, 0); 
		\draw(21) -- (-1.2+14, 0); 
		\draw[dashed] (20) -- (14+0.2, 0);		
		\draw[dashed] (20) -- (14+3.4, 0);		
		\draw[dashed] (21) -- (14+0.2, 0);		
		\draw[dashed] (21) -- (14+3.4, 0);		
		\draw (-2.8+14, 0) node[font=\fontsize{12}{12}, anchor=north] {$N(a)$};
		\draw (-1.2+14, 0) node[font=\fontsize{12}{12}, anchor=north] {$N(a')$};

		\draw (11) node[anchor=south] {};
		\draw (12) node[anchor=south] {};
		\draw (13) node[anchor=south] {};

		\draw (11) -- (2.1,7);
		\draw (11) -- (2.5, 7);
		\draw (11) -- (2.9, 7);
		
		\draw (12) -- (2.2,7);
		\draw (12) -- (2.4,7);
		\draw (12) -- (2.7,7);	

		\draw  (13) -- (2.1,7);
		\draw  (13) -- (2.4, 7);
		\draw  (13) -- (2.9, 7);

		\draw (14) -- (2.2, 7);
		\draw (14) -- (2.5, 7);
		\draw (14) -- (2.9, 7);

		\draw (15) -- (2.1+14,7);
		\draw (15) -- (2.5+14, 7);
		\draw (15) -- (2.9+14, 7);
		
		\draw (16) -- (2.2+14,7);
		\draw (16) -- (2.4+14,7);
		\draw (16) -- (2.7+14,7);	

		\draw  (17) -- (2.1+14,7);
		\draw  (17) -- (2.4+14, 7);
		\draw  (17) -- (2.9+14, 7);

		\draw (18) -- (2.2+14, 7);
		\draw (18) -- (2.5+14, 7);
		\draw (18) -- (2.9+14, 7);

		\draw (2.5, \h) node[font=\fontsize{12}{12}, anchor=south] {$W$};
		\draw (2.5+14, \h) node[font=\fontsize{12}{12}, anchor=south] {$W$};
		
		\draw (1, -0.5) node[font=\fontsize{12}{12}, anchor=south] {$B^c$};
		\draw (1+14, -0.5) node[font=\fontsize{12}{12}, anchor=south] {$B^c$};

	\foreach{\z} in {17} 
	{
	}

	}
	\end{tikzpicture}
	
	\caption{Proofs of Lemma~\ref{covering} (left) and Lemma~\ref{boundeddegree} (right)}
	\label{fig:starproofL6andL7}
\end{figure}
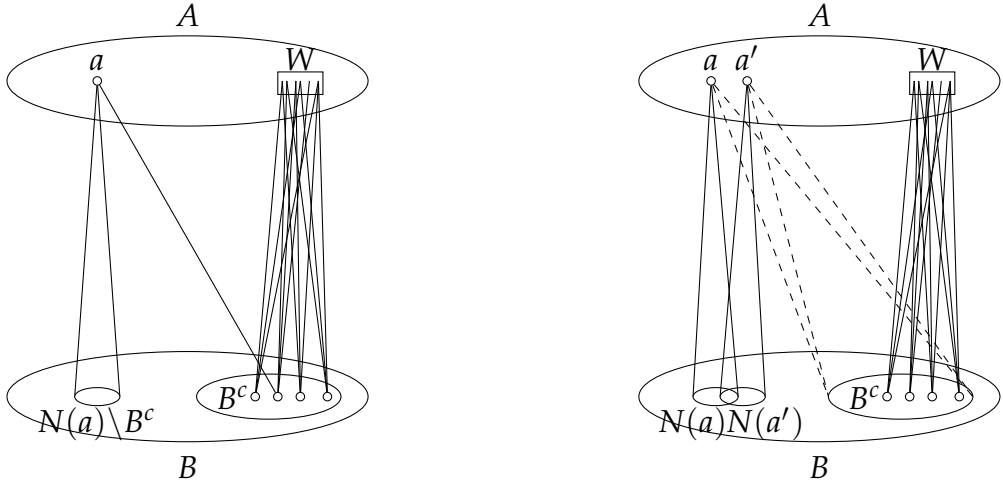

%%%%%%%%%%%%%%%%%%%%%%%%%%%%%%%%%
%%%%%%%%%%   P    A    R    T    I    T    I    O    N   %%%%%%%%
%%%%%%%%%%%%%%%%%%%%%%%%%%%%%%%%%

\subsection{$d$-template}
We will be saying that a subset $A$ is \emph{$d$-joined} (resp. \emph{$d$-co-joined}) to a subset $B$, 
if every vertex $a \in A$ has at most $d$ non-neighbours (resp $d$ neighbours) in $B$. 
We will further say that A and B is a \emph{$d$-join} (resp. \emph{$d$-co-join}), if both $A$ is $d$-joined 
(resp. $d$-co-joined) to $B$, and $B$ is $d$-joined (resp. $d$-co-joined) to $A$.

\begin{definition} \label{partition3}
Let $G=(A, B, \mathcal{E})$ be a bipartite graph. We will call the partition $A=A_1 \cup A_2 \cup \ldots \cup A_z$,
$B=B_1 \cup B_2 \cup \ldots \cup B_z$ a $d$-template, if it satisfies the following two conditions:
\begin{itemize}[noitemsep,topsep=0pt,parsep=0pt,partopsep=0pt,labelwidth=2.5cm,align=left,itemindent=2.5cm]
\item[(*)] $A_i$ is $d$-joined to $\cup_{h=i+2}^z B_h$ and $A_i$ is $d$-co-joined to $\cup_{h=1}^{i-1} B_h$
%$\cup_{h=1}^i B_h$ and $\cup_{h=i+1}^s A_h$ is a $d$-co-join for all $i$. 
\item[(**)] $B_i$ is $d$-joined to $\cup_{h=1}^{i-2} A_h$ and $B_i$ is $d$-co-joined to $\cup_{h=i+1}^z A_h$  
%$\cup_{h=1}^i A_h$ and $\cup_{h=i+2}^{s} B_h$ is a $d$-join for all $i$. 
\end{itemize}

Suppose further, that each part $A_i$ and $B_i$ is subdivided into $q$ pieces $A_i=A_{i1} \cup A_{i2} \cup \ldots \cup A_{iq}$ and $B_i=B_{i1} \cup B_{i2} \cup \ldots \cup B_{iq}$ such that for some integers $n', m', m'', k  \in \mathbb{N}$ and some subset $I \subseteq \{1,2, \ldots, z \}$ the following holds:
\begin{itemize}[noitemsep,topsep=0pt,parsep=0pt,partopsep=0pt,labelwidth=2.5cm,align=left,itemindent=2.5cm]
\item[(***)   ] If $j=i$, then the graph $G[A_{ig}, B_{jh}]$ is in $Free((n'-1)\up_{k})$ for all $1 \leq g,h \leq q$.
\item[(****)   ] If $j=i+1$, then the graph $G[A_{ig}, B_{jh}]$ is either in $Free (\overline{(m'-1)\up_{k}}^{bip})$ or in  $Free (\overline{(m''-1)\Lambda_{k}}^{bip})$, depending on whether $i \in I$ or $i \not\in I$, respectively, for all $1 \leq g, h \leq q$.

%If $j=i+1$ and $i \in I$, then $G[A_{ig}\cup B_{jh}]$ is in $Free (\overline{(m'-1)\up_{k}}^{bip})$ for all $1 \leq g,h \leq q$. If $j=i+1$ and $i \not\in I$ then $G[A_{ig}\cup B_{jh}]$ is in  $Free (\overline{(m''-1)\Lambda_{k}}^{bip})$ for all $1 \leq g,h \leq q$. 

\end{itemize}

Then we will call this refined partition a $(n',m',m'', k, q, d)$-template or a refined $d$-template (with parameters $(n',m',m'', k,q)$). 

\end{definition}

One can observe that a $d$-template is an extension of a chain template, which is just a 0-template. One can see from the definition of $d$-template that there is a $d$-join or a $d$-co-join between all pairs of bags $A_i$ and $B_j$ for all $i, j$ except for the consecutive $i=j$ or $i=j+1$. Thus Figure~\ref{template} is a good visualisation of $d$-template, where we are mostly interested in getting to know some extra structure in the shaded regions. We note, however, that conditions (*) and (**) require not only a $d$-join between each pair of non-consecutive bags in the partition, but between a bag in one part and a union of bags in the other. This implies, for instance, that a vertex in $A_i$ can have at most $d$ neighbours in $\cup_{h=1}^{i-1} B_h$ and at most $d$-non-neighbours in $\cup_{h=i+2}^z B_h$. This ensures that $d$-template partitions the vertices in the bags in such a way that, for any $v \in A_i$ and $w \in A_j$ with $j>i+1$, $N(v)$ includes all but at most $2d$ vertices of $N(w)$. Thus, intuitively, the vertices are partitioned in a way that respects neighbourhood inclusion.

The $d$-template has several advantages over $0$-template or chain template. As we will see later, the $d$-template has strong covering properties, where the bags are covered by many vertices from the preceding bag. This allows us to deduce much stronger structural restrictions on the bipartite graphs induced by consecutive bags. In particular, it allows us to obtain the restrictions described in (***) and (****), which will be crucial for our inductive argument, similar to the one described in Section~\ref{section:matching}. The next Lemma is an extension of Lemma~\ref{partition2} which shows that $d$-template can be collapsed into a bounded number of parts, preserving the structural properties discovered in (***) and (****).

\begin{lemma}\label{templatecollapse}
Let $G=(A, B, \mathcal{E})$ be a bipartite graph that admits a refined $d$-template with parameters $(n',m',m'',k,q)$.
Then $A$ and $B$ can be partitioned into $4q$ parts each such that the graphs between 
any two parts is either in $Free((n-1)\up_{k+2d})$, $Free(\overline{(m-1)\up_{k+2d}}^{bip})$, or in
$Free (\overline{(m'-1)\Lambda_{k+2d}}^{bip})$. 
\end{lemma}

\begin{proof}

We use the notation of the parts of refined $d$-template as given in Definition~\ref{partition3}. 
For each $1 \leq g \leq q$, define
 
\begin{align*}
A'_{1g} = \bigcup_{\substack{i \ odd,\\ i \in I}} A_{ig}, \qquad &  
A'_{2g} = \bigcup_{\substack{i \ odd,\\ i \not\in I}} A_{ig},& 
A'_{3g} = \bigcup_{\substack{i \ even \\ i \in I}} A_{ig}, \qquad &  
A'_{4g} = \bigcup_{\substack{i \ even,\\ i \not\in I}} A_{ig},\\
B'_{1g} = \bigcup_{\substack{i \ odd,\\ i \in I}} B_{ig}, \qquad & 
B'_{2g} = \bigcup_{\substack{i \ odd,\\ i \not\in I}} B_{ig},& 
B'_{3g} = \bigcup_{\substack{i \ even \\ i \in I}} B_{ig}, \qquad & 
B'_{4g} = \bigcup_{\substack{i \ even,\\ i \not\in I}} B_{ig}.
\end{align*}

We will show that $A=\bigcup_{\substack{1\leq i \leq 4,\\ 1 \leq g \leq q}} A'_{ig}$, $B=\bigcup_{\substack{1\leq i \leq 4,\\ 1 \leq g \leq q}}B'_{ig}$ is the required partition into $4q$ parts. In other words, we will show that between any two parts $A'_{ig}$ and $B'_{jh}$, the induced graph $G[A'_{ig}, B'_{jh}]$ is either in $Free((n'-1)\up_{k+2d})$, $Free(\overline{(m'-1)\up_{k+2d}}^{bip})$, or in
$Free (\overline{(m''-1)\Lambda_{k+2d}}^{bip})$. 

Consider the graph $G[A'_{1g}, B'_{1h}]$. Suppose, for contradiction, this part contains an induced subgraph $(n-1)\up_{k+2d}$. Let us denote the centres of the stars $(n-1)\up_{k+2d}$ by $v_1, v_2, \ldots, v_{n-1}$. These centres belong to $B'_{1h}$, thus to the bags $\bigcup_{2i+1 \in I} B_{(2i+1) h}$. Consider first the case when two centres belong to different bags $v_x \in B_{(2i+1)h}$, $v_y \in B_{(2j+1)h}$ with $i<j$. Then, by (**) one can see that $N(v_x)$ contains at most $d$ vertices of $\cup_{l=2i+2}^z A_l$ while $N(v_y)$ contains all except possibly $d$ vertices of $\cup_{l=1}^{2j-1} A_l \supseteq \cup_{l=1}^{2i+1} A_l$ as $i<j$. This means that $|N(v_y) \backslash N(v_x)| \leq 2d$, and hence a contradiction as $v_x$ and $v_y$ being centres of stars satisfy $|N(v_y) \backslash N(v_x)| \geq 2d+k>2d$. We conclude that all the vertices $v_1, v_2, \ldots, v_{n-1}$ belong to some $B_{2i+1, h}$ for some $i$. Now, again, by (**) we have that $B_{(2i+1) h}$ is $d$-joined to $\cup_{l=1}^{2i-1} A_l$ and $d$-co-joined to $\cup_{l=2i+2}^{z} A_l$. This means that for any two centres of stars $v_x$ and $v_y$ we have $|(N(v_x) \backslash N(v_y)) \cap (\cup_{l=1}^{2i-1} A_l \cup_{l=2i+2}^{z} A_l)| \leq 2d$. Thus, since $v_x$ and $v_y$ are centres of $(n-1)\up_{2d+k}$, at least $k$ neighbours of $v_x$ in this star forest must belong to $A_{2i+1, g}$. This holds for all $1 \leq x \leq n-1$, hence we have a $(n-1)\up_{k}$ in $G[A_{(2i+1) g}, B_{(2i+1) h}]$ which is a contradiction. This proves that $G[A'_{1g}, B'_{1h}]$ is in $Free((n-1)\up_{k+2d})$. Similar arguments can be used to deduce the required results for the other pairs of parts of the partition.
\end{proof}

%%%%%%%%%%%%%%%%%%%%%%%%%%%%%%%%%%%%%%
%%%%%%%%%%%%%%%%%%%%%%%%%%%%%%%%%%%%%%

\subsection{Partition procedure and conditions (*)-(**)} \label{partitionprocedure}

Let $G=(A,B, \mathcal{E})$ be a $(n'\up_k, n''\Lambda_k, \overline{m'\up_k}^{bip}, \overline{m''\Lambda_k}^{bip})$-free graph for some $n', n'', m', m'' \in \mathbb{N}$ and $n', n'', m', m'' \geq 3$. Let $n=max(n', n'', m', m'')$. In this section we will describe the partition procedure we use to obtain $d$-template for $d=(n-1)r+kr$. Here $r$ is any positive integer, and the bigger the value of $r$ the stronger the covering properties will be between consecutive bags, which we will use in the next section. However, the increased covering properties between consecutive bags comes with the ''cost'' of increased $d$ resulting in weaker $d$-join and $d$-co-join relationships between the non-consecutive bags.  

We will now create partition $A = A_1 \cup A_2 \cup \ldots \cup A_z$, $B = B_1 \cup B_2 \cup \ldots \cup B_z$ that will lead us to a $d$-template. In the process of construction we will be first constructing supersets $A_i^+ \supseteq A_i$ and $B_i^+ \supseteq B_i$, from which, after removal of some vertices, the required sets of the partition $A_i$ and $B_i$ will be obtained. We start by taking $B_1^+=B_1=\{b\}$ for a vertex $b$ with minimal degree in $B$ and we let $A_1^+=N(B_1)$. We will create the bags $B_i$ and $A_i$ in alternating order and we will confirm that the following conditions hold:

\begin{itemize}
\item[$(P_1i)$] Vertices in $A_{i-1}$ have co-degree at most $d=kr+(n-1)r$ in $B \backslash (B_1 \cup B_2 \cup \ldots \cup B_{i})$
\item[$(P_2i)$] Vertices in $B_{i}$ have degree at most $kr$ in $A \backslash (A_1 \cup A_2 \cup \ldots \cup A^+_i)$
\item[$(P_3i)$] Vertices in $B_{i}$ have degree at most $d=kr+(n-1)r$ in $A \backslash (A_1 \cup A_2 \cup \ldots \cup A_{i})$
\item[$(P_4i)$] Vertices in $A_i$ have co-degree at most $kr$ in $B \backslash (B_1 \cup B_2 \cup \ldots \cup B_{i+1}^+)$
\end{itemize}

Note that when the $B_1$ and $A_1^+$ are constructed, the first two conditions hold trivially for $i=1$. Indeed, the first statement is true as there are no vertices in $A_0$, the second is true as all the neighbours of $B_1$ are in $A_1^+$, so it has degree 0 in $A \backslash A_1^+$. 

We will now assume that $A_1$, $B_1$, $A_2$, $\ldots$, $B_{i}$, $A_{i}^+$ have been constructed. We will show how to construct $A_i$ and $B_{i+1}^+$. Assuming that $(P_2i)$ holds, we will show that $(P_3i)$ and $(P_4i)$ hold. We will first cover the main case when $|A_i^+| \geq (n-1)r$. The construction is illustrated in Figure~\ref{starconstruction}(left).

\begin{itemize} 
\item Construction of $A_i$ and $B_{i+1}^+$ when $|A_i^+|\geq (n-1)r$.

As $|A_i^+|\geq (n-1)r$, we start by applying Lemma~\ref{covering} to the 
complement of graph $G[A_i^+, B \backslash (\cup_{h=1}^{i} B_h)]$. As a result 
we can find a subset $A_i' \subseteq A_i^+$ of size $(n-1)r$ and a subset
$B_{i+1}^+ \subseteq B \backslash (\cup_{h=1}^{i} B_h)$, such that:   
\begin{itemize}
	\item[-] $B_{i+1}^+$ is $r$-co-covered by $A_i'$, i.e. each vertex of $B_{i+1}^+$ 
		has at least $r$ non-neighbours in $A_i'$.   
\end{itemize}

By Lemma~\ref{covering} every vertex of $A_i^+ \backslash A_i'$ has co-degree at most $kr$ in $B \backslash (\cup_{h=1}^{i} B_h \cup B_{i+1}^+)$. Let us denote by $A''_i \subseteq A'_i$ the set of vertices which has 
more than $kr$ non-neighbours in $B \backslash (\cup_{h=1}^i B_h \cup B_{i+1}^+)$.
By Lemma~\ref{covering} it follows that $A_i'' \neq A_i'$. Set $A_i=A_i^+ \backslash A_i''$, which is a non-empty set with every vertex having at most $kr$ non-neighbours in 
$B \backslash (\cup_{h=1}^i B_h \cup B_{i+1}^+)$. This implies that $(P_4i)$ holds. Also note that $(P_3i)$ follows from the fact that $(P_2i)$ holds and by definition that $|A^+_{i} \backslash A_i| \leq (n-1)r$.
\end{itemize}

Now we assume that bags $A_1, B_1, A_2, \ldots, B_i, A_i, B_{i+1}^+$ have been constructed. Below we describe how $B_{i+1}$ and $A_{i+1}^+$ is created in the main case when $|B_{i+1}| \geq (n-1)r$. Assuming $(P_4i)$ holds, we will show that $(P_1(i+1))$ and $(P_2(i+1))$ hold. The construction is illustrated in Figure~\ref{starconstruction}(right) and it is a bipartite complement analog of the previous construction illustrated in Figure~\ref{starconstruction}(left).

\begin{itemize} 
\item Construction of $B_{i+1}$ and $A_{i+1}^+$ when $|B_{i+1}^+|\geq (n-1)r$.

As $|B_{i+1}^+|\geq (n-1)r$, we start by applying Lemma~\ref{covering} to the graph $G[B_{i+1}^+, A \backslash (\cup_{h=1}^{i} A_h)]$. As a result 
we can find a subset $B_{i+1}' \subseteq B_{i+1}^+$ of size $(n-1)r$ and a subset
$A_{i+1}^+ \subseteq A \backslash (\cup_{h=1}^{i} A_h)$, such that:   
\begin{itemize}
	\item[-] $A_{i+1}^+$ is $r$-covered by $B_{i+1}'$, i.e. each vertex of $A_{i+1}^+$ 
		has at least $r$ neighbours in $B_{i+1}'$.   
\end{itemize}

By Lemma~\ref{covering} every vertex of $B_{i+1}^+ \backslash B_{i+1}'$ has degree at most $kr$ in $A \backslash (\cup_{h=1}^{i} A_h \cup A_{i+1}^+)$. Let us denote by $B''_i \subseteq B'_i$ the set of vertices which has 
more than $kr$ neighbours in $A \backslash (\cup_{h=1}^i A_h \cup A_{i+1}^+)$.
By Lemma~\ref{covering} it follows that $B_{i+1}'' \neq B_{i+1}'$. Set $B_{i+1}=B_{i+1}^+ \backslash B_{i+1}''$, which is a non-empty set with every vertex having at most $kr$ neighbours in 
$A \backslash (\cup_{h=1}^i A_h \cup A_{i+1}^+)$. This implies that $(P_2(i+1))$ holds. Also note that $(P_1(i+1))$ follows from the fact that $(P_4i)$ holds and by definition that $|B^+_{i+1} \backslash B_{i+1}| \leq (n-1)r$.

\end{itemize}

%\begin{figure}[H]
%	\centering
%	\includegraphics[scale=0.6]{proofstar01}
%	\caption{Construction of the bag $B_3^+$}
%	\label{fig:starsproof01}
%\end{figure}

	\begin{figure}[H]
       	
       	\begin{subfigure}[b]{0.5\textwidth}
	\begin{centering}           
			\scalebox{.65}
			{
			\begin{tikzpicture}[scale=.6,auto=left]

%%%%%%%%%%%%%%%%%%%%%%%%%%%%%%%%%%%%%%%%%%%%
	
%%%%%%%%%%%%%%%%%%%%%%%%%%%%%%%%%%%%%%%%%%%%%%%

%\tikzstyle{point}=[circle,fill=white!1,text=white,inner sep=0.4mm,draw]
%\tikzstyle{point}=[circle,fill=black,inner sep=0.3mm]

\node[point] at (-3.5, 7) {};
\node[point] at (-2.5, 7) {};
\node[point] at (-1.5, 7) {};

\node[point] at (2.3, 0) {};
\node[point] at (1.3, 0) {};
\node[point] at (0.3, 0) {};

\node[point] at (-1.6, 3.5) {};
\node[point] at (-0.6, 3.5) {};
\node[point] at (0.4, 3.5) {};

	\foreach{\h} in {7} 
	{
	%%%%%%%              Shades            %%%%%%
		
		\foreach{\start} in {7}
			{\fill[blue!15!white] (\start-4,0) -- (\start-1,\h) -- (\start+5,\h) -- (\start+2,0) -- cycle;}
		\foreach{\start} in {0, 7}
			{\fill[blue!15!white] (\start+3,0) -- (\start-1,\h) -- (\start+5,\h) -- (\start+9,0) -- cycle;}
		
		\foreach{\start} in {7}
			{
			\draw (\start+-4, 0) -- (\start-1, \h);
			\draw (\start+2, 0) -- (\start+5, \h);
			}
		\foreach{\start} in {0, 7}
			{
			\draw (\start+-1, \h) -- (\start+3, 0);
			\draw (\start+5, \h) -- (\start+9, 0);		
			}

	%%%%%%%%%%%   Ellipses     %%%%%%%%%%%%
		\foreach \start in {0,7}
		{\draw[fill=white] (\start+2,\h) ellipse (3cm and 1cm);	}
		
		\foreach \start in {0,7}
		{\draw[fill=white] (\start+6,0) ellipse (3cm and 1cm);}

		\foreach \z in {}
			{
			\draw (\z+5,-1) .. controls (\z+5.5,0)  .. (\z+5,1);
			\draw (\z+7,-1) .. controls (\z+8.5,-0.3) and (\z+6.5, 0.3) .. (\z+7,1);
			}
		\foreach \z in {}
			{
			\draw (\z+3+5,\h-1) .. controls (\z+3+5.5,\h)  .. (\z+3+5,1+\h);
			\draw (\z+3+7,\h-1) .. controls (\z+3+8.5,\h-0.3) and (\z+3+6.5, 0.3+\h) .. (\z+3+7,1+\h);
			}

%%%%%%%%%%%%%%%% square %%%%%%%%%%%%%%%%%%%

\foreach \x/\y in {10/6.7}
{
\draw (\x, \y) -- (\x+1, \y) -- (\x+1, \y+0.5) -- (\x, \y+0.5) -- cycle ;
%\draw (14,-1)--(14, 1);
}

\foreach \z/\y in {-1/-0.5}
		{	
		\draw (2, \h+1.5+\y) node[font=\fontsize{12}{12}, anchor=south] {$A_{i-1}$};
		\draw (2+7, \h+1.5+\y) node[font=\fontsize{12}{12}, anchor=south] {$A_{i}^+$};
		\draw (2+8.5, \h+0.5+\y) node[font=\fontsize{12}{12}, anchor=south] {$A_{i}'$};
		}

\foreach \z/\y in {1/0.5}
		{
		% \draw (-3+2, -1.6+\y) node[font=\fontsize{12}{12}, anchor=north] {$B_{1}$};
		\draw (-3+2+7, -1.6+\y) node[font=\fontsize{12}{12}, anchor=north] {$B_{i}$};
		\draw (-3+2+14, -1.6+\y) node[font=\fontsize{12}{12}, anchor=north] {$B_{i+1}^+$};
		% \draw (-3+2+12, -1.6+1.6+\y) node[font=\fontsize{12}{12}, anchor=north] {$B_{3}^c$};
		% \draw (-3+2+16, -1.6+1.6+\y) node[font=\fontsize{12}{12}, anchor=north] {$B_{3}^s$};
		}

		\node[w_vertex] (11) at (5+7, 0 ) {};		
		\node[w_vertex] (12) at (6.7+7, 0) {};
		\node[w_vertex] (13) at (6+7, 0) {};
		\node[w_vertex](14) at (5.5+7, 0){};
		
		\node[w_vertex](15) at (7.4+7, 0){};
		\node[w_vertex](16) at (8+7, 0){};
		\node[w_vertex](17) at (4.3+7, 0){};
		\node[w_vertex](18) at (3.5+7, 0){};

		%\node[w_vertex] (15) at (7+7.5, 0.5 ) {};
		%\node[w_vertex] (16) at (7+8.2, 0.5) {};
		\draw (11) node[anchor=south] {};
		\draw (12) node[anchor=south] {};
		\draw (13) node[anchor=south] {};

		\draw [dashed] (11) -- (7+3,7);
		\draw [dashed] (11) -- (7+3.4, 7);
		\draw [dashed] (11) -- (7+3.9, 7);
		
		\draw [dashed] (12) -- (7+3.4,7);
		\draw [dashed] (12) -- (7+3.7,7);
		\draw [dashed] (12) -- (7+3,7);	

		\draw [dashed] (13) -- (7+3.7,7);
		\draw [dashed] (13) -- (7+3.4, 7);
		\draw [dashed] (13) -- (7+3.1, 7);

		\draw [dashed] (14) -- (7+3, 7);
		\draw [dashed] (14) -- (7+3.5, 7);
		\draw [dashed] (14) -- (7+3.9, 7);

		\draw [dashed] (15) -- (7+3,7);
		\draw [dashed] (15) -- (7+3.4, 7);
		\draw [dashed] (15) -- (7+3.9, 7);
		
		\draw [dashed] (16) -- (7+3.4,7);
		\draw [dashed] (16) -- (7+3.7,7);
		\draw [dashed] (16) -- (7+3,7);	

		\draw [dashed] (17) -- (7+3.7,7);
		\draw [dashed] (17) -- (7+3.4, 7);
		\draw [dashed] (17) -- (7+3.1, 7);

		\draw [dashed] (18) -- (7+3.2, 7);
		\draw [dashed] (18) -- (7+3.5, 7);
		\draw [dashed] (18) -- (7+3.9, 7);

		%\draw [dashed] (15) -- (7-4, 7);
		%\draw [dashed] (15) -- (7+1, 7);
		%\draw [dashed] (15) -- (7+0, 7);
		
		%\draw [dashed] (16) -- (7-3, 7);
		%\draw [dashed] (16) -- (7+1.5, 7);
		%\draw [dashed] (16) -- (7+2, 7);

	}
	\end{tikzpicture}
	
			}
	\end{centering}           
     	\caption*{}
                	%\label{fig:coP6_fig1}
        	\end{subfigure}%
        ~
        	\begin{subfigure}[b]{0.5\textwidth}

\begin{centering}
				\scalebox{.65}
				{
				\begin{tikzpicture}[scale=.6,auto=left]

%%%%%%%%%%%%%%%%%%%%%%%%%%%%%%%%%%%%%%%%%%%%
	
%%%%%%%%%%%%%%%%%%%%%%%%%%%%%%%%%%%%%%%%%%%%%%%

	\foreach{\h} in {7} 
	{
	%%%%%%%              Shades            %%%%%%
		
		\foreach{\start} in {0, 7}
			{\fill[blue!15!white] (\start-4,0) -- (\start-1,\h) -- (\start+5,\h) -- (\start+2,0) -- cycle;}
		\foreach{\start} in {0}
			{\fill[blue!15!white] (\start+3,0) -- (\start-1,\h) -- (\start+5,\h) -- (\start+9,0) -- cycle;}
		
		\foreach{\start} in {0, 7}
			{
			\draw (\start+-4, 0) -- (\start-1, \h);
			\draw (\start+2, 0) -- (\start+5, \h);
			}
		\foreach{\start} in {0}
			{
			\draw (\start+-1, \h) -- (\start+3, 0);
			\draw (\start+5, \h) -- (\start+9, 0);		
			}

	%%%%%%%%%%%   Ellipses     %%%%%%%%%%%%
		\foreach \start in {0,7}
		{\draw[fill=white] (\start+2,\h) ellipse (3cm and 1cm);	}
		
		\foreach \start in {-7, 0}
		{\draw[fill=white] (\start+6,0) ellipse (3cm and 1cm);}

		\foreach \z in {}
			{
			\draw (\z+5,-1) .. controls (\z+5.5,0)  .. (\z+5,1);
			\draw (\z+7,-1) .. controls (\z+8.5,-0.3) and (\z+6.5, 0.3) .. (\z+7,1);
			}
		\foreach \z in {}
			{
			\draw (\z+3+5,\h-1) .. controls (\z+3+5.5,\h)  .. (\z+3+5,1+\h);
			\draw (\z+3+7,\h-1) .. controls (\z+3+8.5,\h-0.3) and (\z+3+6.5, 0.3+\h) .. (\z+3+7,1+\h);
			}

%%%%%%%%%%%%%%%% square %%%%%%%%%%%%%%%%%%%

\node[point] at (-3.5, 7) {};
\node[point] at (-2.5, 7) {};
\node[point] at (-1.5, 7) {};

\node[point] at (-4.7, 0) {};
\node[point] at (-5.7, 0) {};
\node[point] at (-6.7, 0) {};

\node[point] at (-5.1, 3.5) {};
\node[point] at (-4.1, 3.5) {};
\node[point] at (-3.1, 3.5) {};

\foreach \x/\y in {7/0}
{
\draw (\x, \y) -- (\x+1, \y) -- (\x+1, \y+0.5) -- (\x, \y+0.5) -- cycle ;
%\draw (14,-1)--(14, 1);
}

\foreach \z/\y in {-1/-0.5}
		{	
		\draw (2, \h+1.5+\y) node[font=\fontsize{12}{12}, anchor=south] {$A_{i}$};
		\draw (2+7, \h+1.5+\y) node[font=\fontsize{12}{12}, anchor=south] {$A_{i+1}^+$};
	%	\draw (2+8.5, \h+0.5+\y) node[font=\fontsize{12}{12}, anchor=south] {$A_{2}'$};
		}

\foreach \z/\y in {1/0.5}
		{
		\draw (-3+2, -1.6+\y) node[font=\fontsize{12}{12}, anchor=north] {$B_{i}$};
		\draw (-3+2+7, -1.6+\y) node[font=\fontsize{12}{12}, anchor=north] {$B_{i+1}^+$};
		\draw (-1.5+2+7, -0.4+\y) node[font=\fontsize{12}{12}, anchor=north] {$B_{i+1}'$};
		% \draw (-3+2+14, -1.6+\y) node[font=\fontsize{12}{12}, anchor=north] {$B_{3}^+$};
		% \draw (-3+2+12, -1.6+1.6+\y) node[font=\fontsize{12}{12}, anchor=north] {$B_{3}^c$};
		% \draw (-3+2+16, -1.6+1.6+\y) node[font=\fontsize{12}{12}, anchor=north] {$B_{3}^s$};
		}

		\node[w_vertex] (11) at (5+4, 7 ) {};		
		\node[w_vertex] (12) at (6.7+4, 7) {};
		\node[w_vertex] (13) at (6+4, 7) {};
		\node[w_vertex](14) at (5.5+4, 7){};

		\node[w_vertex] (15) at (4.3+4, 7 ) {};		
		\node[w_vertex] (16) at (3.6+4, 7) {};
		\node[w_vertex] (17) at (2.8+4, 7) {};
		\node[w_vertex](18) at (7.3+4, 7){};

		%\node[w_vertex] (15) at (7+7.5, 0.5 ) {};
		%\node[w_vertex] (16) at (7+8.2, 0.5) {};
		\draw (11) node[anchor=south] {};
		\draw (12) node[anchor=south] {};
		\draw (13) node[anchor=south] {};

		\draw  (11) -- (4+3,0.2);
		\draw  (11) -- (4+3.4, 0.2);
		\draw  (11) -- (4+3.9, 0.2);
		
		\draw (12) -- (4+3.4,0.2);
		\draw (12) -- (4+3.7,0.2);
		\draw (12) -- (4+3,0.2);	

		\draw (13) -- (4+3.7, 0.2);
		\draw (13) -- (4+3.4, 0.2);
		\draw (13) -- (4+3.1, 0.2);

		\draw (14) -- (4+3, 0.2);
		\draw (14) -- (4+3.5, 0.2);
		\draw (14) -- (4+3.9, 0.2);

		\draw  (15) -- (4+3,0.2);
		\draw  (15) -- (4+3.4, 0.2);
		\draw  (15) -- (4+3.9, 0.2);
		
		\draw (16) -- (4+3.4,0.2);
		\draw (16) -- (4+3.7,0.2);
		\draw (16) -- (4+3,0.2);	

		\draw (17) -- (4+3.7, 0.2);
		\draw (17) -- (4+3.4, 0.2);
		\draw (17) -- (4+3.1, 0.2);

		\draw (18) -- (4+3.2, 0.2);
		\draw (18) -- (4+3.5, 0.2);
		\draw (18) -- (4+3.9, 0.2);

		%\draw  [dashed] (15) -- (7-4, 7);
		%\draw [dashed] (15) -- (7+1, 7);
		%\draw [dashed] (15) -- (7+0, 7);
		
		%\draw [dashed] (16) -- (7-3, 7);
		%\draw [dashed] (16) -- (7+1.5, 7);
		%\draw [dashed] (16) -- (7+2, 7);

	}
	\end{tikzpicture}
	
				}
\end{centering}
                	\caption*{}
                	%\label{fig:coP6_fig2}
        	\end{subfigure}

       	\caption{Construction of the bag $B_{i+1}^+$ (left) and $A_{i+1}^+$ (right)}\label{starconstruction}
	\end{figure}
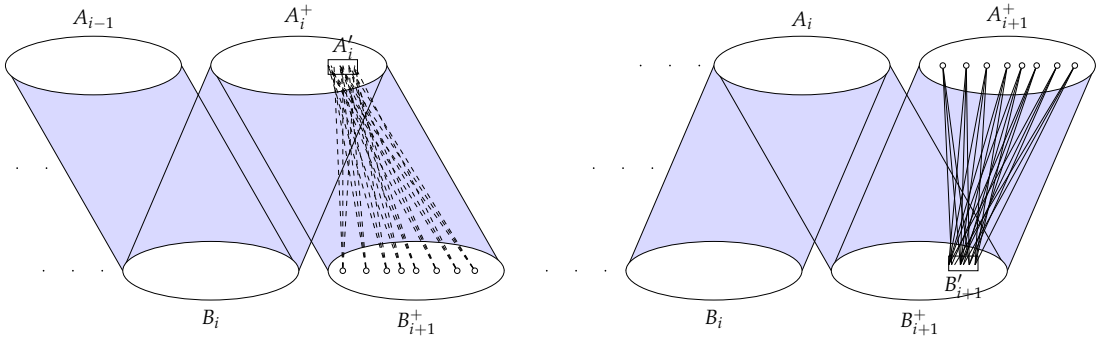

\iffalse
\begin{figure}[H]
  \centering
  \begin{minipage}[b]{0.45\textwidth}
    \includegraphics[width=\textwidth]{proofstar01a}
 %   \caption{Flower one.}
  \end{minipage}
  \hfill
  \begin{minipage}[b]{0.45\textwidth}
    \includegraphics[width=\textwidth]{proofstar01b}
%    \caption{Flower two.}
  \end{minipage}

\caption{Construction of the bag $B_{i+1}^+$ (left) and $A_{i+1}^+$ (right)}
\end{figure}
\fi

In the particular cases where the newly created bag $A_i^+$ (or $B_{i+1}^+$) have size less than $(n-1)r$, we set $A_i$ (or $B_i$) to be an empty set and start the process afresh from a newly selected vertex of the minimum degree. We will cover these two cases $|A_i^+|<(n-1)r$ and $|B_i^+|<(n-1)r$ separately. 

Suppose first that $A_1, B_1, A_2, \ldots, B_i, A_i^+$ have been constructed and $|A_i^+|<(n-1)r$. We will show how to construct bags $A_i$, $B_{i+1}^+$, $B_{i+1}$ and $A_{i+1}^+$. Assuming that $(P_2i)$ holds, we will show that $(P_3i)$, $(P_4i)$, $(P_1(i+1))$ and $(P_2(i+1))$ hold. 

\begin{itemize}
\item Construction of $A_i$, $B_{i+1}^+$, $B_{i+1}$ and $A_{i+1}^+$ or termination of the procedure when $|A_i^+|<(n-1)r$.

In this case, we set $A_i= \emptyset$. If $B \backslash \cup_{h=1}^i B_h=\emptyset$, we let $B_{i+1}=\emptyset$ and $A_{i+1}=A \backslash \cup_{h=1}^i A_h$, set $z=i+1$ and terminate the procedure. Else, if $B \backslash \cup_{h=1}^i B_h \neq \emptyset$ we set $B_{i+1}^+=B_{i+1}=\{v\}$, where $v$ is a vertex with minimal degree in $B \backslash \cup_{h=1}^i B_h$. Then set $A_{i+1}^+= N(B_{i+1}) \backslash \cup_{h=1}^i A_i$. Note that $(P_4i)$ holds trivially because $A_i = \emptyset$ and $(P_3i)$ holds because we assume that $(P_2i)$ holds and $|A_i^+ \backslash A_i| < (n-1)r$. Further, $(P_1(i+1))$ holds because $(P_4i)$ holds and $B_{i+1}^+ =B_{i+1}$ and $(P_2(i+1))$ holds because any vertex in $B_i$ has degree $0$ in $A \backslash (A_1 \cup A_2 \cup \ldots \cup A_i^+)$. 

\end{itemize}

Finally suppose that $A_1, B_1, A_2, \ldots, B_{i+1}^+$ have been constructed and $|B_{i+1}^+|<(n-1)r$. We will show how to construct $B_{i+1}, A_{i+1}^+, A_{i+1}, B_{i+2}^+, B_{i+2}$ and $A_{i+2}^+$. Assuming that $(P_4i)$ holds, we will show that $(P_1(i+1)), (P_2(i+1)), (P_3(i+1)), (P_4(i+1)), (P_1(i+2)), (P_2(i+2))$ hold.

\begin{itemize}
\item Construction of $B_{i+1}, A_{i+1}^+, A_{i+1}, B_{i+2}^+, B_{i+2}$ and $A_{i+2}^+$ or termination of the procedure when $|B_{i+1}^+| < (n-1)r$.

We set $B_{i+1} = \emptyset$, $A_{i+1}^+ = A_{i+1}=\emptyset$.
If $B \backslash \cup_{h=1}^{i+1} B_h=\emptyset$, we let $B_{i+2}=\emptyset$ and $A_{i+2}=A \backslash \cup_{h=1}^{i+1} A_h$, set $z=i+2$ and terminate the procedure. Else, if $B \backslash \cup_{h=1}^{i+1} B_h \neq \emptyset$ we set $B_{i+2}^+=B_{i+2}=\{v\}$, where $v$ is the vertex with the smallest degree in $B \backslash B_1 \cup \ldots \cup B_{i+1}$. We also set $A_{i+2}^+=N(B_{i+1}) \backslash A_1 \cup \ldots \cup A_{i+1}$. One can easily verify that assuming $(P_4i)$ holds, one can deduce that $(P_1(i+1)), (P_2(i+1)), (P_3(i+1)), (P_4(i+1)), (P_1(i+2)), (P_2(i+2))$ all hold.  
\end{itemize}

Hence, by induction it follows that $(P_1i), (P_2i), (P_3i), (P_4i)$ hold for any $i$. This means that the partition satisfies the ''forward edges'' conditions of (*) and (**): $A_i$ is $d$-joined to $\cup_{h=1}^z B_h$ and $B_i$ is $d$-co-joined to $\cup_{h=i+1}^z A_h$ for any $i$. However, the ''backward edges'' conditions of (*) and (**) may not be satisfied. Indeed, there might be some vertices in $A_i$ having more than $d$ neighbours in $\cup_{h=1}^{i-1} B_h$ or some vertices in $B_i$ having more than $d$ non-neighbours in $\cup_{h=1}^{i-2} A_h$. We will show how to remove these vertices from the partition.  

\begin{definition}
Let $A=A_1 \cup A_2 \cup \ldots \cup A_z$, $B=B_1 \cup B_2 \cup \ldots \cup B_z$ be the partition obtained by the procedure described above. Define $M$ to be the set of all vertices $v \in A$, such that $v \in A_i$ for some $i$ and $v$ has more than $d$ neighbours in $\cup_{h=1}^{i-1} B_h$. Define $N$ to be the set of all vertices $v \in B$, such that $v \in B_i$ for some $i$ and $v$ has $d$ non-neighbours in $\cup_{h=1}^{i-2} A_h$.      
\end{definition}

The following Lemma is an immediate consequence of this definition and the discussion above:

\begin{lemma}\label{unmarkedvertices}
The partition $A \backslash M = (A_1 \backslash M) \cup \ldots \cup (A_z \backslash M)$, $B \backslash N = (B_1 \backslash N) \cup \ldots \cup (B_z \backslash N)$ is a $d$-template of $G[A \backslash M, B \backslash N]$.
\end{lemma}

The following lemma shows that vertices in $M$ can be partitioned into bounded number of parts such that a graph induced by any part and set B is  $2\Lambda_{2d}$-free. 

\begin{lemma} \label{markedvertices}
Vertices $M$ can be partitioned into $c=3nkd^2$ bags $M=M_1 \cup M_2 \cup \ldots \cup M_c$ such that the graph $G[M_i, B]$ is $2\Lambda_{2d}$-free for all $i$. Similarly, vertices $N$ can be partitioned into $c=3nkd^2$ bags $N=N_1 \cup N_2 \cup \ldots \cup N_c$ such that the graph $G[A, N_i]$ is $2\up_{2d}$-free for all $i$.
\end{lemma}

\begin{proof}

Let us start with a vertex $v \in M$. We will define a trace of a vertex $v$. First recall that $v \in M$ means that $v \in A_i$ for some $i$ such that $v$ has more than $d$ neighbours in $\cup_{h=1}^{i-1} B_h$. Now we find the smallest value $j$, such that $v$ has more than $d$ neighbours in $\cup_{h=1}^{j-1} B_h$. We define the trace of $v$ to be the interval $tr(v)=\{j, j+1, \ldots, i\}$. In the next two paragraphs we will explore several properties of the trace.  

Consider the set $T_l=\{v \in M: l \in tr(v)\}$. We claim that $|T_l| \leq nkd^2$ for all $l$. Consider the bipartite subgraph of $G$ induced by $\cup_{h=1}^{l-1} B_h$ and $T_l$. Note that $T_l \subseteq \cup_{h=l}^z A_h$. By $(P_3i)$ we have that each vertex of $\cup_{h=1}^{l-1} B_h$ has at most $d$ neighbours in $T_l \subseteq \cup_{h=l}^z A_h$. As the graph $G$ is $n \Lambda_k$ free, Lemma~\ref{boundeddegree} applies, so there are at most $nkd^2$ vertices in $T_l$ that have degree at least $k$ in $\cup_{h=1}^{l-1} B_h$. But by definition of $T_l$, all vertices of $T_l$ have degree at least $d \geq k$ in $\cup_{h=1}^{l-1} B_h$. Hence, we conclude that $|T_l| \leq nkd^2$.

Now, consider two vertices $v, v' \in M$ such that $tr(v)=\{j, j+1, \ldots, i \}$ and $tr(v')=\{j', j'+1, \ldots, i'\}$ with $j'>i$ such that $|j'-i| \geq 3$. Then, we claim that $|N(v') \backslash N(v)| \leq 2d$. Indeed, $N(v)$ has at most $d$ non-neighbours in $B_{i+2} \cup B_{i+3}\cup \ldots \cup B_z$ by $(P_1 i)$. Also, as $j'=min\{tr(v')\}$, it follows that $N(v')$ has at most $d$ neighbours in $B_1 \cup B_2 \cup \ldots \cup B_{j'-2}$. As $j' \geq i+3$, $N(v')$ has at most $d$ neighbours in $B_1 \cup B_2 \cup \ldots \cup B_{i+1}$. Hence the claim follows.  

We will now show that we can partition $M$ into $c=nkd^2$ subsets $M=M_1' \cup M_2' \cup \ldots \cup M_c'$ such that for any $i$ and any $v, v' \in M'_i$, we have $tr(v) \cap tr(v') = \emptyset$. One way of achieving it, is by ordering the vertices of $M=\{v_1, v_2, \ldots, v_m\}$ such that if $i<j$, then $min\{tr(v_i)\} \leq min \{tr(v_j)\}$. Having achieved this order, we set $M'_i = \emptyset$ for $1 \leq i \leq c$ and follow the following procedure. We place $v_1$ into $M_1'$, i.e. we set $M_1':=\{v_1\}$. Suppose now we placed $v_1, v_2 \ldots v_{i-1}$ for some $i \geq 2$. Then we place $v_i$ in some $M_j'$ such that $tr(v_i) \cap tr(v) = \emptyset$ for every $v \in M_j'$. We claim that we can always find such a set $M_j'$. Indeed, suppose for contradiction that we cannnot find such set $M_j'$, i.e. suppose that for every $j$ there is a vertex $w_j \in M_j'$ such that $tr(w_j) \cap tr(v_i) \neq \emptyset$. By the assumed ordering of vertices we have $min\{tr(w_j)\} \leq min\{tr (v_i)\}$, and hence as $tr(w_j) \cap tr(v_i) \neq \emptyset$ we conclude that $min\{tr(v_i)\}$ is contained in $tr(w_j)$ for all $j=1, 2, 3 \ldots, c$. Thus we have at least $c+1$ vertices, namely $\{v_1, w_1, w_2, \ldots, w_c\}$ whose trace contains the number $l=min\{tr(v_i)\}$. This contradicts the fact that $|T_l|\leq c$ that we proved earlier. Thus this procedure will partition all the vertices into $c$ sets such that traces of the vertices in each set are all disjoint. 

Finally, we can subpartition each set $M_i'$ into three to ensure that there are at least two natural numbers in the interval between any two traces. One way of doing this is by listing the vertices of $M_i'=\{u_1, u_2, \ldots, u_l\}$ such that $\min\{tr(u_1)\} \leq \min\{tr(u_2)\} \leq \ldots \leq \min\{tr(u_l)\}$ and then setting $M_{3i-2}=\{u_1, u_4, u_7 \ldots\}$, $M_{3i-1}=\{u_2, u_5, u_8 \ldots \}$ and $M_{3i}=\{u_3, u_6, u_9, \ldots \}$. In this way we obtain the sets $M_1, M_2, \ldots, M_{3c}$ which satisfy the conditions of the lemma.
\end{proof}

\subsection{Conditions (***) and (****)}

Now we will refine the partition obtained in the previous section $A = A_1 \cup A_2 \cup \ldots \cup A_z$, $B = B_1 \cup B_2 \cup \ldots \cup B_z$ so that the resulting refinement satisfies the conditions (***) and (****). Note that the partition $A = A_1 \cup A_2 \cup \ldots \cup A_z$, $B = B_1 \cup B_2 \cup \ldots \cup B_z$ is not a $d$-template, but, as described in the previous section, could be made into one by removing vertex sets $M$ and $N$. Therefore, having obtained the refinement satisfying conditions (***) and (****), we will later restrict it to the partition $A \backslash M = A_1 \backslash M \cup A_2 \backslash M \cup \ldots \cup A_z \backslash M$, $B \backslash N = B_1 \backslash N \cup B_2 \backslash N \cup \ldots \cup B_z \backslash N$ to deduce a refined $d$-template for $G[A \backslash M, B \backslash N]$. 

Recall that the partition procedure starts with $B_1^+=B_1=\{v\}$ followed by $A_1^+=N(B_1)$. The process continues creating new bags depending on the previous one in alternating order until we set some bag to be $\emptyset$. Then we pick again $B_i^+=B_i=\{w\}$ and $A_i^+=N(B_i)$. We will denote the indices of all such starting points by the set $I$, i.e. \[I=\{i : |B_i^+|=|B_i|=1\}.\] The following Lemmas~\ref{initialcase} and~\ref{generalcase} provide a subpartition such that each bipartite graph $G[A_{i}, B_{i+1}]$ satisfies the conditions (***) and (****). The lemmas consider separately the two cases depending on whether $i \in I$ or $i \notin I$. The proofs for the two lemmas are accompanied with the Figures~\ref{condition3a} and~\ref{condition3b}. We note that in the proofs we will be using subsets $A_i, A_i^+, A_i'$ and $B_i, B_i^+, B_i'$ as defined in Section~\ref{partitionprocedure}. All the other vertex subsets, for instance, $B_{i+1}^c$ and $B_{i+1}^s$ in Lemma~\ref{initialcase} or $B_{i+1}^{sc}, B_{i+1}^{ss}, A_{i-1}^{cc}$ in Lemma~\ref{generalcase} are locally defined and with the exception of $B_{i+1}^{ss}$ are not used in the subsequent results. We recall from the previous section that $G$ is defined to be $(n'\up_k, n''\Lambda_k, \overline{m'\up_k}^{bip}, \overline{m''\Lambda_k}^{bip})$-free graph with $n',n'',m', m'' \geq 3$ and $n=\max(n', n'', m',m'')$. In this section we will set $r=kn$ in our general $r$-covering procedure described in the previous section. Thus, in particular $d=(n-1)r+kr=(n-1)kn+k(kn)$ throughout this section. We start with the initial case when $i \in I$. 

\begin{lemma}\label{initialcase}
 For $i \in I$, we can split $B_{i+1}$ into $p \leq \pi_1 = \binom{nk}{k}+nk^5$ parts $B_{i+1}=B_{(i+1)1} \cup B_{(i+1)2} \cup \ldots B_{(i+1)p}$ such that $G[A_i, B_{(i+1)l}]$ is $Free(\overline{(m'-1)\up_k}^{bip})$ for all $1 \leq l \leq p$.
\end{lemma}

%\begin{proof}[Proof of Lemma~\ref{initialcase}] 

\begin{proof}
As $i \in I$, we have $B_i=\{b\}$ for some vertex $b$ with the minimal degree in $B \backslash (B_1 \cup B_2 \cup \ldots \cup B_{i-1})$ and $A_i^+=N(b)$. Let us now apply Lemma~\ref{covering} to the graph $G[A \backslash (A_1 \cup A_2 \cup \ldots \cup A_{i-1} \cup A_i^+, B_{i+1}^+]$. We can find a subset $A' \subseteq A \backslash (A_1 \cup A_2 \cup \ldots \cup A_{i-1} \cup A_i^+)$ of size $|A'| = (n-1)k$ and a partition of $B_{i+1}^+ = B_{i+1}^{c} \cup B_{i+1}^{s}$ such that: 
\begin{itemize}
\item $B_{i+1}^{c}$ is a set $k$-covered by $A'$. 
\item $A \backslash (A_1 \cup A_2 \cup \ldots \cup A_{i-1} \cup A_i^+ \cup A')$ has degree at most $k^2$ in $B_{i+1}^{s}$.
\end{itemize}

As each vertex in $B_{i+1}^{c}$ is $k$-covered by some vertex in $A'$, we can partition $B_{i+1}^{c}$ into $p' \leq \binom{|A'|}{k}$ subsets 
$B_{(i+1)1}, B_{(i+1)2}, \ldots, B_{(i+1)p'}$ such that each subset $B_{(i+1)j}$ is joined to a certain 
subset $S_j \subseteq A'$, of size $|S_j|=k$.
Observe that $B_{(i+1)j}$ is joined to $S_j$ and $A_i$ is joined to $\{b\}$ and the vertex set $\{b\} \cup S_j$ induces a co-star.  
Thus, we conclude that $G[B_{(i+1)j}, A_i]$ is in $Free(\overline{(m'-1)\up_k}^{bip})$ for
all $1 \leq j \leq p'$.

We will now bound the size of $B_{i+1}^{s}$. First of all, as the degree of each vertex in $A \backslash (A_1 \cup A_2 \cup \ldots \cup A_{i-1} \cup A_i^+ \cup A')$ in $B_{i+1}^s$ is bounded by $k^2$, we can apply Lemma~\ref{boundeddegree} to the graph $G[A \backslash (A_1 \cup A_2 \cup \ldots \cup A_{i-1} \cup A_i^+ \cup A'), B_{i+1}^{s}]$. From the Lemma we obtain that there are at most $(nk)(k^2)^2=nk^5$ vertices in $B_{i+1}^{s}$ of degree at least $k$ in $A \backslash (A_1 \cup A_2 \cup \ldots \cup A_{i-1} \cup A_i^+ \cup A')$. We will show that this is indeed a bound on $B_{i+1}$ by showing that all the vertices in $B_{i+1}^s$ have degree at least $k$ in $A \backslash (A_1 \cup A_2 \cup \ldots \cup A_{i-1} \cup A_i^+ \cup A')$. As $B_{i+1}$ is $r$-co-covered by a subset of $A_i^+$, each vertex  has at least $r$ non-neighbours in $A_i^+$. Hence by minimality of the degree of $b$, we have that every vertex of $B_{i+1}$ has at least $r$ neighbours in $A \backslash (A_1 \cup A_2 \cup \ldots \cup A_{i-1} \cup A_i^+)$. Thus, every vertex of $B_{i+1}$ has at least $r - |A'| = r-(n-1)k \geq k$ neighbours in $A \backslash (A_1 \cup A_2 \cup \ldots \cup A_{i-1} \cup A_i^+ \cup A')$ proving our claim. Thus, we obtain $|B_{i+1}^{s}| \leq nk^5$.

We to complete our partition of $B_{i+1}$ by partitioning $B^{s}$ into single vertex sets. i.e. let $B_{(i+1)(p'+1)}, B_{(i+1)(p'+2)}, \ldots, B_{(i+1)p}$ be the sets
containing a single vertex of $B^{s}$. Then clearly, $G[A_1, B_{(i+1)j}]$ is
$Free(\overline{(m'-1)\up_k})$ for all $p'+1 \leq i \leq p$.  Hence we have a partition of $B_{i+1}$ into $p \leq \binom{|A'|}{k}+|B^{s}| \leq \binom{nk}{k}+nk^5$ parts, such that between each part and the set
$A_i$ the graph is $Free(\overline{(m'-1)\up_k})$. 

\end{proof}

\begin{figure}[H]
	\centering
	\begin{tikzpicture}[scale=.6,auto=left]

%%%%%%%%%%%%%%%%%%%%%%%%%%%%%%%%%%%%%%%%%%%%
		%\draw[step=1cm,gray,very thin] (0,0) grid (13,5);
		%\foreach \x in {0,1,2,3,4,5,6,7,8,9,10,11,12,13}
    		%\draw (\x cm,1pt) -- (\x cm,-1pt) node[anchor=north] {$\x$};
		%\foreach \y in {0,1,2,3,4}
    		%\draw (1pt,\y cm) -- (-1pt,\y cm) node[anchor=east] {$\y$};
		\foreach \start in {}
		{		
		\node[w_vertex] (1) at (\start+0, 2) {}; 	
		\node[w_vertex] (2) at (\start+1, 2) { };
		\node[w_vertex] (3) at (\start+2, 2) { };
		\node[w_vertex] (4) at (\start+3, 2) { };
		\node[w_vertex] (5) at (\start+0,0) { };
		\node[w_vertex] (6) at (\start+1,0) { };
		\node[w_vertex] (7) at (\start+2,0) { };
		\node[w_vertex] (8) at (\start+3,0) { };				
		\foreach \from/\to in {1/5,2/6,3/7,4/8}
	    	\draw (\from) -- (\to);
		%\foreach \from/\to in {5/6,6/7,7/8}
	    	%\draw (\from) -- (\to);		
		}
%%%%%%%%%%%%%%%%%%%%%%%%%%%%%%%%%%%%%%%%%%%%%%%

	\foreach{\h} in {5} 
	{
	%%%%%%%              Ellipses            %%%%%%

		\foreach \start in {11}
		{
 		\draw (\start+2,\h) ellipse (7cm and 1cm);
		}

		\foreach \start in {0}
		{
 		\draw (\start+6,0) ellipse (3cm and 1cm);
		}

%		\draw [dashed] (3, 0) -- (-1, \h);
%		\draw [dashed] (9, 0) -- (5, \h);

		\foreach \x/\y in {8/4.7}
		{
		\draw (\x, \y) -- (\x+1.5, \y) -- (\x+1.5, \y+0.5) -- (\x, \y+0.5) -- cycle ;
		}

		%%%%%%           vertex b_1            %%%%%
		\node[w_vertex](1) at (-3, 0) {};
		%\draw (1) node[anchor=north] {$B_i$};
		\draw (-3, -0.1) node[anchor=north] {$B_i$};
		\node[point] (2) at (-1, \h) {};
		\node[point] (3) at (4.6, \h-0.5) {};
		\draw (1) -- (2);
		\draw (1) -- (3);

		\fill[blue!15!white] (6.5,0) -- (8.5,0) -- (8.7, 5) -- (8.1, 5) -- cycle;
		\fill[blue!15!white] (-3,0) -- (-1,\h) -- (4.6,\h-0.5) -- cycle;
	
		\foreach \start in {0}
		{
 		\draw[fill=white] (\start+2,\h) ellipse (3cm and 1cm);
		}

		\draw[fill=white] (8.4,5) ellipse (0.3cm and 0.1cm);
		\draw[fill=white] (7.5, 0) ellipse (1cm and 0.5cm);
		\draw (6.5, 0) -- (8.1, 5);	
		\draw (8.5, 0) -- (8.7,5);
		\draw[dashed] (1) -- (8.1, 5);
		\draw[dashed] (1) -- (8.7, 5);
		
 		\draw  (-0.8, \h-0.35) -- (6.5, 0);
		\draw  (5, \h) -- (8.5, 0);

		\draw (2, \h+0.9) node[anchor=south] {$A_{i}^+$};
		\draw (6+7, \h+0.9) node[anchor=south] {$A \backslash (A_1 \cup A_2 \cup \ldots \cup A_{i-1} \cup A_i^+)$};

		\draw (4+2, -1) node[anchor=north] {$B_{i+1}^+$};

		\draw (3.8, 0.3) node[anchor=north] {$B_{i+1}^s$};
		\draw (4+1.2, 0.3) node[anchor=north] {$B_{i+1}^c$};
	%	\draw (4+4, 0.3) node[anchor=north] {$Y_{23}$};
		
		\draw (7.5, 0.48) node[anchor=north] {$B_{(i+1)j}$};
		\draw (8.5, 6.0) node[anchor=north] {$S_j$};
		\draw (9.9, 5.5) node[anchor=north] {$A'$};

	%%%%%%      vertices a_1 --- a_3     %%%%%%

	%	\node[w_vertex] (4) at (7+2, \h ) {};		
	%	\node[w_vertex] (5) at (7+2.7, \h) {};
	%	\node[w_vertex] (6) at (7+4, \h) {};
	%	\draw (4) node[anchor=south] {$a_{11}$};
	%	\draw (5) node[anchor=south] {$a_{12}$};
	%	\draw (6) node[anchor=south] {$a_{13}$};

	%	\draw  (4) -- (5,1);
	%	\draw  (4) -- (3,0);
	%	\draw  (5) -- (5,1);
	%	\draw  (5) -- (7,1);
	%	\draw  (6) -- (7,1);
	%	\draw  (6) -- (8.7,0.5);

	%	\draw (5,-1) .. controls (5.5,0)  .. (5,1);
	%	\draw (7,-1) .. controls (8.5,-0.3) and (6.5, 0.3) .. (7,1);
	
		\draw (4.5, -0.85) -- (4.5,0.85);

	}
	\end{tikzpicture}
	\caption{Proof of Lemma~\ref{initialcase}}
	\label{condition3a}
\end{figure}
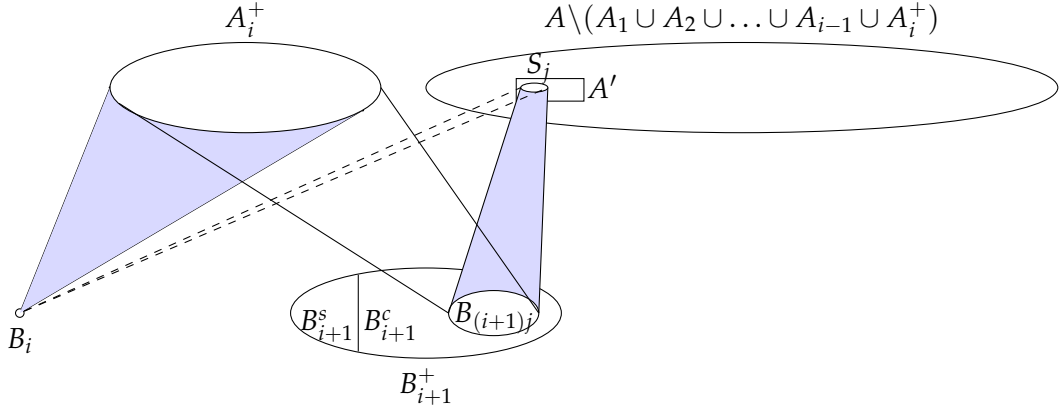

For the partition in general case $i \notin I$, we will use the following observation. 

\begin{observation}\label{observationvertex}
There is a subset $A_{i-1}^{cc} \subseteq A_{i-1}$ of size $n-1$ that co-covers $B_i^+$. Also each vertex of $A_{i-1}^{cc}$ has at most $kr$ non-neighbours in $B_{i+1} \backslash B_i^+$.  Furthermore, for each subset $F \subseteq B_i^+$ of size $|F|=r$ there is a vertex $v \in A_{i-1}^{cc}$ having $k$ non-neighbours in $F$.
\end{observation}

\begin{proof}
The first two conditions statements follow easily from Lemma~\ref{covering}. Indeed, by Lemma~\ref{covering} there is a subset $A_{i-1}^{cc} \subseteq A_{i-1}' \subseteq A_{i-1}^+$ that covers $B_i^+$. Further, Lemma~\ref{covering} proves that $A_{i-1}^{cc}$ has co-degree at most $kr$ in $B \backslash (B_1 \cup B_2 \cup \ldots \cup B_{i-1} \cup B_i^+)$. From this property it follows that, in particular, $A_{i-1}^{cc}$ has co-degree at most $kr$ in $B_{i+1} \backslash B_i^+$ and that $A_{i-1}^{cc} \subseteq A_{i-1}$ (by construction only vertices with co-degree more than $kr$ in $B \backslash (B_1 \cup B_2 \cup \ldots \cup B_{i-1} \cup B_i^+)$ are removed from $A_{i-1}^+$ to form $A_{i-1}$).    

Now, this subset $A_{i-1}^{cc}$ is of size $n-1$ and co-covers $B_i^+$ and hence co-covers $F \subseteq B_i^+$. Therefore, by pigeonhole principle, there will be a vertex in $v \in A_{i-1}'$ which has at least 
$\frac{|F|}{|A_{i-1}'|} \geq \frac{r}{n-1} \geq k$ non-neighbours in $F$.  
\end{proof}

Now we are ready to prove the partition lemma for the general case $i \notin I$.

\begin{lemma}\label{generalcase}
 For $i \notin I$, we can split $A_i$ into $p \leq \pi_2 = \binom{(n-1)r}{r}$ parts $A_i = A_{i1} \cup A_{i2} \cup \ldots \cup A_{ip}$ and find a subset $B_{i+1}^{ss} \subseteq B_{i+1}$ of size at most $\pi_3 = (n-1)r(k+1)$ such that $G[A_{il}, B_{i+1}\backslash B_{i+1}^{ss}]$ is $Free(\overline{(m''-1)\Lambda_k}^{bip})$ for all $1 \leq l \leq p$.  
\end{lemma}

\begin{proof}%[Proof of Lemma~\ref{generalcase}]
For every $F \in B_i'$ of size $|F|=r$, consider the set $A_{i, F}=\{v \in A_{i} | \{v\} $ is joined to $ F\}$. As $A_{i}$ is $r$-covered by $B_i'$, these subsets $A_{i,F}$ cover $A_{i}$. From this cover we can obtain a partition $A_{i} =A_{i1} \cup A_{i2} \cup \ldots \cup A_{ip}$ with $p \leq \binom{|B'|}{r}=\binom{(n-1)r}{r}$ by removing each vertex appearing several subsets of the cover from all but one subset. This partition will satisfy the condition that for each $1\leq l \leq p$, we have $A_{il} \subseteq A_{i,F}$ for some $F$, which means that $A_{il}$ is joined to some $F \subseteq B_i'$ of size $|F|=r$.  

Let $A_{i-1}^{cc}$ be the subset of $A_{i-1}$ provided by Observation~\ref{observationvertex}. Denote the non-neighbourhood of $A_{i-1}^{cc}$ in $B_{i+1} \backslash B_i^+$ by $B_{i+1}^{sc}$. By Observation~\ref{observationvertex}, each vertex in $A_{i-1}^{cc}$ has at most $kr$ non-neighbours in $B_{i+1}$,  hence we have $|B_{i+1}^{sc}| \leq |A_{i-1}^{cc}| \cdot kr = (n-1)kr$. Further, for each subset $F \subseteq B_i'$ of size $|F|=r$, let $v_F \in A_{i-1}^{cc}$ be a vertex which has $k$ non-neighbours in $F$ which exists by Observation~\ref{observationvertex}.

Let $B_{i+1}^{ss}$ be the union of subsets $B_{i+1}^{sc}$ and $B_i^+ \cap B_{i+1}$. As $|B_{i+1}^{sc}| \leq (n-1)kr$ and $|B_i^+ \cap B_{i+1}| \leq (n-1)r$, we conclude that $|B_{i+1}^{ss}| \leq (n-1)r(k+1)$. Note that $B_{i+1} \backslash B_{i+1}^{ss} = B_{i+1} \backslash (B_{i+1}^{sc} \cup B_i^+)$. It is now easy to verify that for any $l$, $G[A_{il}, B_{i+1} \backslash B_{i+1}^{ss}]$ does not contain $\overline{(m''-1)\Lambda_k}^{bip}$. Indeed, by definition of $A_{il}$, there is a subset $F \subseteq B_i'$ of size $|F|=r$ which is joined to $A_{il}$. The vertex $v_F$ is joined to $B_{i+1} \backslash B_{i+1}^{ss}$. Hence a co-star $\overline{\Lambda_k}^{bip}$ consisting of $v_F$ and its non-neighbours in $F$ has the centre joined to $B_{i+1} \backslash B_{i+1}^{ss}$ and the leaves joined to $A_{il} \subseteq A_{i,F}$. As $G$ is $\overline{m''\Lambda_k}^{bip}$-free, we deduce that $G[A_{il}, B_{i+1} \backslash B_{i+1}^{ss}]$ is $\overline{(m''-1)\Lambda_k}^{bip}$-free as required.
\end{proof}

\begin{figure}[H]
	\centering
	\begin{tikzpicture}[scale=.6,auto=left]

%%%%%%%%%%%%%%%%%%%%%%%%%%%%%%%%%%%%%%%%%%%%
		%\draw[step=1cm,gray,very thin] (0,0) grid (13,5);
		%\foreach \x in {0,1,2,3,4,5,6,7,8,9,10,11,12,13}
    		%\draw (\x cm,1pt) -- (\x cm,-1pt) node[anchor=north] {$\x$};
		%\foreach \y in {0,1,2,3,4}
    		%\draw (1pt,\y cm) -- (-1pt,\y cm) node[anchor=east] {$\y$};
		\foreach \start in {}
		{		
		\node[w_vertex] (1) at (\start+0, 2) {}; 	
		\node[w_vertex] (2) at (\start+1, 2) { };
		\node[w_vertex] (3) at (\start+2, 2) { };
		\node[w_vertex] (4) at (\start+3, 2) { };
		\node[w_vertex] (5) at (\start+0,0) { };
		\node[w_vertex] (6) at (\start+1,0) { };
		\node[w_vertex] (7) at (\start+2,0) { };
		\node[w_vertex] (8) at (\start+3,0) { };				
		\foreach \from/\to in {1/5,2/6,3/7,4/8}
	    	\draw (\from) -- (\to);
		%\foreach \from/\to in {5/6,6/7,7/8}
	    	%\draw (\from) -- (\to);		
		}
%%%%%%%%%%%%%%%%%%%%%%%%%%%%%%%%%%%%%%%%%%%%%%%

	\foreach{\h} in {5} 
	{
	%%%%%%%              Ellipses            %%%%%%

		\foreach \start in {17,24}
		{
 		\draw[fill=white] (\start+2,\h) ellipse (3cm and 1cm);
		}

		\foreach \start in {17}
		{
 		\draw[fill=white] (\start+6,0) ellipse (3cm and 1cm);
		}

	\fill[blue!15!white] (20.7, 5) -- (31.8, 0.8) -- (27, 0) -- cycle;
           \node[w_vertex](1) at (20.7, 5) {};
	\draw (1) -- (27,0);
 	\draw (1) -- (31.8, 0.8);

	\fill[blue!15!white] (24,5) -- (24.2, 0.2) -- (24.8, 0.2) -- (25.8, 5) -- cycle;
		\draw (24,5) -- (24.2, 0.2);
		\draw(25.8, 5) -- (24.8, 0.2);

            \foreach \start in {24}
		{
 		\draw[fill=white] (\start+6,0) ellipse (3cm and 1cm);
		}

	\draw (23, 0.8) node[anchor=north] {$B_{i}'$};
	
	\draw (19.4, 4.5) node[anchor=south] {$A_{i-1}^{cc}$};
	\draw (20.5, 5.0) ellipse (0.4cm and 0.2cm);
	\draw[fill=white] (24.9, 5.0) ellipse (0.9cm and 0.5cm);
	\draw (24.9, 4.5) node[anchor=south] {$A_{i,F}$};

	\draw(24, 5) -- (27, 0);	
	\draw(25.8, 5) -- (31.8, 0.8);

%	\node[w_vertex](1) at (20.7, 5) {};
	\draw (20.7, 5) node[anchor=south] {$v_F$};
	\draw (32.4, -0.5) node[anchor=south] {$B_{i+1}^{sc}$};
%	\draw (1) -- (27,0);
%	\draw (1) -- (31.8, 0.8);
	\draw (31.8, 0.8) -- (31.8, -0.8);
	\draw[dashed] (1) -- (24.45, 0.1);
	\draw[dashed] (1) -- (24.6, 0.1);
	\draw[dashed] (1) -- (24.75, 0.1);
 
		\foreach \x/\y in {23.5/0}
		{
		\draw (\x, \y) -- (\x+1.5, \y) -- (\x+1.5, \y+0.5) -- (\x, \y+0.5) -- cycle ;
		}
	\draw[fill=white] (24.5, 0.2) ellipse (0.3cm and 0.1cm);
	\draw (24.5, 0) node[anchor=north] {$F$};

	%	\draw [dashed] (3, 0) -- (-1, \h);
	%	\draw [dashed] (9, 0) -- (5, \h);

	%	\draw (13, 0) -- (16, \h);
	%	\draw (19, 0) -- (22, \h);
	%%	\draw [dashed] (33, 0) -- (29, \h);
	%%	\draw [dashed] (27, 0) -- (23, \h);
		
	%	\draw [dotted] (10, 0) -- (12, 0);
	%	\draw [dotted] (13, \h) -- (15, \h);
	%	\draw [dotted] (11.5, \h*0.5) -- (13.5, \h*0.5);
	%	\draw [dotted] (34, 0) -- (36, 0);	
	%	\draw [dotted] (35, \h*0.5) -- (36, \h*0.5);
		
	%	\draw (2, \h+1.5) node[anchor=south] {$A_{1}$};
	%	\draw (2+7, \h+1.5) node[anchor=south] {$A_{2}$};
		\draw (3+2+14, \h+1.1) node[anchor=south] {$A_{i-1}$};
		\draw (3+2+21, \h+1.1) node[anchor=south] {$A_{i}$};
	%	\draw (3+2+28, \h+1.5) node[anchor=south] {$A_{i+1}$};

	%	\draw (7, \h+0.7) node[anchor=south] {$A_{21}$};
	%	\draw (7+2, \h+0.7) node[anchor=south] {$A_{22}$};
	%	\draw (7+4, \h+0.7) node[anchor=south] {$A_{23}$};

	%	\draw (24, \h-0.3) node[anchor=south] {$A_{i1}$};
	%	\draw (24+2, \h-0.3) node[anchor=south] {$A_{i2}$};
	%	\draw (24+4, \h-0.3) node[anchor=south] {$A_{i3}$};

	%	\draw (4+2, -1.6) node[anchor=north] {$B_{2}$};
	%	\draw (7+2+7, -1.6) node[anchor=north] {$B_{i-1}$};
		\draw (2+21, -1.2) node[anchor=north] {$B_{i}^+$};
		\draw (2+28, -1.2) node[anchor=north] {$B_{i+1} \backslash B_i^+$};

	}
	\end{tikzpicture}
	
	\caption{Proof of Lemma~\ref{generalcase}}
	\label{condition3b}
\end{figure}
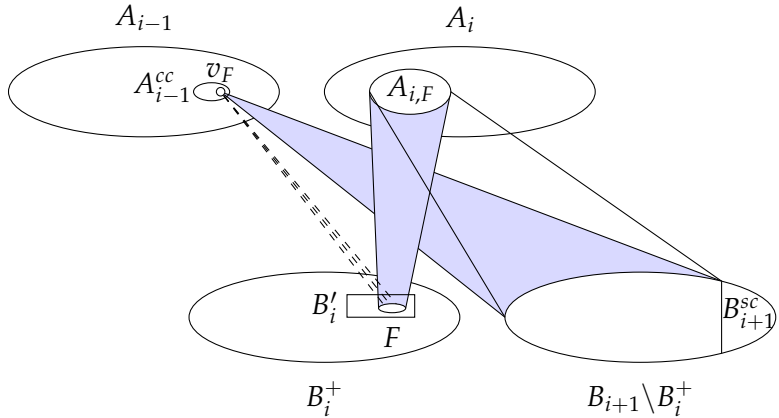

The general case argument also yields a subpartition that satisfies (***) and (****) for the remaining pairs of consecutive bags $G[A_i, B_i]$. 

\begin{lemma}\label{generalcase2}
For any $i$, we can split $B_i$ into $p \leq \pi_2 = \binom{(n-1)r}{r}$ parts $B_i=B_{i1} \cup B_{i2} \cup \ldots \cup B_{ip}$ and find a subset $A_{i}^{ss} \subseteq A_{i}$ of size at most $\pi_3 = (n-1)r(k-1)$ such that $G[A_i \backslash A_i^{ss}, B_{il}]$ is $Free((n'-1) \up_k)$ for all $1 \leq l \leq p$.
\end{lemma}

\begin{proof}
For $i \in I$, the result is trivial as $B_i$ is joined to $A_i$. For $i \notin I$ apply the argument of Lemma~\ref{generalcase} to the bipartite complement of the graph induced by the four consecutive bags $A_{i-1}, B_{i-1}, A_i, B_i$ to yield the required partition. 
\end{proof}

Thus, joining all the partitions of consecutive bags together, we obtain a subpartition that satisfies (***) and (****). Recall that in Lemma~\ref{initialcase}, Lemma~\ref{generalcase} and Lemma~\ref{generalcase2} we used the bounds of partition $\pi_1=\binom{nk}{k}+nk^5$, $\pi_2=\binom{(n-1)r}{r}$ and $\pi_3=(n-1)r(k-1)$. Let $\pi=max(\pi_1, \pi_2, \pi_3)$. We note that as $r=nk$, $n=\max(n', n'', m', m'')$, $\pi=\pi(n', n'', m',m'', k)$ is a function depending on the original input. 

\begin{lemma} \label{conditions3and4}
There is a refinement of the partition $A=A_1 \cup A_2 \cup \ldots \cup A_z$ and $B=B_1 \cup B_2 \cup \ldots B_z$ that splits each bag $A_i$ and $B_i$ into at most $\pi^2$ parts such that the refined partition satisfies (***) and (****). 
\end{lemma}
\begin{proof}
This follows easily taking the refinement of the partitions obtained in Lemma~\ref{initialcase}, Lemma~\ref{generalcase} and
Lemma~\ref{generalcase2}. These Lemmas, do not partition $B_i$ and $A_i$ for $i \in I$ and the first bag to be partitioned is $B_{i+1}$. This bag is partitioned into at most $\pi$ parts twice: to make $G[A_i, B_{i+1}]$ and $G[B_{i+1}, A_{i+1}]$ both satisfy (***) and (****). Taking the refinement (intersection) of these partitions we obtain a partition of $B_{i+1}$ into at most $\pi^2$ parts. The next bag $A_{i+1}$ and all the following bags as long as $i \notin I$ are partitioned into at most $\pi$ parts, but there also is a small set $A_{i+1}^{ss}$ of size less than $\pi$ that is sometimes removed from it while proving Lemma~\ref{generalcase}. To make sure all pairs of bags satisfy conditions (***) and (****), we place each vertex of $A_{i+1}^{ss}$ into separate bag, thus refining the partition of $A_{i+1}$ into $2\pi \leq \pi^2$ parts. Similarly, $B_{i+2}$ which is partitioned into $\pi$ parts together with singletons from $B_{i+1}^{ss}$ provides us with partition of into at most $2\pi \leq \pi^2$ parts. It is now clear that this refined partition satisfies (***) and (****) for all pairs of consecutive bags. 
\end{proof}

Together with conditions (*) and (**) obtained in the previous section, we conclude that we constructed the required $d$-template.

\begin{corollary}~\label{cortemplate}
Let $G=(A,B, \mathcal{E})$ be a bipartite $(n'\up_k, n''\Lambda_k, \overline{m'\up_k}^{bip}, \overline{m''\Lambda_k}^{bip})$-free graph and let $M$ and $N$ be the subsets of $A$ and $B$ respectively obtained in the previous section. Then $G[A \backslash M, B \backslash N]$ admits an $(n', m', m'', k, \pi^2, d)$-template, where $d=(n-1)(kn)+k(kn)$ with $n=max(n',n'', m',m'')$ and $\pi=\pi(n',n'',m',m'',k)$ are both constants depending on $n', n'', m', m''$ and $k$.
\end{corollary}

\begin{proof}
The conditions (***) and (****) proved for a partition of $G=(A, B, \mathcal{E})$ in Lemma~\ref{conditions3and4} clearly remain valid for the same partition restricted to sets $A \backslash M$ and $B \backslash N$. Lemma~\ref{unmarkedvertices} shows that the conditions (*) and (**) are satisfied for the graph induced by $A \backslash M$ and $B \backslash N$. Thus all the conditions (*)-(****) are satisfied, concluding the result. 
\end{proof}

%%%%%%%%%%%%%%%%%%%%%%%%%%%%%%%%%%%%%%%%%%%%%%%%%
%%%%%%% FINAL PROOF OF STARS SECTION %%%%%%%%%%%%
%%%%%%%%%%%%%%%%%%%%%%%%%%%%%%%%%%%%%%%%%%%%%%%%%

\subsection{Proof of Theorem~\ref{thm:stars}}

We can now put the results from the previous sections together to obtain the induction step needed for proving Theorem~\ref{thm:stars}. 

\begin{lemma}~\label{inductionstep}
Suppose a bipartite graph $G=(A,B, \mathcal{E})$ is $(n'\Lambda_k,
n''\up_k, \overline{m'\Lambda_k}^{bip}, \overline{m''\up_k}^{bip})$-free for some integers $n',n'',m',m'' \geq 3$. Then $A$ and $B$ can be partitioned into at 
most $4 \pi^2 + 3nkd^2$ parts each such that between 
the parts the graph is either in $Free((n'-1)\up_{k+2d})$, $Free(\overline{(m'-1)\up_{k+2d}}^{bip})$, or in
$Free (\overline{(m''-1)\Lambda_{k+2d}}^{bip})$.
\end{lemma}

\begin{proof}
For the given graph $G=(A, B, \mathcal{E})$, apply the partition procedure described in Section~\ref{partitionprocedure}. As described in Lemma~\ref{markedvertices}, the vertices $M \subseteq A$ and $N \subseteq B$ can be partitioned into $3nkd^2$ sets $M_i$ and $N_i$ such that $G[M_i,B]$ and $G[A, N_i]$ are $2\Lambda_{2d}$-free and $2\up_{2d}$-free, respectively. In particular, as $m'', n' \geq 3$ one can see that these graphs are $(\overline{(m''-1)\Lambda_{k+2d}}^{bip})$-free and $((n'-1)\up_{k+2d})$-free, respectively.

The remaining graph $G[A \backslash M, B \backslash N]$ is a $(n', m', m'', k, \pi^2, d)$-template by Corollary~\ref{cortemplate}. Thus by Lemma~\ref{templatecollapse} the vertices of $G[A \backslash M, B \backslash N]$ can be partitioned into at most $4\pi^2$ parts such that between the parts the graph is either in $Free((n'-1)\up_{k+2d})$, $Free(\overline{(m'-1)\up_{k+2d}}^{bip})$, or in
$Free (\overline{(m''-1)\Lambda_{k+2d}}^{bip})$.

Putting together the partitions of $M$ and $A \backslash M$ and partitions of $N$ and $B \backslash N$, we obtain the required partition of $G$.  
\end{proof}

Applying the induction step obtained above as long as we can we obtain the following theorem.

\begin{theorem} \label{thm:starseither}
Let $G=(A, B, \mathcal{E})$ be a bipartite $(n'\Lambda_k, n''\up_k, \overline{m'\Lambda_k}^{bip}, \overline{m''\up_k}^{bip})$-free graph and let $\mu=n' + n'' + m' + m'')$. 
Then there are two constants $U'=U'(\mu ,k)$ and $s=s(\mu ,k)$ and a partition
$A=A_1 \cup A_2 \cup \ldots \cup A_u$, 
$B=B_1 \cup B_2 \cup \ldots \cup B_u$, with $u \leq U'$ such that
for any $1 \leq i, j \leq u$ we have $G[A_i, B_j]$ is either $2 \Lambda_s$-free or $2 \up_s$-free.
\end{theorem}

\begin{proof}
We will prove the result by induction on $\mu=n'+n''+m'+m''$. The result is trivial if $n' \leq 2$, $n'' \leq 2$, $m' \leq 2$ or $m'' \leq 2$ as in that case we can take $U'(\mu, k)=1$ and $s(\mu, k)=k$. So suppose $n', m', n'', m'' \geq 3$ and that the result holds for all other quadruples with the sum less than $ \mu =n'+n''+m'+m''$. Applying Lemma~\ref{inductionstep} to the graph $G$, we obtain a partition $A=A_1 \cup A_2 \cup \ldots \cup A_c$, $B=B_1 \cup B_2 \cup \ldots \cup B_c$ with $c=\pi^2+3nkd^2$ parts such for all $1 \leq i,j \leq c$ the graph $G[A_i, B_j]$ is either in $Free((n'-1)\up_{k+2d})$, $Free(\overline{(m'-1)\up_{k+2d}}^{bip})$, or in
$Free (\overline{(m''-1)\Lambda_{k+2d}}^{bip})$. Thus induction hypothesis applies to each graph $G[A_i, B_j]$ showing that $A_i$ and $B_j$ can be subpartitioned into $U'(\mu-1, k+2d)$ parts such that between the parts the graph is either $2\Lambda_{s}$-free or $2 \up_{s}$-free for $s=s(\mu-1, k+2d)$ given by induction. Taking the refinement (intersection) of these partitions obtained for all pairs $A_i$ and $B_j$, we obtain a partition with $U'(\mu, k)=(U'(\mu-1, k+2d))^c$ parts such that between any two parts the graph is either $2\Lambda_{s}$-free or $2 \up_{s}$-free with $s=s(\mu,k)=s(\mu-1, k+2d)$.   
\end{proof}

Recall that our aim is to obtain a partition such that the graph between parts is both $2\Lambda_{2k-1}$-free and $2 \up_{2k-1}$ - free. For this, we will run a modified argument of partitioning into $d$-template two more times. This time the argument provided in the Lemma~\ref{orand} below is simpler as we have additional restrictions on the graph. 

\begin{lemma} \label{orand}
Let $G=(A, B, \mathcal{E})$ be a bipartite graph that is $(2\up_s, n\Lambda_k, \overline{n\Lambda_k}^{bip})$-free. Then we can partition $A=M_1 \cup M_2 \cup \ldots \cup M_u$ into $u \leq \phi=3nk(nks^2)^2+1$ parts such that for all $i$ the graphs $G[M_i, B]$ are $(2\up_s, 2 \Lambda_{2k-1})$-free.  
\end{lemma}

\begin{proof}
We proceed with the following partiton of $G=(A, B, \mathcal{E})$.
We let $B_1=\{b_1\}$ be a vertex with minimal degree in $B$ and let $A_1^+=N(b_1)$. Now, let $A'_1 \subseteq A_1^{+}$ consist of those vertices in  $A_1$ that have at least $k$ non-neighbours in $B \backslash B_1$. 

First of all, we will show that each vertex of $B \backslash B_1$ has at most $s$ non-neighbours in $A_1^+$. Indeed, take a vertex $b \in B \backslash B_1$. By our choice that $b_1$ is of minimal degree in $B$ it follows that $|N(b) \backslash N(b_1)| \geq |N(b_1) \backslash N(b)|$. Now $\{b_1, b_2, N(b) \backslash N(b_1), N(b_1) \backslash N(b)\}$ induces two stars and as the graph is $2\up_s$-free, we must have that $|N(b_1) \backslash N(b)|<s$. In other words, vertex $b$ has less than $s$ non-neighbours in $N(b_1)=A_1^+$.  

Now we claim that $|A_1'| \leq nks^2$. As each vertex in $B \backslash B_1$ has less than $s$ non-neighbours in $A_1^+$ and the graph is $\overline{n\Lambda_k}^{bip}$-free we can apply Lemma~\ref{boundeddegree} to the complement of the graph between $B \backslash B_1$ and $A_1^+$. We obtain that there are at most $nks^2$ vertices in $A_1^+$ of co-degree at most $k$ in $B \backslash B_1$, thus proving the claim.    

We proceed forming the partition of $G$ as follows. Let $A_1=A_1^+ \backslash A_1'$. We create all the other bags similarly. Suppose $B_1, A_1, \ldots, B_{i}, A_i$ have been created for some $1 \leq i < |B|$. Then we let $B_{i+1}^+=\{b_{i+1}\}$, where $b_{i+1}$ is a vertex from $B \backslash (B_1 \cup \ldots \cup B_i)$ which has minimal degree in $A \backslash (A_1 \cup A_2 \cup \ldots \cup A_i)$. We let $A_{i+1}^+=N(b_{i+1}) \backslash (B_1 \cup \ldots \cup B_i)$. Now let $A_{i+1}'$ be the subset of $A_{i+1}^+$ containing vertices of co-degree at least $k$ in $B \backslash (B_1 \cup \ldots \cup B_{i+1})$. By the same reasoning as provided for $A_1'$ it follows that $|A_{i+1}'|<nks^2$ and we set $A_{i+1}=A_{i+1}^+ \backslash A_{i+1}'$. In this way we construct the bags $B_i$ for $i=1,2, \ldots |B|$ containing one vertex of $B$. For $i=|B|+1$ we set $B_i = \emptyset$ and $A_i= A \backslash (A_1 \cup \ldots \cup A_{i-1})$.      

Now we can observe some properties of this partition:
\begin{itemize}
\item[(P1)] The vertices in $A_i$ have co-degree at most $k$ in $B \backslash (B_1 \cup \ldots \cup B_i)$. 
\item[(P2)] The vertices in $B_i$ have degree at most $\delta=nks^2$ in $A \backslash A_1 \cup \ldots \cup A_i$. 
\end{itemize}

Let $M$ be the set of all the vertices $v \in A$ such that $v \in A_i$ for some $i$ and $v$ has more than $k$ neighbours in $\cup_{h=1}^{i-1}B_h$. It is immediate that the graph $G[A \backslash M, B]$ is $2\Lambda_{2k-1}$-free. Indeed, take $v \in A_i$ and $w \in A_j$ with $i \leq j$. Then $N(v)$ has at most $k-1$ non-neighbours in $B_1
\cup \ldots \cup B_j$ and $N(w)$ has at most $k-1$ neighbours in $B \backslash B_1 \cup \ldots B_j$. Hence $|N(v) \backslash N(w)|\leq=2k-2$ and hence no two vertices in $A \backslash M$ can be centres of the two stars $2 \Lambda_{2k-1}$. Thus $A \backslash M$ is $2 \Lambda_{2k-1}$-free. 
 
In the rest of the proof we will show how to partition $M$ into $c=3nk\delta^2$ bags $M=M_1 \cup M_2 \cup \ldots \cup M_c$ such that the graph $G[M_i, B]$ is $2\Lambda_{2k-1}$-free for all $i$. The proof closely follows the proof of Lemma \ref{markedvertices}.

Let us start with a vertex $v \in M$. We will define a trace of a vertex $v$. First recall that $v \in M$ means that $v \in A_i$ for some $i$ such that $v$ has more than $k$ neighbours in $\cup_{h=1}^{i-1} B_h$. Now we find the smallest value $j$, such that $v$ has more than $k$ neighbours in $\cup_{h=1}^{j-1} B_h$. We define the trace of $v$ to be the interval $tr(v)=\{j, j+1, \ldots, i\}$. In the next two paragraphs we will explore several properties of the trace.  

Consider the set $T_l=\{v \in M: l \in tr(v)\}$. We claim that $|T_l| \leq nk\delta^2$. Consider the bipartite subgraph of $G$ induced by $\cup_{h=1}^{l-1} B_h$ and $T_l$. Note that $T_l \subseteq A \backslash \cup_{h=1}^{l-1} A_h$. By $(P2)$ we have that each vertex of $\cup_{h=1}^{l-1} B_h$ has at most $\delta$ neighbours in $T_l \subseteq A \backslash \cup_{h=1}^{l-1} A_h$. As the graph $G$ is $n \Lambda_k$ free, Lemma~\ref{boundeddegree} applies. From Lemma~\ref{boundeddegree} if follows that there are at most $nk\delta^2$ vertices in $T_l$ that have degree at least $k$ in $\cup_{h=1}^{p-1} B_h$. But as $T_l$ is defined to be a subset of $M$, all vertices of $T_l$ has degree at least $k$ in $\cup_{h=1}^{l-1} B_h$. Hence, we conclude that $|T_l| \leq nk\delta^2$.

Now, consider two vertices $v, v' \in M$ such that $tr(v)=\{j, j+1, \ldots, i \}$ and $tr(v')=\{j', j'+1, \ldots, i'\}$ with $j'>i$ such that $|j'-i| \geq 3$. Then, we claim that $|N(v') \backslash N(v)| \leq 2k-2$. Indeed, $N(v)$ has at most $k-1$ non-neighbours in $B \backslash \cup_{h=1}^{i+1}$ by $(P1)$. Also, as $j'=min\{tr(v')\}$, it follows that $N(v')$ has at most $k-1$ neighbours in $B_1 \cup B_2 \cup \ldots B_{j'-2}$. In particular, as $j' \geq i+3$, $N(v')$ has at most $k-1$ neighbours in $B_1 \cup B_2 \cup \ldots \cup B_{i+1}$. As $N(v)$ has at most $k-1$ non-neighbours in $B \backslash \cup_{h=1}^{i+1}$ and $N(v')$ has at most $k-1$ neighbours in $\cup_{h=1}^{i+1} B_{h}$, we conclude that $|N(v) \backslash N(v')| \leq 2k-2$.  

We can partition $M$ into $c=nk\delta^2$ subsets $M=M_1' \cup M_2' \cup \ldots \cup M_c'$ such that for any $i$ and any $v, v' \in M'_i$, we have $tr(v) \cup tr(v') = \emptyset$. This partition procedure that partitions vertices into $c$ sets such that traces of the vertices in each set are all disjoint is described in Lemma \ref{markedvertices}. This Lemma also shows that we can subpartition each set $M_i'$ into three $M_{3i-2}, M_{3i-1}$ and $M_{3i}$ to ensure that there are at least two natural numbers in the interval between two traces. Thus we obtain the sets $M_1, M_2, \ldots, M_{3c}$ such that $G[M_i, B]$ is $2\Lambda_k$-free. We let $M_{3c+1}=A \backslash M$ and note that this is a partition of $A=M_1 \cup \ldots \cup M_{3c+1}$ into $3c+1=3nk\delta^2+1=3nk(nks^2)^2+1$ parts such that $G[M_i, B]$ induces $2\Lambda_{2k-1}$-free graph for all $i$. Hence we are done.
\end{proof}

Now we are ready to deduce Theorem~\ref{thm:stars}.

\begin{proof}[Proof of Theorem~\ref{thm:stars}]
The proof now easily follows from Theorem~\ref{thm:starseither} and Lemma~\ref{orand}. We start with the partition provided by Theorem~\ref{thm:starseither} into $U'(\mu, k)$ parts such that the graph between the parts is either $2\up_s$-free or $2\Lambda_s$-free for $s=s(\mu, k)$. Applying Lemma~\ref{orand} to each pair of bags that are $2\up_s$-free (resp. $2\Lambda_s$-free) we obtain a subpartition that is $(2\up_s, 2 \Lambda_{2k-1})$-free (resp. $(2\Lambda_s, 2 \up_{2k-1})$-free). This provides us with refinement into $U''=U'\phi^{U'}$ bags such that the graph is either $(2\up_s, 2 \Lambda_{2k-1})$-free or $(2\Lambda_s, 2 \up_{2k-1})$-free. A second application of Lemma~\ref{orand} to each pair of refined bags provides with a refined partition into $U=U''\phi^{U''}$ bags that are $(2\up_{2k-1}, 2 \Lambda_{2k-1})$-free. This completes the proof. 
\end{proof}

%%%%%%%%%%%%%%%%%%%%%%%%%%%%%%%%%%%%%%%%%%%%%%%%%%
%%%%%%%%%%%%%%%%%%%%%%%%%%%%%%%%%%%%%%%%%%%%%%%%%%
%%%%%%%%%%%%%%%%%%%%%%%%%%%%%%%%%%%%%%%%%%%%%%%%%%

\section{General graphs without a union of stars and complements of a union of stars} \label{section:starsg}

To extend our result to general graphs without union of stars we will need the following result that bounds cochromatic number of graphs:
\begin{lemma}\label{cochromatic}
For every $n,k,l \in \mathbb{N}$ there exists a number $z(n,k,l)$ such that any graph $G \in Free(nK_{1,k}, $ $\overline{nK_{1,k}}, nK_l, \overline{nK_l})$ has cochromatic number at most $z(n,k,l)$. 
\end{lemma} 

The proof of Lemma~\ref{cochromatic} follows from the works of Kierstead and Penrise \cite{kp94} and Chudnovsky and Seymour in \cite{cs14}. For completeness, we will provide a short outline of the proof in this section. First of all, we will start with several definitions.  

For a graph G, $\chi(G)$ denotes the \emph{chromatic number} of a graph - the minimal number of independent sets that the vertices of the graph $G$ can be partitioned into. $\omega(G)$ denotes the \emph{clique number} - the number denoting the maximum size of a clique contained in the graph $G$. A graph classs $C$ is said to be $\chi$-bounded, with $\chi$-binding function $f$, if for all $G \in C$, $\chi(G) \leq f(\omega(G))$. A connected graph $G$ is said to be \emph{of radius at most $r$} if there is a vertex $v \in V(G)$ such for every $w \in V(G)$ there is a path in $G$ connecting $v$ and $w$ with at most $r$ edges. We say that a connected graph $G$ is \emph{of radius $r$} if $r$ is the smallest integer for which $G$ is of radius at most $r$.  

Kierstead and Penrise in \cite{kp94} have shown the following result:
\begin{theorem}
For any tree $T$ of radius two, the class $C=Free(T)$ is $\chi$-bounded. 
\end{theorem}

It is interesting to note that a conjecture of Gy\'arf\'as asserts that the above statement should be true for any tree $T$ (not only of radius two), but only partial results have been proven so far. For our purposes, we need the following corollary from the result of Kierstead and Penrise:

\begin{theorem}\label{kaibounded}
For every $n, k \in \mathbb{N}$ the class $C=Free(nK_{1,k})$ is $\chi$-bounded.
\end{theorem} 

\begin{proof}
Let $T_{n,k}$ be a tree of radius two formed by adding an extra vertex $v$ to $nK_{1,k}$ and extra edges from $v$ to the centre of each star $K_{1,k}$. Then by Theorem~\ref{kaibounded}, the class $Free(T_{n,k})$ is $\chi$-bounded. That is, there is some $\chi$-binding function $f_{n,k}$, such that for all $G \in Free(T_{n,k})$, $\chi(G) \leq f_{n,k}(\omega(G))$. But then, as $Free(nK_{1,k})$ is a subclass of $Free(T_{n,k})$, we must have $\chi(G) \leq f_{n,k}(\omega(G))$ for every graph $G \in Free(nK_{1,k})$. Thus we conclude that the class $Free(nK_{1,k})$ is $\chi$-bounded.    
\end{proof}

Chudnovsky and Seymour in \cite{cs14} have established a connection between bounded cochromatic number and $\chi$-boundedness. The connection is through the notion of \emph{$p$-split graphs} which is defined as a graph for which the vertices can be partitioned into two parts, one of which does not contain a clique on $p$ vertices and the other does not contain an independent set on $p$ vertices. They prove that a class forbidding a union of cliques and a complement of union of cliques consists of $p$-split graphs for some fixed $p$. We state this result as follows:

\begin{theorem}\label{psplit}
For any $n, l \in \mathbb{N}$, there exist $p=p(n,l)$ such that every $(nK_l, \overline{nK_l})$-free graph is $p$-split. 
\end{theorem}

This now brings us to the proof of Lemma~\ref{cochromatic}:

\begin{proof}
Consider a graph $G$ in $Free(nK_{1,k}, \overline{nK_{1,k}}, nK_l, \overline{nK_l})$. Then as $G$ is $Free(nK_l, \overline{nK_l})$, by Theorem~\ref{psplit} we have that vertices of $G$ can be partitioned into two parts A and B, such that $G[A]$ is $K_p$-free and $G[B]$ is $\overline{K_p}$-free for some $p=p(n,l)$. As $G[A]$ is in $Free(nK_{1,k})$, by Theorem~\ref{kaibounded} it is $\chi$-bounded with some $\chi$-binding function $f_{n,k}$. As $\omega(G[A])<p$, this implies that $\chi(G) \leq f_{n,k}(p)$, thus this graph can be partitioned into $f_{n,k}(p)$ independent sets. Applying the argument to the complement of $G[B]$, we can deduce that $G[B]$ can also be partitioned into constant number $f_{n,k}(p)$ of cliques. Thus, noting that $p=p(n,l)$ every graph in the class has co-chromatic number at most $z(n,k,l) = 2f_{n,k}(p)$ which is a constant depending on the sizes of the input graphs.  
\end{proof}

Thus we are now capable of concluding with our main result of this paper, 
Theorem~\ref{main}:

\begin{proof}
By Lemma~\ref{cochromatic} it follows that for any $n, k, l \in \mathbb{N}$ given, there is $z(n,k,l)$ such that the vertices of any graph $G \in Free(nK_{1,k}, \overline{nK_{1,k}}, nK_k, \overline{nK_k})$ can be partitioned into $z(n,k,l)$ sets, each of which induce either a clique or an independent set. Now, we can apply Theorem~\ref{thm:stars} to each pair of the $z(n,k,l)$ sets. For each pair of sets, Theorem~\ref{thm:stars} gives us a partition into $U=U(4n,k)$ parts with $(\Lambda_{2k-1}, \up_{2k-1})$-free graphs between the parts. Taking the refinement (intersection) of all these partitions we obtain a partition into $T=z U^{z}$ parts such that between the parts the graph is $(\Lambda_{2k-1},\up_{2k-1})$-free. 
\end{proof}

\section{Classes of graphs and factorial speed of growth} \label{section:factorial}

In this section we will prove Theorem~\ref{classes}. As this theorem contains three statements, we will prove these separately in Lemma~\ref{classes1}, Lemma~\ref{classes2} and Lemma~\ref{classes3}. We start with the characterisations of the graph classes whose graphs admit partitions of Theorem~\ref{main} and Theorem~\ref{thm:matchings}. 

\begin{lemma}\label{classes1}
For a hereditary class $\X$ there exists a constant $T=T(\X)$ such that each graph in $\X$ is a $(T,1)$-graph if and only if $\X$ does not contain classes $\Y_1, \Y_2, \ldots, \Y_6$. 
\end{lemma}

\begin{proof}
If the class $\X$ does not contain any of the classes $\Y_1, \Y_2, \ldots, \Y_6$, then $\X$ must exclude at least one graph from each class. Hence it follows that there exists an integer $n$ such that $\X \subseteq Free(\F_{n,1})$  and hence by Theorem~\ref{thm:matchings} there exists a number $T=T(n)$ such that every graph of $\X$ is a $(T,1)$-graph.

Suppose now, the class $\X$ contains $\Y_1$ and suppose, for contradiction, that there exists $T=T(\X)$ such that each graph in $\X$ is a $(T,1)$-graph. As $\X$ contains $\Y_1$, it follows that $\X$ contains arbitrarily large matchings, in particular $G=(T^2+1)K_2 \in \X$. By assumption, $G$ is a $(T,1)$-graph, so admits a partition into $T$ sets $V(G)=V_1 \cup V_2 \cup \ldots \cup V_T$ such that each set induces an independent set or a clique, and between each pair of sets, the induced bipartite subgraph is $2K_2-free$. For each pair of vertices $(v, w) \in V(G) \times V(G)$ such that $vw \in E(G)$ assign a pair of integers $(i,j)$ such that $v \in V_i$ and $w \in V_j$. As there are $T^2+1$ distinct edges in $G$, and $T^2$ possible pairs of integers, by pigeonhole principle two distinct edges will be assigned the same pair $(i,j)$. If $i=j$ this means that the bag $V_i$ contains a $2K_2$ which contradicts an assumption that each bag is either a clique or an independent set, while if $i \neq j$, we have a $2K_2$ in the bipartite graph induced between bags $V_i$ and $V_j$ contradicting the assumption that between the bags the induced subgraph is $2K_2$-free. A similar argument can be applied if $\X$ contains any of the classes $\Y_2, \Y_3, \ldots, \Y_6$.  
\end{proof}

\begin{lemma}\label{classes2}
For a hereditary class $\X$ there exist two constants $T=T(\X)$ and $k=k(\X)$ such that each graph in $\X$ is a $(T,k)$-graph if and only if $\X$ does not contain classes $\X_1, \X_2, \ldots, \X_{10}$. 
\end{lemma}

\begin{proof}
If the class $\X$ does not contain any of the classes $\X_1, \X_2, \ldots, \X_{10}$, then $\X$ must exclude at least one graph from each class. Hence it follows that there exist two integers $n,k$ such that $\X \subseteq Free(\F_{n,k})$ and hence by Theorem~\ref{main} there exist a number $T=T(n,k)$ such that every graph of $\X$ is a $(T,k)$-graph.

Suppose now, the class $\X$ contains $\X_1$ and suppose, for contradiction, that there exist $T=T(\X)$ and $k=k(\X)$ such that each graph in $\X$ is a $(T,k)$-graph. As $\X$ contains $\X_1$ if follows that $G=(T^2+1)K_{1,2kT} \in \X$. By assumption, $G$ is a $(T,k)$-graph, thus it can be partitioned into $T$ sets $V(G)=V_1 \cup V_2 \cup \ldots \cup V_T$ such that each set induces a clique or an independent set and between each pair of sets the induced bipartite graph is $(2\up_{2k-1}, 2\Lambda_{2k-1})$-free. For each star $K_{1,2kT}$ in $G$ we assign a pair of integers $(i,j)$ such that the centre of the star belongs to the bag $V_i$ and $V_j$ is a bag that contains at least $2k$ leaves. Note that such $j$ exists since by pigeonhole principle at least one bag contains $\frac{2kT}{T}=2k$ leaves and if there are several possible $j$'s we choose one arbitrarily. As there are $T^2+1$ different stars, two stars will get assigned to the same pair $(i,j)$. If $i=j$, then this implies that $V_i$ is neither a clique nor indendent set and if $i \neq j$, then this implies that the bipartite graph induced between $V_i$ and $V_j$ is not $(2\up_{2k-1},2\Lambda_{2k-1})$-free. A similar argument can be applied if $\X$ contains any of the classes $\X_1, \X_2, \ldots, \X_{10}$.         
\end{proof}

In the rest of this section we will prove that classes that exclude star-forests and related graphs or, equivalently, classes not containing $\X_1, \X_2, \ldots, \X_{10}$, have at most factorial speed of growth. Recall that the speed of a $\X$ class is the sequence $\X_n$, where $\X_n$ is the number of $n$-vertex graphs in $\X$ with vertex set $\{1,2,\ldots, n\}$. For instance, the speed of the class of all graphs is $2^{\binom{n}{2}}$ and the speed of the class of trees is $n^{n-2}$ as each (labelled) tree on $n$ vertices can be uniquely described by Pr\"{u}fer code consisting of a string of $n-2$ numbers ranging from $1$ to $n$. The factorial layer of growth identifies the classes that have speed $\log(X_n)=O(n \log(n))$. Thus factorial layer includes the class of trees and, in fact, many other important classes of graphs such as planar graphs, classes of bounded degree, classes of bounded clique-width, etc. While containing many classes of theoretical and practical importance, there is no easy decision procedure which tells for which sets of graphs $\F$, the class $\X=Free(\F)$ is of factorial speed of growth. We note that factorial layer is the smallest such layer, as the membership to constant, polynomial and exponential layers can be checked effectively as follows from the work of Alekseev~\cite{general1}. Our work provides a partial result to this question: if $\F$ contains a graph from each class $\X_1, \X_2, \ldots, \X_{10}$, then it is of at most factorial speed of growth. 

To prove that our classes have at most factorial speed of growth we will use the idea of locally bounded coverings introduced in \cite{lmz12}. The idea can be described as follows. 

Let $G$ be a graph. A set of graphs $H_1, \ldots, H_k$ is called a covering of $G$ if the union of $H_1, \ldots, H_k$ coincides with $G$, i.e. if $V(G)=\cup_{i=1}^k V(H_i)$ and $E(G)=\cup_{i=1}^{k} E(H_i)$. The following result was proved in \cite{lmz12}:

\begin{theorem} \label{factorialcovering}
Let $X$ be a class of graphs and $c$ a constant. If every graph $G \in X$ can be covered by graphs from a class $Y$ with $\log Y_n = O(n \log (n))$ in such a way that every vertex of $G$ is covered by at most $c$ graphs, then $\log X_n = O(n \log (n))$. 
\end{theorem}

As each $(T, k)$-graph admits a covering by cliques, independent sets and bipartite $(2\up_{2k-1}, 2\Lambda_{2k-1})$-free graphs, it is enough to show that  bipartite $(2\up_{2k-1}, 2\Lambda_{2k-1})$-free graphs, or more generally  bipartite $2\Lambda_{s}$-free have at most factorial speed of growth. 

\begin{lemma} \label{factorial}
Let $s$ be a natural number and $\X$ be a class of bipartite $2\Lambda_s$-free graphs. Then $\X$ is at most factorial.
\end{lemma}

\begin{proof}
Let $G=(A, B, \mathcal{E}) \in \X$ be a bipartite graph on $n$ vertices not containing $\Lambda_s$. 
Let $A=\{a_1, a_2, \ldots, a_q\}$, such that 
$|N(a_1)| \leq |N(a_2)| \leq \ldots \leq |N(a_q)|$.
Then we can encode the graph $G$ as a sequence
$$[a_1, N(a_1), a_2, N(a_2)\triangle N(a_1), \ldots, a_q, N(a_q) \triangle N(a_{q-1})]$$ 
(Here $\triangle$ denotes the symmetric difference between the sets, i.e. for two sets $V_1, V_2$, we write $V_1 \triangle V_2 = (V_1 \backslash V_2) \cup (V_2 \backslash V_1)$.)
 
First of all, let us see that this sequence describes the graph. 
Indeed, we can prove by induction that one can obtain $N(a_i)$ for any $i$.
For $i=1$, this is already given. Suppose $i>1$ and we have calculated the neighbourhood $N(a_{i-1})$. Then we can read-off the set 
$N(a_i) \triangle N(a_{i-1})$ from the sequence, and we know $N(a_{i-1})$ by induction, so we can calculate $N(a_i)=N(a_i) \triangle (N(a_i) \triangle N(a_{i-1}))$.
Hence, we can recover the graph from the sequence.

We will now show that this sequence contains at most $(2s+2)n$ vertices. To do this we will estimate the size symmetric difference
$N(a_i) \triangle N(a_{i-1})=N(a_i) \backslash N(a_{i-1}) \cup N(a_{i-1}) \backslash N(a_{i})$. Notice that as $|N(a_{i})| \geq |N(a_{i-1})|$ it follows that $|N(a_i) \backslash N(a_{i-1})| \geq |N(a_{i-1}) \backslash N(a_{i})|$. Our first aim is to provide a bound to the smaller part $|N(a_{i-1}) \backslash N(a_{i})|$.  We claim that $|N(a_{i-1})\backslash N(a_{i})|<s$. Suppose, for contradiction that $|N(a_{i-1})\backslash N(a_{i})| \geq s$. Then, $|N(a_{i}) \backslash N(a_{i-1})| \geq s$ as well. So we can pick two subsets $B_i \subseteq N(a_i) \backslash N(a_{i-1})$ and 
$B_{i-1} \subseteq N(a_{i-1}) \backslash N(a_i)$ of size $|B_i|=|B_{i-1}|=s$. But then $G[\{a_i\} \cup B_i \cup \{a_{i-1}\} \cup B_{i-1}]$ induce $2\Lambda_s$, a contradiction. Hence, we obtain that $|N(a_{i-1}) \backslash N(a_{i})| < s$. Now we estimate the size of larger part $N(a_i) \backslash N(a_{i-1})$ as follows. 
\begin{align*}
|N(a_i) \backslash N(a_{i-1})| &= |N(a_i)|-|N(a_i) \cap N(a_{i-1})| \\
&=|N(a_i)|-|N(a_{i-1})|+|N(a_{i-1})|- |N(a_i) \cap N(a_{i-1})| \\
&=|N(a_i)|-|N(a_{i-1})|+|N(a_{i-1}) \backslash N(a_{i})| \\
& < |N(a_i)|-|N(a_{i-1})| + s.
\end{align*} 
Adding the two estimates together we obtain $|N(a_i) \triangle N(a_{i-1})| \leq |N(a_i)| - |N(a_{i-1})| + 2s-2$. Now we can provide a bound on the number of vertices in the sequence as follows:
\begin{align*}
|N(a_1)|+\sum_{i=2}^q |N(a_i) \triangle N(a_{i-1})| + q
&\leq |N(a_1)| + \sum_{i=2}^q (|N(a_i)|-|N(a_{i-1})| + 2s-2) + q\\ 
&= |N(a_q)| + (2s-2)(q-1) + q\\
&\leq n + (2s-2)n + n \\
&= 2sn.
\end{align*}

Thus we can encode each graph in the class by a sequence consisting of $2sn$ labels of vertices (numbers from $1$ to $n$) each followed by a comma or space symbol (with commas as presented above). As there are at most $(2n)^{2sn}$ such codes, we have at most $(2n)^{2sn}$ graphs on $n$ vertices in the class. Hence the class is factorial. 
\end{proof}

Adding the results of Lemma~\ref{factorial} and Theorem~\ref{thm:stars} together we conclude this section with the following lemma.

\begin{lemma}\label{classes3}
Let $\X$ be a class of graphs for which there exist $T=T(\X)$ and $k=k(\X)$ such that every graph in $\X$ is a $(T,k)$-graph. Then $\X$ has at most factorial speed of growth. 
\end{lemma}

\begin{proof}
Let $\Y'$ be a class of bipartite $(\up_{2k-1},\Lambda_{2k-1})$-free graphs and let $\Y''$ be a class of cliques and independent sets. Then by Lemma~\ref{factorial} it follows that the class $\Y'$ has at most factorial speed of growth. As $\Y''$ has only two graphs on $n$ vertices, it is clear that $\Y=\Y' \cup \Y''$ has at most factorial speed of growth. Now, take any $G \in \X$. As $G \in X$ is a $(T,k)$-graph there is a partition $V(G)=V_1 \cup V_2 \cup \ldots \cup V_T$ with vertex sets inducing cliques or independent sets and between the sets the induced bipartite graphs are $(\up_{2k-1},\Lambda_{2k-1})$-free. Thus, we produce a cover of the graph $G$ by graphs $G[V_i] \in \Y$ for all $i$ and bipartite graphs $G[V_i, V_j, \mathcal{E}(G) \cap V_i \times V_j] \in \Y$. Notice that each vertex is covered by $T$ graphs. Thus Theorem~\ref{factorialcovering} applies and we get that $\X$ has at most factorial speed of growth. 
\end{proof}

\section{Concluding remarks and open questions} \label{section:conclusion}

The results presented here provide us with the structural characterisation of new classes from factorial layer. One may also notice that Theorem~\ref{classes} implies that every class of superfactorial growth must contain one of the classes $\X_1, \X_2, \ldots, \X_{10}$. A further extension of our understanding of factorial (and in turn superfactorial) layer, could be made by looking at classes $\X$ for which there is some $T=T(\X)$ such that all prime graphs in $\X$ are $(T,1)$-graphs. By a result on counting prime graphs in \cite{implicit}  it follows that all such classes have at most factorial speed of growth. Hence the aim would be to identify the minimal classes of graphs for which the prime graphs do not admit such partition. These classes would be interesting as all would contain arbitrarily large prime graphs (unlike $\X_1, \X_2, \ldots, \X_{10}$) and every superfactorial class would contain at least one of them. We note that such a characterisation of permutation classes whose simple permutations are monotone griddable (i.e. belongs to the permutation analogue of $(T,1)$-graphs) has been recently obtained in \cite{primepermutations}.           

The results presented here also reveal a potential for precise structural results for graph classes forbidding small subgraphs. It would be interesting to reveal the exact bounds on the number of parts needed in our partitions of bipartite and general graphs and hence precise structural results for a number of classes whose list of forbidden induced subgraphs include a star forest and related graphs.  

\section*{Acknowledgements}
The author would like to thank Robert Brignall for his encouragement and helpful comments on earlier drafts of this paper.

%One of the open questions concern well-quasi-ordering. 
%The procedure of deciding well-quasi-ordering was developed for 
%the griddable classes of permutations. 
%Can someone extend these results to the classes not containing matchings and their complements?

%Similarly, the procedure of deciding well-quasi-ordering was developed for classes with finite distinguishing number.
%Can someone extend these results for classes without stars?

\end{document}